\documentclass[reqno]{amsart}
\usepackage{amsmath}
\usepackage{amsthm}
\usepackage{graphicx}
\usepackage{color}
\usepackage{amsfonts}
\usepackage{amssymb}
\usepackage{enumerate}
\numberwithin{equation}{section}

\newtheorem{theorem}{Theorem}[section]
\newtheorem{lemma}[theorem]{Lemma}
\newtheorem{proposition}[theorem]{Proposition}

\theoremstyle{definition}

\newtheorem{definition}[theorem]{Definition}
\newtheorem{remark}[theorem]{Remark}

\newtheorem{innercustomthm}{Assumption}
\newenvironment{assumption}[1]
  {\renewcommand\theinnercustomthm{#1}\innercustomthm}
  {\endinnercustomthm}

\def\E{{\mathbb E}}
\def\R{{\mathbb R}}
\def\N{{\mathbb N}}
\def\P{{\mathcal P}}
\def\RC{{\mathcal R}}
\def\B{{\mathcal B}}
\def\V{{\mathcal V}}
\def\M{{\mathcal M}}
\def\H{{\mathcal H}}
\def\X{{\mathcal X}}

\def\Q{{\mathcal Q}}

\def\G{{\mathcal G}}

\def\W{{\mathcal W}}

\def\A{{\mathcal A}}
\def\F{{\mathcal F}}

\def\C{{\mathcal C}}

\title{Mean field games with common noise}
\author{Ren\'e Carmona, Fran\c{c}ois Delarue, and Daniel Lacker}

\begin{document}

\begin{abstract}
A theory of existence and uniqueness is developed for general stochastic differential mean field games with common noise. The concepts of strong and weak solutions are introduced in analogy with the theory of stochastic differential equations, and existence of weak solutions for mean field games is shown to hold under very general assumptions. Examples and counter-examples are provided to enlighten the underpinnings of the existence theory. Finally, an analog of the famous result of Yamada and Watanabe is derived, and it is used to prove existence and uniqueness of a strong solution under additional assumptions.
\end{abstract}

\maketitle

\section{Introduction}
While mean field games have been around for quite some time in one form or another, especially in economics, the 
theoretical framework underlying the present work goes back to the pioneering works of Lasry and Lions 
\cite{lasrylionsmfg}, and Huang, Malham\'e and Caines \cite{huangmfg1}. The basic idea is to describe asymptotic consensus among a large population of optimizing individuals interacting with each other in a mean-field way, and subject to constraints of energetic or economical type. The strategy is to take advantage of the mean-field interaction to reduce the analysis of the consensus to a control problem for one single representative individual evolving in, and interacting with, the environment created by
the aggregation of the other individuals. Intuitively, when consensus occurs, symmetries in the system are expected to force the individuals to obey a form law of large numbers and satisfy a propagation of chaos phenomenon  as the size of the population grows. In most of the existing works following \cite{lasrylionsmfg}, and  \cite{huangmfg1}, the sources of randomness in the dynamics of the population are assumed to be independent
from one individual to another.  The purpose of this paper is to analyze the case of correlated randomness in a general setting.  

We concentrate on stochastic differential games for which the epitome of the models can be described as follows.
Given a finite time horizon $T>0$, we start with an $N$-player stochastic differential game, in which the private state process $X^i$ of player $i$ is given by
the solution of the stochastic differential equation:
\begin{align*}
&dX^i_t = b(t,X^i_t,\bar{\mu}^N_t,\alpha^i_t)dt + \sigma(t,X^i_t,\bar{\mu}^N_t)dW^i_t + \sigma_0(t,X^i_t,\bar{\mu}^N_t)dB_t, \quad \text{for} \ t \in [0,T],
\\
&\text{with} \quad \bar{\mu}^N_t = \frac{1}{N}\sum_{j=1}^n\delta_{X^j_t}. 
\end{align*}
Here $B$ is a Wiener process called the \emph{common noise}, and $W^1,\ldots,W^N$ are independent Wiener processes, independent of $B$. The processes $W^1,\ldots,W^N$ are called the \emph{independent} or \emph{idiosyncratic noises}. The objective of player $i$ is to choose a control $\alpha^i$ in order to maximize
the quantity:
\[
J^i(\alpha^1,\dots,\alpha^N) := \E\biggl[\int_0^Tf(t,X^i_t,\bar{\mu}^N_t,\alpha^i_t)dt + g(X^i_T,\bar{\mu}^N_T)\biggr],
\]
the difficulty coming from the fact that these $N$ optimizations are conducted simultaneously.
Besides the correlations coming through the common noise $B$, the optimization problems are coupled through the marginal 
empirical distributions $(\bar{\mu}^N_t)_{t \in [0,T]}$ of the state processes. Additionally, the individuals share the same coefficients and objective functions, and thus the game is symmetric as long as the initial conditions $X^1_0,\ldots,X^N_0$ are exchangeable.

The symmetry is a very important feature of mean field games. However, since the controls are allowed to differ from one player to another, 
the expected reward functionals $J^1,\dots,J^N$ may not be the same. In particular,
except for some very specific cases, there is no hope to find controls $\alpha^1,\dots,\alpha^N$ that maximize
simultaneously  all the reward functionals $J^1,\dots,J^N$. Instead of a global maximizer, the idea
of consensus is formalized by the concept of \textit{Nash equilibrium}. In short, an $N$-tuple $(\alpha^{1,\star},\dots,\alpha^{N,\star})$ is a Nash equilibrium if 
the reward $J^i$ of the particle $i$ attains a maximum at $\alpha^{i,\star}$
when all the other particles $j \not = i$ use the controls $\alpha^{j,\star}$. 
Because of the symmetric structure of the game, it then makes sense to investigate 
the asymptotic behavior of exchangeable equilibria. Drawing intuition from the theory of propagation of chaos, one may anticipate that effective equations may hold in the limit as the number of players $N$ tends to infinity, and hope that their solutions may be more manageable than the search for Nash equilibria for large stochastic differential games of the type described above. This is the rationale for the formulation of the \emph{mean field game} (MFG) problem introduced in \cite{lasrylionsmfg}, and  \cite{huangmfg1}. See also \cite{carmona:delarue:4lyons} for recent developments. We stress that the goal of the present paper is not to justify the passage to the limit, but to study the resulting asymptotic optimization problem. 

This informal discussion suggests that the MFG is essentially an asymptotic formulation of the game, in which the influence of each player on the empirical measure is small, hinting at the fact that the asymptotic optimization problems could be decoupled and identical in nature. Put differently, the limiting equilibrium problem consists of a standard optimization problem for one representative player only (instead of $N$) interacting (competing) with the environment provided by the asymptotic 
behavior (as $N$ tends to $\infty$) of the marginal empirical measures $(\bar\mu^N_{t})_{t \in [0,T]}$ corresponding to an exchangeable equilibrium 
 $(\alpha^{1,\star},\dots,\alpha^{N,\star})$.  Without common noise,  the classical law of large numbers says that  the limit 
 environment should be a deterministic flow of  probability measures $(\mu_{t})_{t \in [0,T]}$ 
describing the statistical distribution of the population in equilibrium. 
When $\sigma_{0}$ is non-zero, the impact of the common noise does not average out, and since it does not disappear, the limiting environment must be given by  a stochastic flow $(\mu_{t})_{t \in [0,T]}$ of probability measures describing the conditional distribution of the population in equilibrium \textit{given the common noise}.
Therefore, we introduce the following generalization to the MFG problem proposed in \cite{lasrylionsmfg,huangmfg1,carmonadelarue-mfg} in the absence of common noise:
\begin{enumerate}
\item For a fixed adapted process $(\mu_t)_{t \in [0,T]}$ with values in the space $\P(\R^d)$ of probability measures on $\R^d$, solve the optimal control problem given by
\begin{align}
\sup_\alpha \ &\E\left[\int_0^Tf(t,X_t,\mu_t,\alpha_t)dt + g(X_T,\mu_T)\right], \text{ such that} \label{introOPT} \\
 dX_t &= b(t,X_t,\mu_t,\alpha_t)dt + \sigma(t,X_t,\mu_t)dW_t + \sigma_0(t,X_t,\mu_t)dB_t. \label{introSDE}
\end{align}
\item Given an optimal control, find the corresponding conditional laws $(\mu^\star_t)_{t \in [0,T]}$ of the optimally controlled state process $(X^\star_t)_{t \in [0,T]}$ given $B$.
\item Find a fixed point $(\mu_t)_{t \in [0,T]}$, such that the resulting $\mu^\star_t$ equals $\mu_t$ for all $t \in [0,T]$.
\end{enumerate}
The \emph{fixed point problem} or \emph{consistency condition} (3) 
characterizes the fact that, under the conditional equilibrium measure, the optimal state 
(conditional on $B$)
must be typical 
of the population. This is exactly the usual MFG problem except for the fact that the solution $(\mu_{t})_{t \in [0,T]}$ is now a \emph{random} measure flow. Again, the conditioning on $B$ appears because the effect of the independent noises $W^i$ on the empirical measure flow averages out as $N$ tends to infinity, but the effect of the common noise $B$ does not.

The goal of this paper is to discuss the existence and, possibly, the uniqueness of an equilibrium in the presence of a common noise. 
Often times,  the proof of the existence of an equilibrium without common noise relies on 
Schauder's fixed point theorem, applied to a compact subset of the space
${\mathcal C}([0,T],{\mathcal P}(\R^d))$ of continuous functions
from $[0,T]$ into the space of probability measures on $\R^d$. 
The application of Schauder's theorem is then quite straightforward
as the standard topology on ${\mathcal C}([0,T],{\mathcal P}(\R^d))$ is simple, 
the compact subsets being easily described by means of classical tightness arguments.
In the presence of a common noise, the problem is much more complicated, as the natural space in which one searches for 
the fixed point is   $[{\mathcal C}([0,T],{\mathcal P}(\R^d))]^{\Omega}$, where $\Omega$ denotes the underlying 
probability space carrying the common noise. Except when 
$\Omega$ is finite, this space is far too large and it is too difficult to find compact subsets left invariant by the transformations of interest.
For that reason, the existence proof is done first on the level of a discretized version of the mean field game, in which the conditioning on
the common noise $B$ in the step (2) of the MFG procedure is replaced by a conditioning on a finitely-supported approximation of $B$. 
The introduction of such a discretization procedure seems to be original in the context of MFG problems, and the approximation of the full fledge
MFG by finite-time finite-space MFG appears to be a powerful idea on its own. See for example \cite{lacker-mfgcontrolledmartingaleproblems} for related developments.
Most importantly, this discretization procedure crucially bypasses a key technical difficulty: in general, the operation of \emph{conditioning} fails to be continuous in any useful sense, and this puts a wrench in any effort to directly apply fixed point theorems. However, when the conditioning $\sigma$-field is finite, enough continuity is recovered; for example, if $X,Y,Y_n$ are random variables, $X$ is nonatomic, and $\G$ is a finite sub-$\sigma$-field of $\sigma(X)$, then $(X,Y_n) \rightarrow (X,Y)$ in distribution implies $\text{Law}(Y_n \ | \ \G) \rightarrow \text{Law}(Y \ | \ \G)$ (weakly) in distribution. Exploiting this remark, the existence proof for the discretized MFG becomes a simple application of Kakutani's fixed point theorem,  Kakutani's fixed point theorem being preferred to Schauder's because of the possible existence of multiple optimal controls.

The existence for the true mean field game is then obtained by refining the discretization, proving tightness of the sequence of solutions and taking limits. In this way, solutions 
are constructed as \textit{weak limits} and read as \textit{weak MFG solutions}. 
The word \emph{weak} refers to the fact that in the limit, the fixed point $(\mu_{t})_{t \in [0,T]}$ may not be adapted to the filtration of the 
common noise $B$ any longer. Such a phenomenon is well-known in stochastic calculus: when 
solving a stochastic differential equation, solutions 
need not be adapted with respect to the noise driving the equation, in which case they are called \textit{weak}. We use here the same terminology. 
Because of that lack of adaptedness, we weaken the fixed point condition and merely require $\mu = \text{Law}(X \ | \ B,\mu)$.

We refer to a solution of the fixed point problem (1--3)
with the more desirable fixed point condition $\mu = \text{Law}(X \ | \ B)$
as a \emph{strong MFG solution}. A strong solution is then a weak solution for which the measure flow $\mu$ happens to be measurable with respect to the common noise $B$. 
Again, the terminology \textit{strong} is used in analogy with the theory of stochastic differential equations. This brings us to the famous result by Yamada and Watanabe
 \cite{yamadawatanabe-uniqueness} 
in stochastic analysis: whenever a stochastic differential equation has the
pathwise uniqueness property, any weak solution is in fact a strong solution. 
In this paper, we develop a similar notion of pathwise uniqueness for mean field games and provide an analog of the theorem of Yamada and Watanabe in this context. From this result we conclude that, whenever pathwise uniqueness holds for a MFG with common noise, the unique weak solution is in fact a strong solution, which then completes our program. 

Our analysis relies on one important additional ingredient.
In order to guarantee compactness (or at least closedness) of the sets of controls 
in a sufficiently weak sense, it is also useful for existence proofs to enlarge the family of admissible controls. Precisely, we allow for relaxed (i.e. measure-valued) controls which may be randomized externally to the inputs of the control problems.
With this extension, we first treat the case when controls take values in 
a compact set and the state coefficients $b$, $\sigma$ and $\sigma_{0}$
are bounded. Another approximation procedure is then needed to derive the general case. 
Existence and the limiting arguments are all derived at the level of the joint law of $(B,W,\mu,\alpha,X)$ in a suitable function space. In the search for a weak MFG solution, the filtration of the control problem is generated by the two Wiener processes $B$ and $W$ but also by the measure flow $\mu$, which we do not require to be adapted to $B$ or $W$. Allowing the controls to be randomized externally to the inputs $(B,W,\mu)$ requires specifying an admissible family of enlargements of the probability space supporting these inputs. Because the original filtration is not necessarily Brownian, special care is needed in choosing the correct type of allowable extensions. This leads to the important, though rather technical, notion of  \emph{compatibility}. The delicate measure theoretic arguments required for the proof are described in detail in  Subsection \ref{subse:setup}.

The main contributions of the paper are as follows. We prove first that there exists a weak MFG solution under general assumptions. Under additional convexity assumptions we derive existence results without relaxed or externally randomized controls. Under a monotonicity assumption due to Lasry and Lions \cite{lasrylionsmfg}, we prove that pathwise uniqueness
holds and, as a consequence, that existence and uniqueness hold in the strong sense.
Our results appear to be the first general existence and uniqueness results for mean field games with common noise, which have been touted in various forms in \cite{ahuja-mfgwellposedness,gueantlasrylionsmfg,gueantlasrylionsmfg-growth,carmonafouque-systemicrisk,carmona:delarue:4lyons,bensoussan-masterequation,gomessaude-mfgsurvey}. The latter papers \cite{carmona:delarue:4lyons,bensoussan-masterequation,gomessaude-mfgsurvey} discuss the formulation of the problem in terms of the \emph{master equation}, which is a single partial differential equation (PDE) in infinite dimension which summarizes the entire system. Ahuja \cite{ahuja-mfgwellposedness} finds (in our terminology) strong solutions of a class of essentially linear-quadratic mean field games with common noise, but with non-quadratic terminal objective $g$. The papers \cite{gueantlasrylionsmfg,gueantlasrylionsmfg-growth} of Gu\'eant et al. solve explicitly some specific common noise models of income distribution. On the other hand, Carmona et al. \cite{carmonafouque-systemicrisk} compute explicit solutions for both the finite-player game and the mean field game in a certain linear-quadratic common noise model, verifying directly the convergence as the number of agents tends to infinity. Although we will not discuss finite-player games in this paper, a follow-up paper will provide rigorous convergence results.

The analysis of this paper allows for degenerate volatilities and thus includes mean field games without common noise (where $\sigma_0 \equiv 0$) and deterministic mean field games (where $\sigma_0 \equiv \sigma \equiv 0$). However, the solutions we obtain still involve \emph{random} measure flows and are thus weaker than the MFG solutions typically considered in the literature. For background on mean field games without common noise, refer to \cite{lasrylionsmfg,cardaliaguet-mfgnotes} for PDE-based analysis and \cite{carmonadelarue-mfg,bensoussan-mfgbook} for a more probabilistic analysis. The analysis of \cite{carmonalacker-probabilisticweakformulation} and especially \cite{lacker-mfgcontrolledmartingaleproblems} are related to ours in that they employ weak formulations of the optimal control problems. The latter paper \cite{lacker-mfgcontrolledmartingaleproblems} especially mirrors ours in several ways, in particular in its use of relaxed controls in conjunction with Kakutani's theorem as well as measurable selection arguments for constructing strict (non-relaxed) controls. However, the presence of common noise necessitates a much more careful formulation and analysis of the problem.

The paper is organized as follows. First, Section \ref{se:strongsolutionsanddiscretization} discusses the main assumptions \ref{assumption:A}, definitions of strong MFG solutions, and existence of discretized MFG solutions. Section \ref{se:weaksolutions} defines weak MFG solutions in detail, discusses some of their properties, and proves existence by refining the discretizations of the previous section and taking limits. Section \ref{se:strictstrongcontrols} discusses how to strengthen the notion of control, providing general existence results without relaxed controls under additional convexity hypotheses. The brief Section \ref{se:counterexamples} discusses two counterexamples, which explain why we must work with weak solutions and why we cannot relax the growth assumptions placed on the coefficients. Uniqueness is studied in Section \ref{se:uniqueness}, discussing our analog of the Yamada-Watanabe theorem and its application to an existence and uniqueness result for strong MFG solutions.

\section{Strong MFG solutions and discretization} \label{se:strongsolutionsanddiscretization}

\subsection{General set-up and standing assumption}
Fix a time horizon $T > 0$. For a measurable space $(\Omega,\F)$, let $\P(\Omega,\F)$ denote the set of probability measures on $(\Omega,\F)$. When the $\sigma$-field is understood, we write simply $\P(\Omega)$. When $\Omega$ is a metric space, let $\B(\Omega)$ denote its Borel $\sigma$-field, and endow $\P(\Omega)$ with the topology of weak convergence.  Let $\C^k = C([0,T];\R^k)$ denote the set of continuous functions from $[0,T]$ to $\R^k$. Define 
the evaluation mappings $\pi_{t}$ on $\C^k$ by $\pi_{t}(x)=x_{t}$ and 
the truncated supremum norms $\|\cdot\|_t$ on $\C^k$ by
\[
\|x\|_t := \sup_{s \in [0,t]}|x_s|, \ t \in [0,T].
\]
Unless otherwise stated, $\C^k$ is endowed with the norm $\|\cdot\|_T$.
Let $\W^k$ denote Wiener measure on $\C^k$. For $\mu \in \P(\C^k)$, let $\mu_t \in \P(\R^k)$ denote the image of $\mu$ under $\pi_{t}$. 
For $p \ge 0$ and a separable metric space $(E,\ell)$, 
let $\P^p(E)$ denote the set of $\mu \in \P(E)$ with $\int_E \ell^p(x,x^0)\mu(dx) < \infty$ for some (and thus for any) $x^0 \in E$. For $p \ge 1$ and $\mu,\nu \in \P^p(E)$, let $\ell_{E,p}$ denote the $p$-Wasserstein distance, given by
\begin{align}
\ell_{E,p}(\mu,\nu) := \inf\left\{\left(\int_{E \times E}\gamma(dx,dy) \ell^p(x,y)\right)^{1/p} : \gamma \in \P(E \times E) \text{ has marginals } \mu,\nu\right\} \label{def:wasserstein}
\end{align}
Unless otherwise stated, the space $\P^p(E)$ is equipped with the metric $\ell_{E,p}$, and $\P(E)$ has the topology of weak convergence. Both are equipped with the corresponding Borel $\sigma$-fields, which coincide
with the $\sigma$-field generated by the mappings $\P^p(E) \, (\text{resp.} \, \P(E))
 \ni \mu \mapsto \mu(F)$, $F$ being any Borel subset of $E$.  
 Appendix \ref{ap:wasserstein} discusses the topological properties of Wasserstein distances relevant to this paper.

We are given two exponents $p',p \ge 1$, a control space $A$, and the following functions:
\begin{align*}
(b,f) &: [0,T] \times \R^d \times \P^p(\R^d) \times A \rightarrow \R^d \times \R, 
\\ 
(\sigma,\sigma_{0}) &: [0,T] \times \R^d \times \P^p(\R^d) \rightarrow \R^{d \times m}
\times \R^{d \times m_0}, 
\\ 
g &: \R^d \times \P^p(\R^d) \rightarrow \R.
\end{align*}
The standing assumptions for our existence and convergence theorems are as follows. Continuity and measurability statements involving $\P^p(\R^d)$ are with respect to the Wasserstein distance $\ell_{\R^d,p}$ and its corresponding Borel $\sigma$-field.

\begin{assumption}{\textbf{A}} \label{assumption:A} 
The main results of the paper will be proved under the following assumptions, which we assume to hold throughout the paper:
\begin{enumerate}
\item[(A.1)] $A$ is a closed subset of a Euclidean space. (More generally, as in \cite{haussmannlepeltier-existence}, a closed $\sigma$-compact subset of a Banach space would suffice.)
\item[(A.2)] $p' > p \ge 1 \vee p_\sigma$, $p_\sigma \in [0,2]$, and $\lambda \in \P^{p'}(\R^d)$. (Here $a \vee b := \max(a,b)$.)
\item[(A.3)] The functions $b$, $\sigma$, $\sigma_0$, $f$, and $g$ of $(t,x,\mu,a)$ are jointly measurable and are continuous in $(x,\mu,a)$ for each $t$.
\item[(A.4)] There exists $c_1 > 0$ such that, for all $(t,x,y,\mu,a) \in [0,T] \times \R^d \times \R^d \times \P^p(\R^d) \times A$,
\begin{align*}
|b(t,x,\mu,a) - b(t,y,\mu,a)| &+ |(\sigma,\sigma_{0})(t,x,\mu) - (\sigma,\sigma_{0})(t,y,\mu)| 
 \le c_1|x-y|,
\end{align*}
and
\begin{equation*}
\begin{split}
&|b(t,0,\mu,a)| \le c_1\left[1 + \left(\int_{\R^d}|z|^p\mu(dz)\right)^{1/p} + |a| \right], 
\\
&|\sigma(t,x,\mu)|^2 + |\sigma_0(t,x,\mu)|^2 \le c_1\left[1 + |x|^{p_\sigma} + \left(\int_{\R^d}|z|^p\mu(dz)\right)^{p_\sigma/p}\right]x	.
\end{split}
\end{equation*}
\item[(A.5)] There exist $c_2, c_3 > 0$ such that, for each $(t,x,\mu,a) \in [0,T] \times \R^d \times \P^p(\R^d) \times A$, 
\begin{align*}
-c_2\left(1 + |x|^p + \int_{\R^d}|z|^p\mu(dz)\right) &\le g(x,\mu) \le c_2\left(1 + |x|^p + \int_{\R^d}|z|^p\mu(dz)\right), \\
-c_2\left(1 + |x|^p + \int_{\R^d}|z|^p\mu(dz) + |a|^{p'}\right) &\le f(t,x,\mu,a) \le c_2\left(1 + |x|^p + \int_{\R^d}|z|^p\mu(dz)\right) - c_3|a|^{p'}.
\end{align*}
\end{enumerate}
\end{assumption}
Examples under which Assumption {\bf A} holds will be discussed in Section \ref{se:counterexamples}.
\subsection{General objective} Ideally, we are interested in the following notion of strong MFG solution:

\begin{definition}[Strong MFG solution with strong control]
\label{def:MFG:strong:strong}
A \emph{strong MFG solution with strong control} 
and with initial condition $\lambda$ is a tuple $(\Omega,
(\F_t)_{t \in [0,T]},P,B,W,\mu,\alpha,X)$, where $(\Omega,(\F_t)_{t \in [0,T]},P)$ is a filtered probability space supporting $(B,W,\mu,\alpha,X)$ satisfying
\begin{enumerate}
\item $(\F_t)_{t \in [0,T]}$ is the $P$-complete filtration generated by the process $(X_0,B_t,W_t)_{t \in [0,T]}$. 
\item The processes 
$(B_{t})_{t \in [0,T]}$ and $(W_{t})_{t \in [0,T]}$ are independent $({\mathcal F}_{t})_{t \in [0,T]}$
Wiener processes of respective dimension $m_0$ and $m$, the processes $(\mu_t = \mu \circ \pi_t^{-1})_{t \in [0,T]}$ 
and $(X_{t})_{t \in [0,T]}$ are 
$(\F_t)_{t \in [0,T]}$-adapted processes (with values in 
${\mathcal P}^p(\R^d)$ and $\R^d$ respectively), and $P \circ X_0^{-1} = \lambda$.
\item $(\alpha_t)_{t \in [0,T]}$ is $(\F_t)_{t \in [0,T]}$-progressively measurable 
with values in $A$ and $\E\int_0^T|\alpha_t|^pdt < \infty$.
\item The state equation holds
\begin{align}
dX_t = b(t,X_t,\mu_t,\alpha_t)dt + \sigma(t,X_t,\mu_t)dW_t + \sigma_0(t,X_t,\mu_t)dB_t, \quad 
t \in [0,T].
 \label{def:SDE-strong-strong}
\end{align}
\item If $(\alpha'_t)_{t \in [0,T]}$ is another $(\F_t)_{t \in [0,T]}$-progressively measurable $A$-valued process
 satisfying $\E\int_0^T|\alpha'_t|^pdt < \infty$, and $X'$ is the unique strong solution of 
\[
dX'_t = b(t,X'_t,\mu_t,\alpha'_t)dt + \sigma(t,X'_t,\mu_t)dW_t + \sigma_0(t,X'_t,\mu_t)dB_t, \ X'_0 = X_0,
\]
then
\[
\E\biggl[\int_0^Tf(t,X_t,\mu_t,\alpha_t)dt + g(X_T,\mu_T)\biggr] \ge \E\biggl[\int_0^Tf(t,X'_t,\mu_t,\alpha'_t)dt + g(X'_T,\mu_T)\biggr].
\]
\item $P$-almost surely, $\mu(\cdot) = P(X \in \cdot | \  B)$. That is, $\mu$ is a version of the conditional law of $X$ given $B$.
\end{enumerate}
\end{definition}
Pay attention that
$E \int_{0}^T \vert \alpha_{t} \vert^{p'} dt$
is not required to be finite. Thanks to (A.5), there is no need. 
When  $E \int_{0}^T \vert \alpha_{t} \vert^{p'} dt=\infty$, the reward functional is well-defined and is equal to 
$-\infty$. 

Definition \ref{def:MFG:strong:strong} may be understood as follows. Points (1), (2) and (3) 
are somewhat technical requirements that fix the probabilistic set-up under which the MFG solution 
is defined. Given $\mu$ as in the definition, (4) and (5) postulate that 
$(X_{t})_{t \in [0,T]}$ is a solution of the stochastic optimal control problem driven 
by the reward functionals $f$ and $g$ in 
the random environment $\mu$. Condition (6) is a fixed point condition. It is an adaptation of the condition 
$\mu = P(X \in \cdot)$ used in the MFG literature to describe asymptotic Nash equilibria between 
interacting particles $X^1,\dots,X^N$ submitted to independent noises:
\begin{equation*}
dX_t^i = b(t,X_t^i,\bar{\mu}_t^N,\alpha_t^i)dt + \sigma(t,X_t^i,\bar{\mu}_t^N)dW_t^i, \quad
i=1,\dots,N, 
\end{equation*}
where $W^1,\dots,W^N$ are independent Wiener processes,
$\bar{\mu}^N_{t}$ is the empirical distribution of the $N$-tuple $(X^1_{t},\dots,X^N_{t})$
and $\alpha^1,\dots,\alpha^N$ are control processes.
In (6), the conditioning by $B$ reflects 
correlations between the particles when their dynamics are governed by a common noise:
\begin{equation*}
dX_t^i = b(t,X_t^i,\bar{\mu}_t^N,\alpha_t^i)dt + \sigma(t,X_t^i,\bar{\mu}_t^N)dW_t^i
+ \sigma_{0}(t,X_{t}^i,\bar{\mu}_t^N) dB_{t}, \quad
i=1,\dots,N, 
\end{equation*}
where $B,W^1,\dots,W^N$ are independent Wiener processes. 
Intuitively, conditioning in (6) follows from a conditional application of the law of large numbers; see 
\cite{carmona:delarue:4lyons} for an overview. In our definition, the equilibrium is called \textit{strong}
as it is entirely described by the common noise $B$.

\begin{remark}
The fact that $(X_{t})_{t \in [0,T]}$ is $(\F_t)_{t \in [0,T]}$-adapted and 
$(B_{t})_{t \in [0,T]}$ is an $(\F_t)_{t \in [0,T]}$-Wiener process in the above definition
implies, with (6), that
$\mu_t = P(X_t \in \cdot \ | B) = P(X_t \in \cdot \ | \sigma(B_s : s \le t))$ $P$ a.s. The filtration being complete,
$(\mu_t)_{t \in [0,T]}$ is automatically 
$(\F_t)_{t \in [0,T]}$-adapted (without requiring it in (2)).
Note also that $(\mu_{t})_{t \in [0,T]}$ has continuous trajectories (in $\P^p(\R^d)$) 
as $\mu$ is ${\mathcal P}^p(\C^d)$-valued.
\end{remark}

We will not be able to prove existence of such a solution under the general assumptions \ref{assumption:A}. It is not until Section \ref{se:uniqueness} that we find additional assumptions which do ensure the existence and uniqueness 
of a strong MFG solution (either in the sense of Definition \ref{def:MFG:strong:strong} or the following weaker Definitions \ref{def:MFG:strong:weak}). Assuming only \ref{assumption:A}, a general existence theorem will hold if we relax the notion of solution. As the first of two relaxations, the class of admissible controls will be enlarged to include what we call \emph{weak controls}. Weak controls are essentially $\P(A)$-valued processes rather than $A$-valued processes, which may be interpreted as a randomization of the control; moreover, weak controls are also allowed to be randomized externally to the given sources of randomness $(X_0,B,W)$. The first of such relaxations we investigate is the following:

\begin{definition}[Strong MFG solution with weak control]
\label{def:MFG:strong:weak}
A \emph{strong MFG solution with weak control} is a tuple $(\Omega,(\F_t)_{t \in [0,T]},P,B,W,\mu,\Lambda,X)$, where $(\Omega,(\F_t)_{t \in [0,T]},P)$ is a probability space with a complete filtration
supporting $(B,W,\mu,\Lambda,X)$ satisfying
\begin{enumerate}
\item The processes 
$(B_{t})_{t \in [0,T]}$ and $(W_{t})_{t \in [0,T]}$ are independent $({\mathcal F}_{t})_{t \in [0,T]}$
Wiener processes of respective dimension $m_0$ and $m$, the processes $(\mu_t = \mu \circ \pi_t^{-1})_{t \in [0,T]}$ 
and $(X_{t})_{t \in [0,T]}$ are 
$(\F_t)_{t \in [0,T]}$-adapted processes (with values in 
${\mathcal P}^p(\R^d)$ and $\R^d$ respectively) and $P \circ X_0^{-1} = \lambda$.
\item $(\Lambda_t)_{t \in [0,T]}$ is $(\F_t)_{t \in [0,T]}$-progressively measurable 
with values in $\P(A)$ and 
\[
\E\int_0^T \int_{A} |a|^p \Lambda_{t}(da) dt < \infty.
\]
\item The state equation holds \footnote{\label{footnote:sde} Throughout the paper, we avoid augmenting filtrations to be right-continuous, mostly because it could cause real problem in point (3) of Definition \ref{def:mfg:weak:weak}. The concerned reader is referred to \cite[Lemma 4.3.3]{stroockvaradhanbook} for a carefully discussion of stochastic integration without completeness or right-continuity of the filtration.}:
\begin{align*}
dX_t = \biggl\{ \int_Ab(t,X_t,\mu_t,a)\Lambda_t(da) \biggr\} dt + \sigma(t,X_t,\mu_t)dW_t + \sigma_0(t,X_t,\mu_t)dB_t.
\end{align*}
\item If $(\Omega',\F'_t,P')$ is another filtered probability space supporting processes $(B',W',\mu',\Lambda',X')$ satisfying (1-3) and $P \circ (B,\mu)^{-1} = P' \circ (B',\mu')^{-1}$, then
\[
\E\biggl[\int_0^T\int_Af(t,X_t,\mu_t,a)\Lambda_t(da)dt + g(X_T,\mu_T)\biggr] \ge \E\biggl[\int_0^T\int_Af(t,X'_t,\mu'_t,a)\Lambda'_t(da)dt + g(X'_T,\mu'_T)\biggr].
\]
\item $\mu$ is a version of the conditional law of $X$ given $B$.
\end{enumerate}
\end{definition}
Note that (1) and (5) are the same as (2) and (6) of Definition \ref{def:MFG:strong:strong}. Given a MFG solution in either of the above senses, we call the corresponding measure flow $(\mu_t)_{t \in [0,T]}$ an \emph{equilibrium}.

\subsection{Relaxed controls}
\label{subse:relaxed:controls}
We now specify the notion of relaxed controls. Recall that Assumption \ref{assumption:A} is in force at all times. Define $\V$ to be the set of measures $q$ on $[0,T] \times A$ satisfying both $q(\cdot \times A) = \text{Lebesgue}$
(that is the image of $q$ by the projection on $[0,T]$ is the Lebesgue measure on $[0,T]$)
 and
\[
\int_{[0,T] \times A}q(dt,da)|a|^p < \infty.
\]
An element of $\V$ is called a \emph{relaxed control}.
Any element $q \in {\mathcal V}$ may be rescaled into $q/T \in {\mathcal P}^p([0,T] \times A)$.
This permits to endow $\V$ with the $p$-Wasserstein metric, denoted by $\ell_{\V}$. 
It follows from results of \cite{jacodmemin-stable} that $\V$ is a Polish space (since $A$ is), and in fact if $A$ is compact then so is $\V$, and in this case $\ell_\V$ metrizes the topology of weak convergence. See Appendix \ref{ap:wasserstein} for some details about this space. 

Each $q \in \V$ may be identified with a measurable function $[0,T] \ni t \mapsto q_t \in \P^p(A)$, determined uniquely (up to a.e. equality) by $dtq_t(da) = q(dt,da)$. As in \cite[Lemma 3.8]{lacker-mfgcontrolledmartingaleproblems}, we can find a predictable version of 
$(q_t)_{t \in [0,T]}$ in the following sense. Let $\Lambda$ denote the identity map on $\V$, and let 
\begin{align}
\F^\Lambda_t := \sigma\left(\Lambda(C) : C \in \B([0,t] \times A)\right), \quad t \in [0,T]
\label{def:flambda}
\end{align}
Then, there exists an $(\F^{\Lambda}_{t})_{t \in [0,T]}$-predictable process $\overline{\Lambda} : [0,T] \times \V \rightarrow \P(A)$ such that, for each $q \in \V$, $\overline{\Lambda}(t,q) = q_t$ for almost every $t \in [0,T]$. In particular, $q = dt[\overline{\Lambda}(t,q)](da)$ for each $q \in \V$, and it is immediate that $\F^\Lambda_t = \sigma(\overline{\Lambda}(s,\cdot) : s \le t)$.
We will abuse notation somewhat by writing $\Lambda_t := \overline{\Lambda}(t,\cdot)$. Before we proceed, we first state a frequently useful version of a standard moment estimate for the state equation (4) in Definition \ref{def:MFG:strong:weak}.  

\begin{lemma} \label{le:stateestimate}
On some filtered probability space $(\Omega,(\F_t)_{t \in [0,T]},P)$, suppose $B$ and $W$ are independent $(\F_t)_{t \in [0,T]}$-Wiener processes, suppose $\mu$ is a $\P^p(\C^d)$-valued random variable such that 
$(\mu_t=\mu \circ \pi_{t}^{-1})$ is 
$(\F_t)_{t \in [0,T]}$-progressive, suppose $(\Lambda_t)_{t \in [0,T]}$ is an $(\F_t)_{t \in [0,T]}$-progressive $\P^p(A)$-valued process, and suppose $\xi$ is a $\F_0$-measurable random vector with law $\lambda$. Assume \ref{assumption:A} holds. Then there exists a unique solution $X$
of the state equation (4) in Definition \ref{def:MFG:strong:weak} with $X_{0}=\xi$ as initial condition.  

For each $\gamma \in [p,p']$, there exists a constant $c_4 > 0$, depending only on $\gamma$, $\lambda$, $T$, and the constant $c_1$ of (A.4) such that, 
\begin{align*}
\E\|X\|_T^\gamma &\le c_4\biggl(1 + \int_{\C^d}
\|z\|_T^\gamma \mu(dz) + \E\int_0^T \int_{A} \vert a \vert^\gamma \Lambda_{t}(da) dt\biggr).
\end{align*}
Moreover, if $P(X \in \cdot \ | \ B) = \mu$, then we have
\begin{align*}
\E \int_{\C^d} \|z\|_T^\gamma \mu(dz) = \E\|X\|_T^\gamma &\le c_4\biggl(1 + \E\int_0^T \int_{A} |a|^\gamma \Lambda_{t}(da) dt\biggr).
\end{align*}
\end{lemma}
\begin{proof}
Existence and uniqueness are standard. The Burkholder-Davis-Gundy inequality and Jensen's inequality yield a constant $C$ (depending only on $\gamma$, $\lambda$, $c_1$, and $T$, and which may then change from line to line) such that, if $\Sigma := \sigma\sigma^\top + \sigma_0\sigma_0^\top$, then
\begin{align*}
\E\|X\|_t^\gamma \le &C\E\left[|X_0|^\gamma + \int_0^tds\int_A\Lambda_s(da)|b(s,X_s,\mu_s,a)|^\gamma + \left(\int_0^tds|\Sigma(s,X_s,\mu_s)|\right)^{\gamma/2}\right] \\
	\le &C\E\biggl\{|X_0|^\gamma +
	c_1^\gamma
	 \int_0^tds
	 \biggl[ 1+ \|X\|_s^\gamma + \biggl(\int_{{\mathcal C}^d}
	 \|z\|_{s}^p \mu(dz) \biggr)^{\gamma/p} + \int_{A} |a|^\gamma \Lambda_s(da) 
	 \biggr]
	   \\
	&\quad\quad+ \biggl[c_{1}\int_0^tds \biggl(
	1 + \|X\|_s^{p_\sigma} + \biggl(\int_{{\mathcal C}^d}\|z\|_s^p\mu(dz)\biggr)^{p_\sigma/p}\biggr)\biggr]^{\gamma/2}\biggr\}
	 \\
	&\le C\E\biggl[1 + |X_0|^\gamma + \int_0^tds\biggl(
	1 + \|X\|_s^\gamma + \int_{{\mathcal C}^d}\|z\|_s^\gamma \mu(dz) + \int_A|a|^\gamma\Lambda_s(da)\biggr)\biggr]
\end{align*}
To pass from the second to the last line, we used the bound
$(\int\|z\|_s^p\mu(dz))^{\gamma/p} \le 
\int\|z\|_s^\gamma\mu(dz)$, which holds true since $\gamma \geq p$.
To bound $(\int\|z\|_s^p\mu(dz))^{p_{\sigma}/p}$ in the third line, we used the 
following argument.
If $\gamma \ge 2$, we can 
pass the power $\gamma/2$ inside 
the integral in time by means of
Jensen's inequality and then use the inequality $|x|^{p_\sigma\gamma/2} \le 1 + |x|^\gamma$, which holds since $p_\sigma \le 2$. If $\gamma \le 2$, we can use the inequality $|x|^{\gamma / 2} \le 1 + |x|$ followed by $|x|^{p_\sigma} \le 1 + |x|^\gamma$, which holds since $\gamma \ge p_\sigma$. The first claim follows now from Gronwall's inequality. 
If $P(X \in \cdot \ | \ B) = \mu$, then the above becomes
\begin{align*}
\E \int_{{\mathcal C}^d} \|z\|_t^\gamma \mu(dz) = \E\|X\|_t^\gamma \le C\E\left[|X|_0^\gamma + \int_0^t\left(1 
+ 2  \int_{{\mathcal C}^d}\|z\|_s^\gamma \mu(dz) + \int_A|a|^\gamma\Lambda_s(da)\right)ds\right].
\end{align*}
The second claim now also follows from Gronwall's inequality.
\end{proof}

\subsection{Discretized mean field games} \label{se:discretizedmfg}
Quite often, existence of a solution to a mean-field game without common noise 
is proved by means of Schauder's fixed point theorem. See
for instance \cite{cardaliaguet-mfgnotes,carmonadelarue-mfg}. Schauder's theorem is then applied 
on ${\mathcal P}^p(\C^d)$ (with $p=2$ in usual cases), 
for which compact subsets may be easily described. In the current setting, the presence of the 
common noise makes things much more complicated. Indeed, an equilibrium, denoted
by $\mu$ in Definitions \ref{def:MFG:strong:strong} and \ref{def:MFG:strong:weak},
is an element of the much larger space $[{\mathcal P}^p({\mathcal C}^d)]^{\C^{m_0}}$,
and the difficulty is to identify compact sets which could be stable under the transformations we consider.

\subsubsection{Set-up}
In this subsection we thus define a discretization of the mean field game
 for which equilibria only depend on a finite number
of random outcomes. Roughly speaking, equilibria can then be viewed as elements of the space    
$[{\mathcal P}^p({\mathcal C}^d)]^k$ for some integer $k \geq 1$, the compact sets of which may be described 
quite simply. Compactness will be much easier to come by when the state coefficients are bounded and the control space compact, and so we will begin the search for MFG solutions by working under the following assumptions:
\begin{assumption}{\textbf{B}} \label{assumption:B}
Assume that the following \ref{assumption:B}(1-5) hold for the rest of the subsection:
\begin{enumerate}
\item[(B.1)] $A$ is a compact metric space, and $(b,\sigma,\sigma_0)$ is uniformly bounded.
\item[(B.2)] $\lambda \in \P^{p'}(\R^d)$, and $p ' > p \ge 1$.
\item[(B.3)] The functions $b$, $\sigma$, $\sigma_0$, $f$, and $g$ of $(t,x,\mu,a)$ are jointly measurable and continuous in $(x,\mu,a)$ for each $t$.
\item[(B.4)] There exists $c_1 > 0$ such that, for all $(t,x,y,\mu,a) \in [0,T] \times \R^d \times \R^d \times \P^p(\R^d) \times A$,
\begin{align*}
|b(t,x,\mu,a) - b(t,y,\mu,a)| + |(\sigma,\sigma_{0})(t,x,\mu) - (\sigma,\sigma_{0})(t,y,\mu)| &\le c_1|x-y|.
\end{align*}
\item[(B.5)] There exists $c_2 > 0$ such that, for each $(t,x,\mu,a) \in [0,T] \times \R^d \times \P^p(\R^d) \times A$, 
\begin{align*}
|f(t,x,\mu,a)| +|g(x,\mu)|  &\le c_2\biggl(1 + |x|^p + \int_{\R^d}|z|^p\mu(dz)\biggr).
\end{align*}
\end{enumerate}
\end{assumption}

Note in particular that $\V$ is compact. Define
then the canonical spaces
\begin{align*}
\Omega_0 := \R^d \times \C^{m_0} \times \C^m, \quad
\Omega_f := \R^d \times \C^{m_0} \times \C^m \times \V \times \C^d.
\end{align*}
Let $\xi$, $B$, $W$, $\Lambda$, and $X$ denote the identity maps on $\R^d$, $\C^{m_0}$, $\C^m$, $\V$, and $\C^d$ respectively. With a
slight abuse of notation, we will also denote by $\xi$, $B$ and $W$
the projections from $\Omega_{0}$ onto $\R^d$, $\C^{m_{0}}$ and $\C^m$ respectively, and 
by $\xi$, $B$, $W$, $\Lambda$ and $X$
the projections from $\Omega_{f}$ onto $\R^d$, $\C^{m_{0}}$, $\C^m$, $\V$ and $\C^d$ respectively. 

The canonical processes $B$, $W$, and $X$ generate obvious natural filtrations on $\Omega_f$: 
$(\F^B_t)_{t \in [0,T]}$, $(\F^W_t)_{t \in [0,T]}$, and 
$(\F^X_t)_{t \in [0,T]}$. Recall the definition of $(\F^\Lambda_t)_{t \in [0,T]}$ on $\V$ from \eqref{def:flambda}. We will frequently work with filtrations generated by several canonical processes, such as 
$\F^{\xi,B,W}_t := \sigma(\xi,B_s,W_s : s \le t)$ defined on $\Omega_0$, and $\F^{\xi,B,W,\Lambda}_t = \F^{\xi,B,W}_t \otimes \F^{\Lambda}_t$ defined on $\Omega_0 \times \V$. When needed, we will use the same symbol $(\F_{t})_{t \in [0,T]}$ to denote the natural extension of a filtration 
$(\F_t)_{t \in [0,T]}$ on a space $\Omega$ to any product $\Omega \times \Omega'$, given by $(\F_t \otimes \{\emptyset, \Omega'\})_{t \in [0,T]}$. This permits to use 
$(\F^{\xi,B,W}_t)_{t \in [0,T]}$ for the filtration on $\Omega_0 \times \V$ generated by $(\xi,B,W)$, and it should be clear from context on which space the filtration is defined.

\subsubsection{Discretization procedure}
To define the discretized MFG problem, we discretize both time and the space of the common noise $B$. For each $n \ge 1$, let $t^n_i = i2^{-n}T$ for $i=0,\ldots,2^n$. For each positive integer $n$, we choose a partition $\pi^n := 
\{C^n_1,\ldots,C^n_{n}\}$ of $\R^{m_0}$ into $n$ measurable sets of strictly positive Lebesgue measure, such that $\pi^{n+1}$ is a refinement of $\pi^n$ for each $n$, and $\B(\R^{m_0}) = \sigma(\bigcup_{n=1}^\infty\pi^n)$.
For a given $n$, the 
time mesh $(t^n_{i})_{i=0,\ldots,2^n}$ and the spatial
partition $\pi^n$ yield a time-space grid along which we can discretize the trajectories in ${\mathcal C}^{m_{0}}$
(which is the space carrying the common noise $B$). Intuitively, the idea is to project 
 the increments of the trajectories between two consecutive times of the mesh 
 $(t^n_{i})_{i=0,\ldots,2^n}$ onto the spatial partition $\pi^n$.
For $1 \le k \le 2^n$ and $\underline{i} = (i_1,\ldots,i_k) \in \{1,\ldots,n\}^k$, 
we thus define $S^{n,k}_{\underline{i}}$ as  the set of trajectories with increments up until time $t_{k}$
in $C^n_{i_{1}},\dots,C^n_{i_{k}}$, that is:
\begin{equation*}
S^{n,k}_{\underline{i}} 
= \{\beta \in \C^{m_0} : \beta_{t^n_j} - \beta_{t^n_{j-1}} \in C^n_{i_j}, \ \forall j=1,\ldots,k\}.
\end{equation*}
Obviously, the $S^{n,k}_{\underline{i}}$'s, $\underline{i} \in  \{1,\ldots,n\}^k$, form a 
finite partition (of cardinal $n^k$) of ${\mathcal C}^{m_{0}}$, each $S^{n,k}_{\underline{i}}$ writing 
as a set of trajectories having the same discretization up until $t_{k}$ and having 
a strictly positive $\W^{m_{0}}$-measure. 
The collection of all the possible discretization classes up until $t_{k}$ thus reads:
\begin{equation*}
\Pi^n_k := \left\{S^{n,k}_{\underline{i}} : \underline{i} \in \{1,\ldots,n\}^k \right\}. 
\end{equation*}
When $k=0$, we let $\Pi^n_0 := \{\C^{m_0}\}$, since all the trajectories are in the same 
discretization class. 

For any $n \geq 0$, the filtration $(\sigma(\Pi^n_{k}))_{k=0,\dots,2^n}$
is the filtration generated by the discretization of the canonical process. Clearly, 
$\sigma(\Pi^n_{k}) \subset \F^B_{t^n_k}$ and $\sigma(\Pi^n_k) \subset \sigma(\Pi^{n+1}_k)$. 
For each $t \in [0,T]$, define
\[
\lfloor t \rfloor_n := \max\left\{t^n_k : 0 \le k \le 2^n, \ t^n_k \le t\right\}.
\]
Let $\Pi^n(t)$ equal $\Pi^n_k$, where $k$ is the largest integer such that $t^n_k \le t$, and let $\G^n_t := \sigma(\Pi^n(t))
= \G^n_{\lfloor t \rfloor_n}$. It is straightforward to verify that $(\G^n_t)_{t \in [0,T]}$ is a filtration (i.e. $\G^n_s \subset \G^n_t$ when $s < t$) for each $n$ and that
\begin{align*}
\F^B_t &= \sigma\biggl(\bigcup_{n=1}^\infty\G^n_t \biggr).
\end{align*}

\subsubsection{Measures parameterized by discretized trajectories}
\label{subsubse:Mn}
The purpose of the discretization procedure described right below is to reduce the complexity 
of the scenarios upon which an equilibrium $\mu$ depends in Definitions 
\ref{def:MFG:strong:strong} and \ref{def:MFG:strong:weak}. Roughly speaking, the strategy is to force
$\mu$ to depend only on the discretization of the canonical process $B$ on $\C^{m_{0}}$. 
A natural way to do so is to restrict (in some way) the analysis to functions $\mu : 
\Pi_{2^n}^n \rightarrow \P^p(\C^d)$ (instead of 
$\mu : 
\C^{m_{0}} \rightarrow \P^p(\C^d)$) 
or equivalently to functions 
$\mu : \C^{m_{0}} \rightarrow \P^p(\C^d)$ that are $\G^n_T$-measurable.
In addition, some adaptedness is needed. We thus let $\M_n$ denote the set of functions 
$\mu : \C^{m_0} \rightarrow \P^p(\C^d)$ that are $\G^n_T$-measurable
such that for each $t \in [0,T]$ and $C \in {\mathcal F}_{t}^X$
the map $\beta \mapsto [\mu(\beta)](C)$ is ${\mathcal G}^n_{t}$-measurable. In 
particular, the process $(\mu_{t} := \mu \circ \pi_{t}^{-1})_{t \in [0,T]}$ is 
$(\G^n_{t})_{t \in [0,T]}$-adapted and c\`adl\`ag (with values in 
$\P^p(\R^d)$).

 Note that any $\mu \in \M_n$ is constant on $S$ for each $S \in \Pi^n_{2^n}$
 in the sense that 
 $\beta \mapsto [\mu(\beta)](F)$ (which depends on the discretized trajectory) is constant on $S$ for each Borel subset $F$ of $\P^p(\C^d)$. Endow $\M_n$ with the topology of pointwise convergence, which of course is the same as the topology of uniform convergence since the common domain of each $\mu \in \M_n$ is effectively $\Pi^n_{2^n}$, which is finite. Since $\G^n_T = \sigma(\Pi^n_{2^n})$ is finite, the space $\M_n$ is homeomorphic to a closed subset of $\P^p(\C^d)^{|\Pi^n_{2^n}|}$. Hence, $\M_n$ is a metrizable closed convex subset of a locally convex topological vector space.

\subsubsection{Control problems} \label{subsubse:control} 
Control problems will be described in terms of measures on $\Omega_{0} \times 
{\mathcal V}$. Let
\begin{align}
\W_\lambda := \lambda \times \W^{m_0} \times \W^m \in \P^{p'}(\Omega_0) \label{def:wlambda}
\end{align}
denote the distribution of the given sources of randomness on $\Omega_{0}$; note that $p'$-integrability follows from the assumption $\lambda \in \P^{p'}(\R^d)$. The set of \emph{admissible control rules} $\A_f$ is defined to be the set of $Q \in \P(\Omega_0 \times \V)$ such that $B$ and $W$ are independent $(\F^{\xi,B,W,\Lambda}_t)_{t \in [0,T]}$-Wiener processes under $Q$ and $Q \circ (\xi,B,W)^{-1} = \W_\lambda$. Equivalently, $Q \in \P(\Omega_0 \times \V)$ is in $\A_f$ if 
$Q \circ (\xi,B,W)^{-1} = \W_\lambda$ and $(B_t-B_s,W_t-W_s)$ is $Q$-independent of $\F^{\xi,B,W,\Lambda}_s$ for each $0 \le s < t \le T$. Intuitively, this is just the set of ``reasonable'' joint laws of the control process with the given randomness. It is easy to check that $\A_f$ is closed in the topology of weak convergence. 

Given $\mu \in \M_n$ and $Q \in \A_f$, on the completion of the filtered probability space $(\Omega_0 \times \V, 
(\F^{\xi,B,W,\Lambda}_t)_{t \in [0,T]},Q)$ we may find a process $Y$ such that $(\xi,B,W,\Lambda,Y)$ satisfy the SDE 
\begin{align}
Y_t &= \xi + \int_0^tds\int_A\Lambda_s(da)b(s,Y_s,\mu_{s}(B),a) 
\nonumber \\
	&\hspace{5pt} 
	+ \int_0^t\sigma(s,Y_s,\mu_{s}(B))dW_s + \int_0^t\sigma_0(s,Y_s,\mu_{s}(B))dB_s. \label{def:SDE-discrete}
\end{align}
Define the law of the solution and the interpolated solution by
\begin{align*}
\RC_f(\mu,Q) &:= Q \circ (\xi,B,W,\Lambda,Y)^{-1}, \quad
\RC^n_f(\mu,Q) := Q \circ (\xi,B,W,\Lambda,\hat{Y}^n)^{-1},
\end{align*}
where, for an element $x \in \C^{d}$, $\hat{x}^n$ is the (delayed) linear interpolation of 
$x$ along the mesh $(t^n_{i})_{i=0,\dots,2^n}$:
\begin{align}
\hat{x}_{t}^n &= \frac{2^n}{T} \bigl( t - t^n_i \bigr) 
x_{t^n_i} + \frac{2^n}{T}\bigl( t^n_{i+1} - t \bigr) x_{t^n_{(i-1)^+}}, \quad \text{for} \ t \in [t^n_{i},t^n_{i+1}], \ i=0,\dots,2^n-1.  \label{eq:interpolation}
\end{align}
The delay ensures that $\hat{X}^n$ is $({\mathcal F}_{t}^X)_{t \in [0,T]}$-adapted.
By Lemma \ref{le:stateestimate} and compactness of $A$,
$\RC_f(\mu,Q)$ and $\RC^n_{f}(\mu,Q)$ are in $\P^p(\Omega_f)$.
Note that $\RC_f$ and $\RC^n_f$ are well-defined; by the uniqueness part in Lemma \ref{le:stateestimate}, 
$\RC_f(\mu,Q)$ is the unique element $P$ of $\P(\Omega_f)$ such that 
$P \circ (\xi,B,W,\mu,\Lambda)^{-1} = Q$ and such that the canonical processes verify the SDE \eqref{def:SDE-discrete} under $P$. Again, as in footnote \ref{footnote:sde} on page \pageref{footnote:sde}, it is no cause for concern that the $Q$-completion of the canonical filtration $(\F^{\xi,B,W,\Lambda}_t)_{t \in [0,T]}$ may fail to be right-continuous.

The objective of the discretized control problem is as follows. Define the reward functional $\Gamma : \P^p(\C^d) \times \V \times \C^d \rightarrow \R$ by
\begin{align}
\Gamma(\mu,q,x) := \int_0^Tdt\int_Aq_t(da)f(t,x_t,\mu_t,a) + g(x_T,\mu_T), \label{def:gamma}
\end{align}
and the expected reward functional $J_f : \M_n \times \P^p(\Omega_f) \rightarrow \R$ by
\begin{align*}
J_f(\mu,P) &:= \E^P\left[\Gamma(\mu(B),\Lambda,X)\right]. 
\end{align*}
For a given $\mu \in {\mathcal M}_{n}$, we are then dealing with the optimal control problem
(with random coefficients) consisting in maximizing 
 $J_{f}(\mu,P)$
over $P \in \RC_{f}^n(\mu,\A_f)$. The set of maximizers is given by 
\begin{align*}
\RC^{\star,n}_f(\mu) &:= \arg\max_{P \in \RC_f^n(\mu,\A_f)}J_f(\mu,P).
\end{align*}
The set $\RC^{\star,n}_f(\mu)$ represents the optimal controls for the $n^{\text{th}}$ discretization corresponding to $\mu$. 
A priori, it may be empty.

\subsubsection{Strong MFG solutions}
The main result of this section is the following theorem, which proves the existence of a strong MFG solution with weak control for our discretized mean field game.

\begin{theorem} \label{th:approxexistence}
For each $n$, there exist $\mu \in \M_n$ and $P \in \RC^{\star,n}_f(\mu,\A_f)$ such that
$\mu = P(\left.X \in \cdot \ \right| \G^n_T)$
($P(\left.X \in \cdot \ \right| \G^n_T)$ being seen as a map from ${\mathcal C}^{m_{0}}$ to ${\mathcal P}^p(\C^d)$, constant on each $S \in \Pi^n_{2^n}$.)
\end{theorem}

\begin{proof}
A MFG equilibrium may be viewed as a fixed point of a set-valued function. Defining the set-valued map $F : \M_n \rightarrow 2^{\M_n}$ (where $2^{\M_{n}}$ is seen as the collection of subsets of $\M_{n}$) by
\[
F(\mu) := \left\{P\left(\left.X \in \cdot \ \right| \G^n_T\right) : P \in \RC^{\star,n}_f(\mu,\A_f)\right\},
\]
the point is indeed to prove that $F$ admits a fixed point, that is a point $\mu \in F(\mu)$. 
Since the unique event in $\G^n_{T}$ of null probability under $P$ is the empty set, we notice 
that $G(P):= P(X \in \cdot \ | \G^n_T)$ is uniquely defined for each $P \in \P^p(\Omega_{f})$.
Let $\P^p_{f}$ denote those elements $P$ of $\P^p(\Omega_{f})$ for which $P \circ (\xi,B,W,\Lambda)^{-1}$ is admissible, that is $\P^p_{f} := \{ P \in \P^p(\Omega_{f}): P \circ (\xi,B,W,\Lambda)^{-1} \in \A_{f}\}$.
For $P \in \P^p_{f}$, $G(P)$ is given by
\begin{equation}
\label{eq:G:discrete}
G(P) : \C^{m_{0}} \ni \beta \mapsto \sum_{S \in \Pi^{n}_{2^n}} P(X \in \cdot \ \vert B \in S) 1_{S}(\beta)=
\sum_{S \in \Pi^{n}_{2^n}} \frac{P(\{X \in \cdot\} \cap  \{B \in S\})}{\W^{m_{0}}(S)} 1_{S}(\beta).
\end{equation}
The very first step is then to check that $F(\mu) \subset {\mathcal M}_{n}$ for each $\mu \in {\mathcal M}_{n}$. 
The above formula shows that, for $P \in \P^p_{f}$,
 $G(P)$ reads as a $\G^n_{T}$-measurable function 
from $\C^{m_{0}}$ to ${\mathcal P}^p(\C^d)$. To prove that $G(P) \in {\mathcal M}_{n}$, 
it suffices to check the adaptedness condition in the definition of ${\mathcal M}_{n}$ (see 
Paragraph \ref{subsubse:Mn}).
For our purpose, we can restrict the proof to the case when 
$X$ is $P$ a.s. piecewise affine as in \eqref{eq:interpolation}. For  each $t \in [0,T]$ and $C \in {\mathcal F}_{t}^X$, we have that 
$1_{C}(X) = 1_{C'}(X)$ $P$ a.s. for some $C' \in  {\mathcal F}_{\lfloor t \rfloor_{n}}^X$. 
Now,
$
\G^n_T = \G^n_{\lfloor t \rfloor_n} \vee \mathcal{H}$,  where 
$\mathcal{H} \subset \sigma(B_s-B_{\lfloor t \rfloor_n} : s \in [{\lfloor t \rfloor_n},T])$. Since 
$\H$ is $P$-independent of $\F^X_{\lfloor t \rfloor_n} \vee \G^n_{\lfloor t \rfloor_n}$, 
we deduce that, $P$ a.s.,
$P(X \in C \ | \G^n_T)=
P(X \in C'  \ | \G^n_{\lfloor t \rfloor_n})$. 
Since the unique event in $\G^n_{T}$ of null probability under $P$ is the empty set, we deduce that
the process $(P(X_{\lfloor t \rfloor_n} \in \cdot \ | \G^n_T))_{t \in [0,T]}$ is 
$(\G^n_{t})_{t \in [0,T]}$-adapted. This shows that
$G(P) \in \M_{n}$ and thus 
$F(\mu) \subset \M_{n}$. 

We will achieve the proof by verifying the hypotheses of the Kakutani-Fan-Glicksberg fixed point theorem
for set-valued functions \cite[Corollary 17.55]{aliprantisborder}. Namely, we will show that $F$ is upper hemicontinuous with nonempty compact convex values, and we will find a compact convex subset $\Q \subset \M_{n}$ such that $F(\mu) \subset \Q$ for each $\mu \in \Q$. 

\subsubsection*{First step: Continuity of set-valued functions}
For the necessary background on set-valued analysis the reader is referred to \cite[Chapter 17]{aliprantisborder}. For this paragraph, fix two metric spaces $E$ and $F$. A set valued function $h : E \rightarrow 2^F$ is \emph{lower hemicontinuous} if, whenever $x_n \rightarrow x$ in $E$ and $y \in h(x)$, there exists $y_{n_k} \in h(x_{n_k})$ such that $y_{n_k} \rightarrow y$. If $h(x)$ is closed for each $x \in E$ then $h$ is called \emph{upper hemicontinuous} if, whenever $x_n \rightarrow x$ in $E$ and $y_n \in h(x_n)$ for each $n$, the sequence $(y_n)$ has a limit point in $h(x)$. We say $h$ is \emph{continuous} if it is both upper hemicontinuous and lower hemicontinuous. If $h(x)$ is closed for each $x \in E$ and $F$ is compact, then $h$ is upper hemicontinuous if and only if its graph $\{(x,y) : x \in E, \ y \in h(x)\}$ is closed.

First, we check the continuity of the function
\[
\P^p_{f} \ni P \mapsto P\left(\left. X \in \cdot \ \right| B \in S\right) \in \P^p(\C^d), \text{ for } 
S \in \Pi^n_{2^n}.
\] 
This is straightforward, thanks to the finiteness of the conditioning $\sigma$-field. Let $\phi : \C^d \rightarrow \R$ be continuous with $|\phi(x)| \le c(1 + 
\|x\|_{T}^p)$ for all $x \in \C^d$, for some $c > 0$. Proposition 
\ref{pr:wasserstein}(3) in Appendix says that it is enough to prove that 
$\E^{P_{k}}[\phi(X) | B \in S] \rightarrow \E^{P}[\phi(X) | B \in S]$ whenever
$P_{k} \rightarrow \P$ in $\P^p(\Omega_f)$. This follows from
Lemma \ref{le:componentwise}, which implies that the following real-valued function is continuous:
\[
\P^p(\Omega_f) \ni P \mapsto \E^P\left[\left.\phi(X) \right| B \in S\right] = \left. \E^P\left[\phi(X) 1_S(B) \right] \right/ \ \W^{m_0}(S).
\]
Basically, Lemma \ref{le:componentwise} handles the discontinuity of the indicator function $1_{S}$
together with the fact that $\phi$ is not bounded.
It follows that the function $G : \P^p_f \rightarrow \M_n$ given by
\eqref{eq:G:discrete} is continuous. The set-valued function $F$ is simply the composition of $G$ with
the set-valued function $\mu \mapsto \RC^{\star,n}_f(\mu,\A_f)$.
Therefore, to prove that $F$ is upper hemicontinuous, it is sufficient to prove that 
$\mu \mapsto \RC^{\star,n}_f(\mu,\A_f)$ is upper hemicontinuous. 

\subsubsection*{Second Step: Analysis of the mapping: $\mu \mapsto \RC^{n}_f(\mu,\A_f)$}
Following the first step, the purpose of the second step is to prove continuity of the set-valued function 
\begin{equation*}
\M_n \ni \mu \mapsto \RC_f^n(\mu,\A_f) := \left\{\RC_f^n(\mu,Q) : Q \in \A_f\right\} \in 2^{\P^p(\Omega_f)}
\end{equation*}
Since the map $\C^d \ni x \mapsto \hat{x}^n \in \C^d$ is continuous (see \eqref{eq:interpolation}), it suffices to prove 
continuity with $\RC_f^n$ replaced by $\RC_f$. To do so, we prove first that $\RC_f(\M_n,\A_f)$ is relatively compact
by showing that each of the sets of marginal measures is relatively compact; see Lemma \ref{le:productrelcompactness}. Clearly $\{P \circ (\xi,B,W)^{-1} : P \in \RC_f(\M_n,\A_f)\} = \{\W_\lambda\}$ is compact in $\P^p(\Omega_0)$. Since $A$ is compact, so is $\V$, and thus $\{P \circ \Lambda^{-1} : P \in \RC_f(\M_n,\A_f)\}$ is relatively compact in $\P^p(\V)$. Since $b$, $\sigma$, and $\sigma_0$ are bounded, Aldous' criterion (see Proposition \ref{pr:itocompact} for details) shows that $\{P \circ X^{-1} : \RC_f(\M_n,\A_f)\}$ is relatively compact in $\P^p(\C^d)$.

Continuity of the set-valued function $\RC_f(\cdot,\A_f)$ will follow from continuity of the single-valued function $\RC_f$. Since the range is relatively compact, it suffices to show that the graph of $\RC_f$ is closed. Let $(\mu_k,Q_k) \rightarrow (\mu,Q)$ in $\M_n \times \A_f$ and $P_k := \RC_f(\mu_k,Q_k) \rightarrow P$ in $\P^p(\Omega_f)$. It is clear that
\[
P \circ (\xi,B,W,\Lambda)^{-1} = \lim_{k\rightarrow\infty}P_k \circ (\xi,B,W,\Lambda)^{-1} = \lim_{k\rightarrow\infty}Q_k = Q.
\]
It follows from the results of Kurtz and Protter \cite{kurtzprotter-weakconvergence} that the state SDE 
\eqref{def:SDE-discrete} holds under the limiting measure $P$, since it holds under each $P_k$. 
Since $\RC_f(\mu,Q)$ is the unique law on $\Omega_{f}$ under which $(\xi,B,W,\Lambda)$ has law
$Q$ and $(\xi,B,W,\Lambda,X)$ solves \eqref{def:SDE-discrete}, we deduce that 
$P = \RC_{f}(\mu,Q)$. We finally conclude that $\RC_{f}(\cdot,\A_{f})$ and thus $\RC_{f}^n(\cdot,\A_{f})$ are continuous. 

\subsubsection*{Third Step: Analysis of the mapping: $\mu \mapsto \RC^{\star,n}_f(\mu,\A_f)$}
As a by-product of the previous analysis, we notice that, for each $\mu \in \M_{n}$, 
$\RC_{f}(\mu,\A_{f})$ is closed and relatively compact and thus compact. By continuity of the 
map $\C^d \ni x \mapsto \hat{x}^n \in \C^d$ (see \eqref{eq:interpolation}), $\RC_{f}^n(\mu,\A_{f})$
is also compact. 

Since $f$ and $g$ are continuous in $(x,\mu,a)$ and have $p$-order growth, 
it can be checked that the reward functional $\Gamma$ is continuous (although quite elementary, the proof is given in 
Appendix, see
Lemma \ref{le:usc}).
This implies that the expected reward functional 
$$\M_n \times \P^p(\Omega_f) \ni (\mu,P) \mapsto J_f(\mu,P) \in \R$$
is also continuous. If $\Gamma$ is bounded, continuity follows
from the fact that $(\mu_{k},P_{k}) \rightarrow (\mu,P)$ implies 
$P_{k} \circ (\mu_{k}(B),\Lambda,X)^{-1} \rightarrow 
P \circ(\mu(B),\Lambda,X)^{-1}$. In the general case when 
$\Gamma$
has $p$-order growth, it follows from Lemma \ref{le:componentwise}.

By compactness of $\RC_{f}^n(\mu,A_{f})$ and by continuity of $J_{f}$, $\RC_{f}^{\star,n}(\mu,A_{f})$ is nonempty and compact. 
Moreover, from a well known theorem of Berge \cite[Theorem 17.31]{aliprantisborder},  
the set-valued function $\RC^{\star,n}_f : \M_n \rightarrow 2^{\P(\Omega_f)}$ is upper hemicontinuous. 

\subsubsection*{Fourth step: Convexity of $\RC^{\star,n}_{f}(\mu,\A_{f})$}
We now prove that, for each $\mu \in {\mathcal M}_{n}$,
$\RC_f^n(\mu,\A_f)$ is convex. By linearity of the 
map $\C^d \ni x \mapsto \hat{x}^n \in \C^d$ (see \eqref{eq:interpolation}), 
it is sufficient to prove that $\RC_f(\mu,\A_f)$ is convex.
To this end, we observe first that $\A_{f}$ is convex. Given 
$Q_{i}$, $i=1,2$, in $\A_{f}$, and $c \in (0,1)$, we notice that 
$(B,W)$ is a Wiener process with respect to 
$({\mathcal F}^{\xi,B,W,\Lambda}_{t})_{t \in [0,T]}$
under $cP^1 + (1-c)P^2$, where 
$P^i : =\RC_{f}(\mu,Q^i)$ for $i=1,2$. (Use the fact that $(B,W)$ is a Wiener process under both 
$P^1$ and $P^2$.) Moreover, the state equation holds under $cP^1+(1-c)P^2$. Since  
$(cP^1+(1-c)P^2) \circ (\xi,B,W,\Lambda)^{-1}=c Q^1+(1-c) Q^2$, we deduce that 
$cP^1+(1-c) P^2$ is the unique probability on $\Omega_{f}$ under which 
$(\xi,B,W,\Lambda)$ has law $c Q^1+(1-c)Q^2$ and 
$(\xi,B,W,\Lambda,X)$ solves the state equation. This proves that 
$cP^1 + (1-c)P^2 = \RC_f(\mu,cQ^1 + (1-c)Q^2)$.

By linearity of the map $P \mapsto J_f(\mu,P)$, we deduce that 
the set-valued function $\RC^{\star,n}_f : \M_n \rightarrow 2^{\P(\Omega_f)}$ has nonempty convex values. 
(Non-emptiness follows from the previous step.) 

\subsubsection*{Conclusion}

Finally, we place ourselves in a convex compact subset of $\M_n$, by first finding a convex compact set $\Q_0 \subset \P^p(\C^d)$ containing $\{P \circ X^{-1} : P \in \RC_f^n(\M_n,\A_f) \}$. To this end, note that the boundedness of $(b,\sigma,\sigma_0)$ of assumption (B.1) implies that for each smooth $\phi : \R^d \rightarrow \R$ with compact support,
\[
C_\phi := \sup_{t,x,\mu,a}\Bigl|b(t,x,\mu,a)^\top D\phi(x) + \frac{1}{2}\text{Tr}\left[(\sigma\sigma^\top+\sigma^{\vphantom{\top}}_0\sigma_0^\top)(t,x,\mu)D^2\phi(x)\right]\Bigr| < \infty,
\]
where $D$ and $D^2$ denote gradient and Hessian, respectively. Following Lemma \ref{le:stateestimate} and again using boundedness of $(b,\sigma,\sigma_0)$, it is standard to show that
\[
M := \sup\left\{\E^P\|X\|^{p'}_T : P \in \RC_f^n(\M_n,\A_f)\right\} < \infty.
\]
Now, define $\Q_1$ to be the set of $P \in \P^p(\C^d)$ satisfying
\begin{enumerate}
\item $P \circ X_0^{-1} = \lambda$, 
\item $\E^P\|X\|^{p'}_T \le M$,
\item for each nonnegative smooth $\phi : \R^d \rightarrow \R$ with compact support, the process $\phi(X_t) + C_\phi t$ is a $P$-submartingale,
\end{enumerate}
It is clear that $\Q_1$ is convex and contains  $\{P \circ X^{-1} : P \in \RC_f^n(\M_n,\A_f) \}$. Using a well known tightness criterion of Stroock and Varadhan \cite[Theorem 1.4.6]{stroockvaradhanbook}, conditions (1) and (3) together imply that $\Q_1$ is tight, and the $p'$-moment bound of (2) then ensures that it is relatively compact in $\P^p(\C^d)$ (see Proposition \ref{pr:wasserstein}). It is straightforward to check that $\Q_1$ is in fact closed, and thus it is compact. Next, define
\[
\Q_2 := \left\{P \circ (\hat{X}^n)^{-1} : P \in \Q_1\right\} \subset \P^p(\C^d),
\]
and note that $\Q_2$ is also convex and compact, since $x \mapsto \hat{x}^n$ is continuous and linear.

Recalling the definition of $\P^p_f$ from the first step, let
\[
\Q_3 := \left\{P \in \P^p_f : \ P \circ X^{-1} \in \Q_2\right\} =\left\{P \in \P^p(\Omega_f) : P \circ (\xi,B,W,\Lambda)^{-1} \in \A_f, \ P \circ X^{-1} \in \Q_2\right\}.
\]
It is easily checked that $\A_f$ is a compact set: closedness is straightforward, and, as in the second step, $\A_f$ is relatively compact since $A$ is compact and the $(\xi,B,W)$-marginal is fixed. It follows from compactness of $\A_f$ and $\Q_2$ that $\Q_3$ is compact (see Lemma \ref{le:productrelcompactness}). Similarly, it follows from convexity of $\A_f$ and $\Q_2$ that $\Q_3$ is convex.

Finally, define $\Q := G(\Q_3)$. Note that $\Q \subset \M_n$, since we saw at the beginning of the proof that indeed $G(P) \in \M_n$ whenever $P \in \P^p_f$ satisfies $P(X = \hat{X}^n)=1$. As emphasized by \eqref{eq:G:discrete}, $G$ is linear. Hence, $\Q$ is convex and compact since $\Q_3$ is. Moreover, for each $\mu \in \M_n$, $F(\mu) = G(\RC^{\star,n}_f(\mu,\A_f))$ is convex and compact, since $\RC^{\star,n}_f(\mu,\A_f)$ is convex and compact (see the third and fourth steps). Since $F(\mu) \subset \Q$ for each $\mu \in \Q$, the proof is complete.
\end{proof}

\section{Weak limits of discretized MFG} 
\label{se:weaksolutions}
We now aim at passing to the limit in the discretized MFG as the time-space grid is refined, the limit being taken in the weak sense. To do so, we show that any sequence of solutions of the discretized MFG is relatively compact, and we characterize the limits. 
This requires a lot of precaution, the main reason being that measurability properties are not preserved under weak limits. In particular, we cannot generally ensure 
that in the limit, the conditional measure $\mu$ remains $B$-measurable in the limit. This motivates the new notion of \textit{weak} MFG solution in the spirit of weak solutions to standard stochastic differential equations. We will thus end up with \textit{weak solutions with weak controls}.  
Assumption \ref{assumption:A} holds throughout the section.

\subsection{Weak MFG solution with weak control}

Since the conditional measure $\mu$ is no longer expected to be measurable with respect to 
$B$, we need another space for it. One of the main idea in the sequel is 
to enlarge the space supporting $\mu$. Namely, instead of considering $\mu$ as the
conditional distribution of $X$ given some $\sigma$-field, we will see $\mu$ as the conditional distribution of 
the whole $(W,\Lambda,X)$. This will allow us to describe in a complete way the correlations between the different processes. 
In other words, $\mu$ will be viewed as an element of ${\mathcal P}^p(\X)$, with 
$\X := \C^m \times \V \times \C^d$, and with $\mu^x := \mu(\C^m \times \V \times \cdot)$ denoting the $\C^d$-marginal.

This brings us to the following definition of a weak MFG solution, the term \textit{weak} referring to the fact 
that the conditional distribution $\mu$ may not be adapted to the noise $B$:

\begin{definition}[Weak MFG solution with weak control]
\label{def:mfg:weak:weak}
A \emph{weak MFG solution with weak control} (or simply a \emph{weak MFG solution})
 with initial condition $\lambda$ is a tuple $(\Omega,
(\F_t)_{t \in [0,T]},P,B,W,\mu,\Lambda,X)$, where $(\Omega,(\F_t)_{t \in [0,T]},P)$ is a probability space with a complete filtration
supporting $(B,W,\mu,\Lambda,X)$ satisfying
\begin{enumerate}
\item The processes 
$(B_{t})_{t \in [0,T]}$ and $(W_{t})_{t \in [0,T]}$ are independent $({\mathcal F}_{t})_{t \in [0,T]}$
Wiener processes of respective dimension $m_0$ and $m$, the process $(X_{t})_{t \in [0,T]}$ is 
$(\F_t)_{t \in [0,T]}$-adapted with values in 
$\R^d$, and $P \circ X_0^{-1} = \lambda$. Moreover, $\mu$ is a random element of $\P^p(\X)$ such that $\mu(C)$ is $\F_t$-measurable for each $C \in \F^{W,\Lambda,X}_t$ and $t \in [0,T]$.
\item $X_0$, $W$, and $(B,\mu)$ are independent.
\item $(\Lambda_t)_{t \in [0,T]}$ is $(\F_t)_{t \in [0,T]}$-progressively measurable with values in $\P(A)$ and 
\[
\E\int_0^T \int_{A} |a|^p \Lambda_{t}(da) dt < \infty.
\]
Moreover, $\sigma(\Lambda_s : s \le t)$ is conditionally independent of $\F^{X_0,B,W,\mu}_T$ given $\F^{X_0,B,W,\mu}_t$ for each $t \in [0,T]$, where
\[
\F^{X_0,B,W,\mu}_t = \sigma(X_0,B_s,W_s : s \le t) \vee \sigma\left(\mu(C) : C \in \F^{W,\Lambda,X}_t\right).
\]
\item The state equation holds:
\begin{align}
dX_t = \int_Ab(t,X_t,\mu^x_t,a)\Lambda_t(da)dt + \sigma(t,X_t,\mu^x_t)dW_t + \sigma_0(t,X_t,\mu^x_t)dB_t. \label{def:relaxedSDE-weak}
\end{align}
\item If $(\Omega',(\F'_t)_{t \in [0,T]},P')$ is another filtered probability space supporting processes $(B',W',\nu,\Lambda',X')$ satisfying (1-4) and $P \circ (X_0,B,W,\mu)^{-1} = P' \circ (X_0',B',W',\nu)^{-1}$, then
\[
\E^P\left[\Gamma(\mu^x,\Lambda,X)\right] \ge \E^{P'}\left[\Gamma(\nu^x,\Lambda',X')\right].
\]
where $\Gamma$ was defined in \eqref{def:gamma}.
\item $\mu = P((W,\Lambda,X) \in \cdot \ | \ B,\mu)$ a.s. That is $\mu$ is a version of the conditional law of $(W,\Lambda,X)$ given $(B,\mu)$.
\end{enumerate}
If there exists an $A$-valued process $(\alpha_t)_{t \in [0,T]}$ such that $P(\Lambda_t = \delta_{\alpha_t} \ a.e. \ t) = 1$, then we say the tuple is a \emph{weak MFG solution with weak strict control}.
It is said to be a \emph{weak MFG solution with strong control} if the 
process $(\alpha_{t})_{t \in [0,T]}$ is progressive with respect to the $P$-completion of $({\mathcal F}_{t}^{X_0,B,W,\mu})_{t \in [0,T]}$. 
\end{definition}

A few comments regarding this definition are in order. 
The MFG solution is strong (see  
Definitions \ref{def:MFG:strong:strong} and \ref{def:MFG:strong:weak}),
if $\mu$ is $B$-measurable, and it is weak otherwise. Similarly, whether or not $\mu$ is $B$-measurable, the control 
is weak if it is not progressively measurable with respect to the completion of $(\F^{X_0,B,W,\mu}_t)_{t \in [0,T]}$. Note finally that assumption (6) in the definition of weak MFG solution with weak control ensures that $\mu^x_t$ is $\F_t$-adapted, as will be seen in Remark \ref{re:tmarginals}.

Since this notion of ``weak control'' is unusual, especially the conditional independence requirement in (3), we offer the following interpretation. An agent has full information, in the sense that he observes (in an adapted fashion) the initial state $X_0$, the noises $B$ and $W$, and also the distribution $\mu$ of the (infinity of) other agents' states, controls, and noises. That is, the agent has access to $\F^{X_0,B,W,\mu}_t$ at time $t$. Controls are allowed to be randomized externally to these observations, but such a randomization must be \emph{conditionally independent of future information given current information}. This constraint will be called \emph{compatibility}.

The main result of this section is:

\begin{theorem} \label{th:existence}
Under assumption \ref{assumption:A}, there exists a weak MFG solution with weak control that 
satisfies (with the notation of 
Definition \ref{def:mfg:weak:weak})
 $\E\int_0^T \int_{A} |a|^{p'} \Lambda_{t}(da) dt < \infty$.
\end{theorem}

\subsection{Canonical space}
\label{subse:setup}
In order to take weak limits of the discretized MFG, which is our purpose, 
it is convenient to work on a canonical space. 
As in the previous section, 
$\Omega_0 := \R^d \times \C^{m_0} \times \C^m$ will support 
the initial condition and the two Wiener processes driving the state equation. 
We also need the space $\V$ defined in the previous Subsection
\ref{subse:relaxed:controls} to handle the relaxed controls and the space $\C^d$ to handle the solution of the state equation. To sum up, we have:
\begin{align*}
\X := \C^m \times \V \times \C^d, \quad
\Omega_0 := \R^d \times \C^{m_0} \times \C^m, \quad
\Omega := \R^d \times \C^{m_0} \times \C^m \times \P^p(\X) \times \V \times \C^d.
\end{align*}
The identity map on $\Omega_{0}$ is still denoted by $(\xi,B,W)$
and the identity map on $\Omega$ by $(\xi,B,W,\mu,\Lambda,X)$. The map
$\mu$ generates the canonical filtration
\begin{equation}
\label{eq:filtration:mu}
\F^\mu_t := \sigma\left(\mu(C) : C \in \F^{W,\Lambda,X}_t\right).
\end{equation}
Recall from \eqref{def:flambda} the definition of the canonical filtration $(\F^\Lambda_t)_{t \in [0,T]}$ on $\V$,
and recall from \eqref{def:wlambda} the definition of $\W_\lambda \in \P(\Omega_0)$. 
We next specify how $\mu$ and $\Lambda$ are allowed
to correlate with each other and with the given sources of randomness $(\xi,B,W)$.
We will refer to the conditional independence requirement (3) of Definition 
\ref{def:mfg:weak:weak} as \emph{compatibility}, defined a bit more generally as follows:

\begin{enumerate}
\item An element $\rho \in \P^p(\Omega_{0} \times \P^p(\X))$ is said to be in $\P^p_{c}[(\Omega_{0},\W_{\lambda}) \leadsto 
{\mathcal P}^p(\X)]$ if 
$(\xi,B,W)$ has law $\W_{\lambda}$ under $\rho$ 
and if 
$B$ and $W$ are independent $(\F^{\xi,B,W,\mu}_t)_{t \in [0,T]}$-Wiener processes
under $\rho$. The subscript $c$ and the symbol $\leadsto$ in 
$\P^p_{c}[(\Omega_{0},\W_{\lambda})\leadsto 
{\mathcal P}^p(\X)]$ indicate that the 
extension of the probability measure $\W_\lambda$ from $\Omega_{0}$ to $\Omega_{0} \times 
{\mathcal P}^p(\X)$ is \textit{compatible}. 
\item For $\rho\in \P^p(\Omega_{0} \times \P^p(\X))$, an element $Q \in \P^p(\Omega_{0} \times \P^p(\X)\times \V)$ is said to be in $\P^p_{c}[(\Omega_{0} \times {\mathcal P}^p(\X),
\rho) \leadsto \V]$ if $(\xi,B,W,\mu)$ has law $\rho$ under $Q$ and 
${\mathcal F}_{T}^{\xi,B,W,\mu}$ and ${\mathcal F}_{t}^{\Lambda}$
are conditionally independent given ${\mathcal F}_{t}^{\xi,B,W,\mu}$. Again, 
$Q$ is then \textit{compatible} with $\rho$ in the sense that, given the
observation of $(\xi,B,W,\mu)$ up until time $t$, the observation of $\Lambda$
up until $t$
has no influence on the future of $(\xi,B,W,\mu)$. 
\end{enumerate}

\begin{remark}
These notions of compatibility are special cases of a more general idea, which goes by several names in the literature. It can be viewed as a compatibility of a larger filtration with a smaller one on a single probability space, in which case this is sometimes known as the \emph{H-hypothesis} \cite{bremaudyor-changesoffiltrations}. Alternatively, this can be seen as a property of an extension of a filtered probability space, known as a \emph{very good extension} \cite{jacodmemin-weaksolution} or \emph{natural extension} \cite{elkaroui-compactification}. The term \emph{compatible} is borrowed from Kurtz \cite{kurtz-yw2013}. The curious reader is referred to \cite{bremaudyor-changesoffiltrations,jacodmemin-weaksolution,kurtz-yw2013} for some equivalent definitions, but we will derive the needed results as we go, to keep the paper self-contained.
\end{remark}

We now have enough material to describe the optimization problem we will deal with. 
Given $\rho \in \P^p(\Omega_{0} \times \P^p(\X))$ (that is given the  
original sources of randomness and a compatible random measure), 
we denote by ${\mathcal A}(\rho) := \P^p_{c}[(\Omega_{0} \times {\mathcal P}^p(\X),
\rho) \leadsto \V]$ (see (2) above) the \textit{set of admissible relaxed controls}. 

Observe from (1) and (2) right above that, for $\rho
\in \P^p_{c}[(\Omega_{0},\W_{\lambda}) \leadsto 
{\mathcal P}^p(\X)]$ and $Q \in {\mathcal A}(\rho)$, the process $(B,W)$ is a Wiener process
with respect to the filtration $({\mathcal F}_{t}^{\xi,B,W,\mu,\Lambda})_{t \in [0,T]}$. 
Following (1), we will denote by $\P^p_{c}[(\Omega_{0},\W_{\lambda}) \leadsto 
{\mathcal P}^p(\X) \times \V]$ the elements of $\P^p(\Omega_{0} \times 
{\mathcal P}^p(\X) \times \V)$ under which $(B,W)$ is a Wiener process
with respect to the filtration $({\mathcal F}_{t}^{\xi,B,W,\mu,\Lambda})_{t \in [0,T]}$, so that, if 
$Q \in {\mathcal A}(\rho)$ with $\rho \in \P^p_{c}[(\Omega_{0},\W_{\lambda}) \leadsto 
{\mathcal P}^p(\X)]$, then $Q \in \P^p_{c}[(\Omega_{0},\W_{\lambda}) \leadsto 
{\mathcal P}^p(\X) \times \V]$.

For $Q \in \P^p_{c}[(\Omega_{0},\W_{\lambda}) \leadsto 
{\mathcal P}^p(\X) \times \V]$,  $\Lambda$ is $p$-integrable, that is 
$\E^Q\int_0^T\int_A|a|^p\Lambda_t(da)dt < \infty$. On the completion of the space $(\Omega_0 \times \P^p(\X) \times \V,( \F^{\xi,B,W,\mu,\Lambda}_t)_{t \in [0,T]},Q)$ there exists a unique strong solution $X$ of the SDE
\begin{align}
X_t = \xi + \int_0^tds\int_A\Lambda_s(da)b(s,X_s,\mu^x_s,a) + \int_0^t\sigma(s,X_s,\mu^x_s)dW_s + \int_0^t\sigma_0(s,X_s,\mu^x_s)dB_s. \label{def:SDE}
\end{align}
where we recall that $\mu^x (\cdot) = \mu (\C^m \times \V \times \cdot)$ is the marginal 
law of $\mu$ on $\C^d$ and $\mu^x_{s} := \mu^x \circ \pi_{s}^{-1}$. 
We then denote by $\RC(Q) := Q \circ (\xi,B,W,\mu,\Lambda,X)^{-1} \in \P(\Omega)$ the joint law of the solution. 
$\RC(Q)$ is the unique element $P$ of $\P(\Omega)$ such that $P \circ (\xi,B,W,\mu,\Lambda)^{-1} = Q$ and such that the canonical processes verify the SDE \eqref{def:SDE} under $P$ 
(again, see footnote${ \ }^{\ref{footnote:sde}}$
 on page
\pageref{footnote:sde} for a related discussion about the choice of the filtration). It belongs to 
$\RC(Q) \in \P^p(\Omega)$, see Lemma \ref{le:stateestimate}.

For each $\rho \in \P^p_{c}[(\Omega_{0},\W_{\lambda}) \leadsto 
{\mathcal P}^p(\X)]$, define
\[
\RC\A(\rho) := \RC(\A(\rho)) = \left\{\RC(Q) : Q \in \A(\rho)\right\}.
\]
Recalling the definition of $\Gamma$ from \eqref{def:gamma}, the expected reward functional $J : \P^p(\Omega) \rightarrow \R$ is defined by
\begin{equation}
\label{eq:cost:functional}
J(P) := \E^P\left[\Gamma(\mu^x,\Lambda,X)\right].
\end{equation}
The problem of maximizing $J(P)$ over $P \in \RC\A(\rho)$ is called the \emph{control problem associated to $\rho$}. Define the set of optimal controls corresponding to $\rho$ by
\begin{equation}
\label{eq:astar}
\A^\star(\rho) := \arg\max_{Q \in \A(\rho)}J(\RC(Q)),
\end{equation}
and note that
\begin{align*}
\RC\A^\star(\rho) &:= \RC(\A^\star(\rho)) = \arg\max_{P \in \RC\A(\rho)}J(P).
\end{align*}
Pay attention that, a priori, the set $\A^\star(\rho)$ may be empty.

\subsection{Relative compactness and MFG pre-solution}
With the terminology introduced above, we make a useful intermediate definition:
\begin{definition}[MFG pre-solution] 
\label{def:mfgsolution}
Suppose $P \in \P^p(\Omega)$ satisfies the following:
\begin{enumerate}
\item $(B,\mu)$, $\xi$ and $W$ are independent under $P$.
\item $P \in \RC\A(\rho)$ where $\rho := P \circ (\xi,B,W,\mu)^{-1}$ is in 
$\P^p_c[(\Omega_{0},\W_{\lambda}) \leadsto
\P^p(\X)]$.
\item $\mu = P((W,\Lambda,X) \in \cdot \ | \ B,\mu)$ a.s. That is, $\mu$ is a version of the conditional law of $(W,\Lambda,X)$ given $(B,\mu)$.
\end{enumerate}
Then we say that $P$ is a \emph{MFG pre-solution}. 
\end{definition}

\begin{remark} \label{re:tmarginals}
If $P$ is a MFG pre-solution then the condition (3) implies that $\mu^x_t = P(X_t \in \cdot \ | \ \F^{B,\mu^x}_t)$ for each $t$, where
\[
\F^{B,\mu^x}_t := \sigma(B_s, \mu^x_s : s \le t).
\]
Indeed, for any bounded measurable $\phi : \R^d \rightarrow \R$, since $\F^{B,\mu^x}_t \subset \F^{B,\mu}_T$ and $\mu^x_t$ is $\F^{B,\mu^x}_t$-measurable, we may condition by $\F^{B,\mu^x}_t$ on both sides of the equation $\E[\phi(X_t) \ | \ \F^{B,\mu}_T] = \int\phi\,d\mu^x_t$ to get the desired result. More carefully, this tells us $\E[\phi(X_t) \ | \ \F^{B,\mu^x}_t] = \int\phi\,d\mu^x_t$ a.s. for each $\phi$, and by taking $\phi$ from a countable sequence which is dense in pointwise convergence we conclude that $\mu^x_t$ is a version of the regular conditional law of $X_t$ given $\F^{B,\mu^x}_t$.
\end{remark}

Definition \ref{def:mfgsolution} is motivated by:

\begin{lemma} \label{le:limitpresolution}
Assume that \ref{assumption:B} holds. 
For each $n$, by Theorem \ref{th:approxexistence} we may find $\mu^n \in \M_n$ and $P_n \in \RC^{\star,n}_f(\mu^n,\A_f)$ such that $\mu^n = P_n(X \in \cdot \ | \ \G^n_T)$
(both being viewed as random probability measures on $\C^d$). On $
\X$, define
\[
\bar{\mu}^n := P_n\left((W,\Lambda,X) \in \cdot \ | \ \G^n_T\right),
\]
so that $\bar{\mu}^n$ can be viewed as a map
from $\C^{m_{0}}$ into $\P^p(\X)$ and 
$\bar{\mu}^n(B)$
as a random element of $\P^p(\X)$. 
Then the probability measures
\[
\overline{P}_n := P_n \circ (\xi,B,W,\bar{\mu}^n(B),\Lambda,X)^{-1}
\]
are relatively compact in $\P^p(\Omega)$, 
and every limit point is a MFG pre-solution.
\end{lemma}

\begin{proof}
\textit{First step.} Write $P_n = \RC^n_f(\mu^n,Q_n)$, for some $Q_n \in \A_f$, and define $P'_n := \RC_f(\mu^n,Q_n)$. Let
\[
\overline{P}'_n = P'_n \circ (\xi,B,W,\bar{\mu}^n(B),\Lambda,X)^{-1},
\]
so that $\overline{P}_n = \overline{P}'_n \circ (\xi,B,W,\mu,\Lambda,\hat{X}^n)^{-1}$, where $\hat{X}^n$ was defined in \eqref{eq:interpolation}. We first show that $\overline{P}'_n$ are relatively compact in $\P^p(\Omega)$. Clearly $P'_n \circ (B,W)^{-1}$ are relatively compact, and so are $P'_n \circ \Lambda^{-1}$ by compactness of $\V$. Since $A$ is compact, the moment bound of Lemma \ref{le:stateestimate} yields
\begin{align}
\sup_n\E^{\overline{P}'_n}\int_{\C^d}\|x\|_T^{p'}\mu^x(dx) < \infty. \label{pf:exist001}
\end{align}
Thus $P'_n \circ X^{-1}$ are relatively compact, by an application of Aldous' criterion (see Proposition \ref{pr:itocompact}). 
By Proposition \ref{pr:ptight}, relative compactness of $P'_n \circ (\bar{\mu}^n(B))^{-1}$ follows from that of the mean measures $P'_n \circ (W,\Lambda,X)^{-1}$ and from the uniform $p'$-moment bound of Lemma \ref{le:stateestimate}. Precisely, for a point $\chi_{0} \in \X$ and a metric $\ell$ on $\X$ compatible with the topology, 
\begin{equation*}
\begin{split}
\sup_{n}
\int_{\Omega} \biggl( \int_{\X} 
\ell^{p'}(\chi_{0},\chi) [\bar{\mu}^n(B)](d\chi) \biggr) dP'_{n}
&= \sup_{n}
{\mathbb E}^{P'_{n}}\bigl[ {\mathbb E}^{P'_{n}} \bigl[
\ell^{p'}\bigl(\chi_{0},(W,\Lambda,X)
\bigr) \vert {\mathcal G}^n_{T} \bigr] \bigr]
\\
&= \sup_{n}
{\mathbb E}^{P'_{n}}\bigl[ 
\ell^{p'}\bigl(\chi_{0},(W,\Lambda,X)
\bigr) \bigr] < \infty.
\end{split}
\end{equation*}
Hence $\overline{P}'_n$ are relatively compact in $\P^p(\Omega)$. 

\textit{Second step.} Next, we check that $\overline{P}_n = \overline{P}'_n \circ (\xi,B,W,\mu,\Lambda,\hat{X}^n)^{-1}$ are relatively compact and have the same limits as $\overline{P}'_n$. This will follow essentially from the fact that $\hat{x}^n \rightarrow x$ as $n\rightarrow\infty$ \emph{uniformly} on compact subsets of $\C^d$. Indeed, for $t \in [t^n_i,t^n_{i+1}]$, the definition of $\hat{x}^n$ implies
\begin{align*}
|\hat{x}^n_t - x_t| &\le |\hat{x}^n_t - x_{t^n_{i-1}}| + |x_{t^n_{i-1}} - x_t| \le |x_{t^n_i} - x_{t^n_{i-1}}| + |x_{t^n_{i-1}} - x_t|.
\end{align*}
Since $|t - t^n_{i-1}| \le 2 \cdot 2^{-n}T$ for $t \in [t^n_i,t^n_{i+1}]$, we get
\[
\|\hat{x}^n - x\|_T \le 2\sup_{|t - s| \le 2^{1-n}T}|x_t-x_s|, \ \forall x \in \C^d.
\]
If $K \subset \C^d$ is compact, then it is equicontinuous by Arzel\`a-Ascoli, and the above implies $\sup_{x \in K}\|\hat{x}^n - x\|_T \rightarrow 0$. With this uniform convergence in hand, we check as follows that $\overline{P}_n$ has the same limiting behavior as $\overline{P}'_n$. By Prohorov's theorem, for each $\epsilon > 0$ there exists a compact set $K_\epsilon \subset \C^d$ such that $\E^{\overline{P}'_n}[\|X\|_T^p1_{\{X \in K_\epsilon^c\}}] \le \epsilon$ for each $n$. Using the obvious coupling and the fact that $\|\hat{x}^n\|_T \le \|x\|_T$ for all $x \in \C^d$, 
\begin{align*}
\ell_{\Omega,p}(\overline{P}_n,\overline{P}'_n) \le \E^{\overline{P}'_n}\left[\|X - \hat{X}^n\|_T^p\right]^{1/p} \le 2\epsilon^{1/p} + \sup_{x \in K_\epsilon}\|\hat{x}^n-x\|_T.
\end{align*}
Send $n\rightarrow\infty$ and then $\epsilon \downarrow 0$.

\textit{Third step.} It remains to check that any limit point $\overline{P}$ of $\overline{P}_n$ (and thus of $\overline{P}'_n$) satisfies the required properties.
Note first that $(B,\mu)$, $\xi$, and $W$ are independent under $\overline{P}$, since $\bar{\mu}^n(B)$ is $B$-measurable and since $B$, $\xi$, and $W$ are independent under $P_n$. Moreover, 
$(B,W)$ is an $(\F_{t}^{\xi,B,W,\mu,\Lambda,X})_{t \in [0,T]}$ 
Wiener process (of dimension 
$m_{0}+m$) under $\overline{P}$
since it is under $P_{n}$. 
In particular, $\rho:=\overline{P} \circ (\xi,B,W,\mu)^{-1} \in
\P^p_{c}[(\Omega_{0},\W_{\lambda}) \leadsto 
{\mathcal P}^p(\X)]$.
Since $(\bar{\mu}^n(B))^x = \mu^n(B)$, the canonical processes $(\xi,B,W,\mu,\Lambda,X)$ verify the state equation \ref{def:SDE} under $\overline{P}'_n$ for each $n$. Hence, it follows from the results of Kurtz and Protter \cite{kurtzprotter-weakconvergence} that \eqref{def:SDE} holds under the limiting measure $\overline{P}$ as well.

We now check that $\mu = \overline{P}(\left. (W,\Lambda,X) \in \cdot \ \right| \F^{B,\mu}_T)$. Let $\overline{P}_{n_k}$ be a subsequence converging to $\overline{P}$. Fix $n_0 \in \N$ and $S \in \G^{n_0}_T$, and let $\psi : \P(\X) \rightarrow \R$ and $\phi : \X \rightarrow \R$ be bounded and continuous. Then, since $\bar{\mu}^n = P_n(\left.(W,\Lambda,X) \in \cdot \ \right| \G^n_T)$ and $\G^{n_0}_T \subset \G^n_T$ for $n \ge n_0$, we compute (using Lemma \ref{le:componentwise} to handle the indicator function)
\begin{align}
\E^{\overline{P}}\left[1_S(B)\psi(\mu)\phi(W,\Lambda,X) \right] &= \lim_{k\rightarrow\infty}\E^{P_{n_k}}\left[1_S(B)\psi(\bar{\mu}^{n_k})\phi(W,\Lambda,X) \right] \nonumber \\
	&= \lim_{k\rightarrow\infty}\E^{P_{n_k}}\left[1_S(B)\psi(\bar{\mu}^{n_k})\int\phi\,d\bar{\mu}^{n_k} \right] \nonumber 
	= \E^{\overline{P}}\left[1_S(B)\psi(\mu)\int\phi\,d\mu \right]. \label{pf:exist2}
\end{align}
Conclude by noting that $\sigma\left(\bigcup_{n=1}^\infty\G^n_T\right) = \sigma(B)$.

\textit{Conclusion.} We have checked (1) and (3) in Definition \ref{def:mfgsolution}. 
Concerning (2), we 
already know from the beginning of the second step that 
$\rho=\overline{P} \circ (\xi,B,W,\mu)^{-1} \in
\P^p_{c}[(\Omega_{0},\W_{\lambda}) \leadsto 
{\mathcal P}^p(\X)]$. It thus remains to prove that $Q=P \circ (\xi,B,W,\mu,\Lambda)^{-1}$ is in 
${\mathcal A}(\rho)$ (that is the relaxed control is admissible). This follows from the more general Lemma
\ref{le:a.s.compatible}
right below. 
\end{proof}

The definition of MFG pre-solution requires that $\rho$ is compatible with $\W_\lambda$ (in the sense of point (1) 
in Subsection \ref{subse:setup}), but also the admissibility $P \in \RC\A(\rho)$ requires that $P \circ (\xi,B,W,\mu,\Lambda)^{-1}$ is compatible with $\rho$ (in the sense of (2) in Subsection \ref{subse:setup}). Because the latter compatibility does not behave well under limits, it will be crucial to have an alternative characterization of MFG pre-solutions which allows us to avoid directly checking admissibility. 
Namely, Lemma \ref{le:a.s.compatible} below shows that admissibility essentially follows automatically from the fixed point condition (3) of Definition \ref{def:mfgsolution}. 
In fact, Lemma \ref{le:a.s.compatible} is the main reason we work with the conditional law of $(W,\Lambda,X)$, and not just $X$.

\begin{lemma} \label{le:a.s.compatible}
Let $P \in \P^p(\Omega)$ such that $(B,W)$ is a Wiener process with respect to the filtration 
$({\mathcal F}_{t}^{\xi,B,W,\mu,\Lambda,X})_{t \in [0,T]}$ under $P$, and define 
$\rho := P \circ (\xi,B,W,\mu)^{-1}$. Suppose that 
(1) and (3) in Definition \ref{def:mfgsolution} are satisfied and that $P(X_{0}=\xi)=1$. 
Then, for $P \circ \mu^{-1}$-almost every $\nu \in \P^p(\X)$,
$(W_{t})_{t \in [0,T]}$ is an $({\mathcal F}_{t}^{W,\Lambda,X})_{t \in [0,T]}$
Wiener process under $\nu$. 
Moreover, $Q = P \circ (\xi,B,W,\mu,\Lambda)^{-1}$ is in ${\mathcal A}(\rho)$. 
\end{lemma}
\begin{proof}
\textit{First step.}
For $\nu \in \P(\X)$, let $\nu^w = \nu \circ W^{-1} \in \P(\C^m)$. To prove the first claim, let $\phi_1 : \P^p(\X) \rightarrow \R$ and $\phi_2 : \C^m \rightarrow \R$ be bounded and measurable. Then, since $P \circ W^{-1} = \W^m$ (with $\E$ denoting expectation under $P$),
\begin{align*}
\E\left[\phi_1(\mu)\right]\int_{\C^m}\phi_2\,d\W^m &= \E\left[\phi_1(\mu)\phi_2(W)\right] = \E\left[\phi_1(\mu)\int_{\C^m}\phi_2\,d\mu^w\right].
\end{align*}
The first equality follows from (1) in Definition \ref{def:mfgsolution} and the second one from 
(3) in Definition \ref{def:mfgsolution}. 

This holds for all $\phi_1$, and thus $\int\phi_2\,d\mu^w = \int\phi_2\,d\W^m$ a.s. This holds for all $\phi_2$, and thus $\mu^w = \W^m$ a.s. Now fix $t \in [0,T]$. Suppose $\phi_1 : \P^p(\X) \rightarrow \R$ is bounded and $\F^{\mu}_t$-measurable, $\phi_2 : \C^m \rightarrow \R$ is bounded and $\sigma(W_s - W_t : s \in [t,T])$-measurable, and $\phi_3 : \X \rightarrow \R$ is bounded and $\F^{W,\Lambda,X}_t$-measurable. Then $\phi_2(W)$ and $(\phi_1(\mu),\phi_3(W,\Lambda,X))$ are $P$-independent (since $W$ is a Wiener process with respect to 
$({\mathcal F}_{s}^{\xi,B,W,\mu,\Lambda})_{s\in [0,T]}$), and so
\begin{align*}
\E\left[\phi_1(\mu)\int_{\X}\phi_3\,d\mu\right]\int_{\C^m}\phi_2\,d\W^m	&= \E\left[\phi_1(\mu)\phi_3(W,\Lambda,X)\right]\int_{\C^m}\phi_2\,d\W^m \\
	&= \E\left[\phi_1(\mu)\phi_2(W)\phi_3(W,\Lambda,X)\right] \\
	&= \E\left[\phi_1(\mu)\int_{\X}\phi_2(w)\phi_3(w,q,x)\,\mu(dw,dq,dx)\right],
\end{align*}
the first and third equalities following from (3) in Definition \ref{def:mfgsolution}.
This holds for all $\phi_1$, and thus
\[
\int_{\C^m}\phi_2\,d\W^m\int_{\X}\phi_3(w,q,x)\mu(dw,dq,dx) = \int_{\X}\phi_2(w)\phi_3(w,q,x)\,\mu(dw,dq,dx), \ a.s.
\]
This holds for all $\phi_2$ and $\phi_3$, and thus it holds $P$-a.s. that $\sigma(W_s - W_t : s \in [0,T])$ and $\F^{W,\Lambda,X}_t$ are independent under almost every realization of $\mu$.

\textit{Second step.}
We now prove that $Q$ is in ${\mathcal A}(\rho)$ (notice that, by assumption, 
$(B,W)$ is a Wiener process with respect to the filtration 
$({\mathcal F}_{t}^{\xi,B,W,\mu})_{t \in [0,T]}$ under $P$). Fix $t \in [0,T]$. Let $\phi_t : \V \times \C^d \rightarrow \R$ be $\F^{\Lambda,X}_t$-measurable, let $\phi_t^w : \C^m \rightarrow \R$ be $\F^W_t$-measurable, let $\phi^w_{t+} : \C^m \rightarrow \R$ be $\sigma(W_s - W_t : s \in [t,T])$-measurable, let $\psi_T : \C^{m_0} \times \P^p(\X) \rightarrow \R$ be $\F^{B,\mu}_T$-measurable, and let $\psi_t : \C^{m_0} \times \P^p(\X) \rightarrow \R$ be $\F^{B,\mu}_t$-measurable. Assume all of these functions are bounded. We first compute
\begin{align*}
\E\left[\psi_T(B,\mu)\phi^w_{t+}(W)\psi_t(B,\mu)\phi^w_t(W)\right] &= \E\left[\psi_T(B,\mu)\psi_t(B,\mu)\right]\E\left[\phi^w_{t+}(W)\right]\E\left[\phi^w_t(W)\right] \\
	&= \E\left[\E\left[\left.\psi_T(B,\mu)\right|\F^{B,\mu}_t\right]\psi_t(B,\mu)\right]\E\left[\phi^w_{t+}(W)\right]\E\left[\phi^w_t(W)\right] \\
	&= \E\left[\E\left[\left.\psi_T(B,\mu)\right|\F^{B,\mu}_t\right]\phi^w_t(W)\psi_t(B,\mu)\right]\E\left[\phi^w_{t+}(W)\right],
\end{align*}
the first and third lines following from (1) in Definition \ref{def:mfgsolution}. 
This shows that
\begin{equation}
\label{eq:proof:compatibility:1}
\E\left[\left.\psi_T(B,\mu)\phi^w_{t+}(W)\right|\F^{B,W,\mu}_t\right] = \E\left[\left.\psi_T(B,\mu)\right|\F^{B,\mu}_t\right]\int_{\C^m}\phi^w_{t+}\,d\W^m.
\end{equation}
On the other hand, the first result of this Lemma implies that 
$(W_{t})_{t \in [0,T]}$ is an $({\mathcal F}_{t}^{W,\Lambda,X})_{t \in [0,T]}$ Wiener process under almost every realization of $\mu$, so that 
\begin{equation}
\label{eq:proof:compatibility:2}
\int_{\X}\phi_t(q,x)\phi^w_{t+}(w)\phi^w_t(w)\,\mu(dw,dq,dx) = 
\int_{\X}\phi_t(q,x)\phi^w_t(w)\,\mu(dw,dq,dx) \int_{\C^m}\phi^w_{t+}\,d\W^m, \ a.s.
\end{equation}
By \eqref{eq:filtration:mu}, note also that $\int_{\X}\phi_t(q,x)\phi^w_t(w)\,\mu(dw,dq,dx)$ is $\F^{B,\mu}_t$-measurable, since $\phi_t(\Lambda,X)\phi^w_t(W)$ is $\F^{W,\Lambda,X}_t$-measurable. Putting it together
(see right after the computations for more explanations):
\begin{align*}
\E&\left[\phi_t(\Lambda,X)\psi_T(B,\mu)\phi^w_{t+}(W)\psi_t(B,\mu)\phi^w_t(W)\right] \\
	&= \E\left[ \biggl( \int_{\X}\phi_t(q,x) \phi^w_{t+}(w)\phi^w_t(w)\,\mu(dw,dq,dx) \biggr)\,\psi_T(B,\mu)\psi_t(B,\mu)\right] \\
	&= \E\left[ \biggl( \int_{\X}\phi_t(q,x)\phi^w_t(w)\,\mu(dw,dq,dx) \biggr) \,\psi_T(B,\mu)\psi_t(B,\mu)\right]\int\phi^w_{t+}\,d\W^m \\
	&= \E\left[\biggl( \int_{\X}\phi_t(q,x)\phi^w_t(w)\,\mu(dw,dq,dx) \biggr)\,\E\left[\left.\psi_T(B,\mu)\right|\F^{B,\mu}_t\right]\psi_t(B,\mu)\right]\int_{\C^m}\phi^w_{t+}\,d\W^m \\
	&= \E\left[\phi_t(\Lambda,X)\phi^w_t(W)\E\left[\left.\psi_T(B,\mu)\right|\F^{B,\mu}_t\right]\psi_t(B,\mu)\right]\int_{\C^m}\phi^w_{t+}\,d\W^m \\
	&= \E\left[\E\left[\left.\phi_t(\Lambda,X)\right|\F^{B,W,\mu}_t\right]\E\left[\left.\psi_T(B,\mu)\right|\F^{B,\mu}_t\right]\psi_t(B,\mu)\phi^w_t(W)\right]\int_{\C^m}\phi^w_{t+}\,d\W^m \\
	&= \E\left[\E\left[\left.\phi_t(\Lambda,X)\right|\F^{B,W,\mu}_t\right]\E\left[\left.\psi_T(B,\mu)\phi^w_{t+}(W)\right|\F^{B,W,\mu}_t\right]\psi_t(B,\mu)\phi^w_t(W)\right],
\end{align*}
the first equality following from (3) in Definition \ref{def:mfgsolution}, the second one
from \eqref{eq:proof:compatibility:2}, the third one from the fact that 
$\int\phi_t(q,x)\phi^w_t(w)\,\mu(dw,dq,dx)$ is $\F^{B,\mu}_t$-measurable, 
the fourth one from (3) in Definition \ref{def:mfgsolution} and the last one from
\eqref{eq:proof:compatibility:1}.

Replacing $\phi_t^w(W)$ with $\phi_t^w(W)\psi_t^w(W)$, where both $\phi_t^w$ and $\psi_t^w$ are $\F^W_t$-measurable, we see that
\begin{align*}
\E&\left[\left.\phi_t(\Lambda,X)\psi_T(B,\mu)\phi^w_{t+}(W)\phi_t^w(W)\right|\F^{B,W,\mu}_t\right] \\
	&= \E\left[\left.\phi_t(\Lambda,X)\right|\F^{B,W,\mu}_t\right]\E\left[\left.\psi_T(B,\mu)\phi^w_{t+}(W)\phi_t^w(W)\right|\F^{B,W,\mu}_t\right].
\end{align*}
Since random variables of the form $\phi_t^w(W)\phi_{t+}^w(W)$ generate $\F^W_T$, this shows that $\F^{\Lambda,X}_t$ is conditionally independent of $\F^{B,W,\mu}_T$ given $\F^{B,W,\mu}_t$.

\textit{Last step.} It now remains to prove that 
$\F^{\Lambda}_t$ is conditionally independent of $\F^{\xi,B,W,\mu}_T$ given $\F^{\xi,B,W,\mu}_t$, which is slightly different from the result of the previous step. To do so, we use the fact that $P(X_{0}=\xi)=1$. Let $\phi_t : \V \rightarrow \R$ be $\F^{\Lambda}_t$-measurable, 
$\psi_t : \C^{m_{0}} \times \C^m \times \P^p(\X)
 \rightarrow \R$ be $\F^{B,W,\mu}_t$-measurable, 
 $\psi_T : \C^{m_{0}} \times \C^m \times \P^p(\X)
 \rightarrow \R$ be $\F^{B,W,\mu}_T$-measurable
 and $\zeta_{0} : \R \rightarrow \R$ be Borel measurable. Assume all of these functions are bounded. 
 From the previous step, we deduce that
 \begin{equation*}
 \begin{split}
&\E\bigl[\phi_t(\Lambda)\psi_T(B,W,\mu)\psi_t(B,W,\mu) \zeta_{0}(\xi)\bigr] 
\\
&= \E\bigl[\phi_t(\Lambda)\zeta_{0}(X_{0})\psi_T(B,W,\mu)\psi_t(B,W,\mu) \bigr] 
\\
&= \E\left[\E\left[\left. \phi_t(\Lambda)\zeta_{0}(X_{0})
\right|\F^{B,W,\mu}_t
\right]
\E\left[\left.
\psi_T(B,W,\mu)
\right|\F^{B,W,\mu}_t
\right]
\psi_t(B,W,\mu) 
\right] 
\\
&= \E\left[ \phi_t(\Lambda)\zeta_{0}(X_{0})
\E\left[\left.
\psi_T(B,W,\mu)
\right|\F^{B,W,\mu}_t
\right]
\psi_t(B,W,\mu) 
\right] 
\\
&= \E\left[ \phi_t(\Lambda)
\E\left[\left.
\psi_T(B,W,\mu)
\right|\F^{B,W,\mu}_t
\right]
\psi_t(B,W,\mu) \zeta_{0}(\xi)
\right],
\end{split}
 \end{equation*}
the second equality following from the conditional independence of ${\mathcal F}_{t}^{\Lambda,X}$
and ${\mathcal F}_{T}^{B,W,\mu}$ given ${\mathcal F}_{t}^{B,W,\mu}$. In order to complete the proof, notice 
that
$\E[
\psi_T(B,W,\mu)
|\F^{B,W,\mu}_t
]=\E[
\psi_T(B,W,\mu)
|\F^{\xi,B,W,\mu}_t
]$ since
$\xi$ and $(B,W,\mu)$ are independent under $P$ (see (1) in Definition \ref{def:mfgsolution}). 
Therefore, for another bounded Borel measurable function $\zeta_{0}' : \R \rightarrow \R$, we get
 \begin{equation*}
 \begin{split}
&\E\left[\phi_t(\Lambda)\psi_T(B,W,\mu)\psi_t(B,W,\mu) \zeta_{0}(\xi)\zeta_{0}'(\xi)\right] 
\\
&= \E\left[ \phi_t(\Lambda)
\E\left[\left.
\psi_T(B,W,\mu)
\right|\F^{\xi,B,W,\mu}_t
\right]
\psi_t(B,W,\mu) \zeta_{0}(\xi)\zeta_{0}'(\xi)
\right] 
\\
&=
\E\left[ \phi_t(\Lambda)
\E\left[\left.
\zeta_{0}'(\xi) \psi_T(B,W,\mu)
\right|\F^{\xi,B,W,\mu}_t
\right]
\psi_t(B,W,\mu) \zeta_{0}(\xi)
\right],
\end{split}
 \end{equation*}
which proves that 
${\mathcal F}_{t}^{\Lambda}$
and ${\mathcal F}_{T}^{\xi,B,W,\mu}$ are conditionally independent 
given ${\mathcal F}_{t}^{\xi,B,W,\mu}$.
\end{proof}

\subsection{Existence of a MFG solution under Assumption \ref{assumption:B}}

The goal of this section is to prove that the limit points constructed in the previous paragraph are not
only MFG pre-solutions but are weak MFG solutions:
 	
\begin{theorem} \label{th:mainexistence-bounded}
Assume that \ref{assumption:B} holds
and keep the notation of Lemma 
\ref{le:limitpresolution}. Then, 
every limit point is a weak MFG solution with weak control.
\end{theorem}

Generally speaking, it remains to show that any limit point of the sequence of Lemma \ref{le:limitpresolution} is optimal for the corresponding control problem:

\begin{lemma} \label{le:consistence:definitions}
Assume that a MFG pre-solution $P$ 
satisfies $P \in \RC\A^\star(\rho)$, 
with $\rho$ given by $\rho := P \circ (\xi,B,W,\mu)^{-1}$,
then 
$(\Omega,(\F^{\xi,B,W,\mu,\Lambda,X}_t)_{t \in [0,T]},P,B,W,\mu,\Lambda,X)$ is a weak MFG solution with weak control.
\end{lemma}

\begin{proof} The proof is quite straightforward since the pre-solution properties of $P$ guarantee
that the canonical process under $P$ satisfy (1--4) and (6) in Definition \ref{def:mfg:weak:weak}.
Condition (3) of Definition \ref{def:mfg:weak:weak} uses $\sigma(\Lambda_s : s \le t)$,
whereas the notion of compatibility in the definition of MFG pre-solutions uses the canonical filtration $\F^\Lambda_t$ defined by \eqref{def:flambda},
but this is no cause for concern in light of the discussion following \eqref{def:flambda}.
The additional condition $P \in \RC\A^\star(\rho)$ permits to verify (5) 
in Definition \ref{def:mfg:weak:weak}
by transferring any 
$(\Omega',(\F'_t)_{t \in [0,T]},P',B',W',\nu,\Lambda',X')$
as in (5) onto the canonical space. 
\end{proof}

\subsubsection{Strategy}
In order to check the condition $P \in \RC\A^\star(\rho)$ in Lemma 
\ref{le:consistence:definitions}, the idea is to approximate any alternative MFG control by a sequence of particularly well-behaved controls for the discretized game. The crucial technical device is Lemma \ref{le:adapteddensity}, but we defer its proof to the appendix. The following definition is rather specific to the setting of compact control space $A$ (we assume that
 \ref{assumption:B} holds throughout the section), but it will return in a more general form in Section \ref{se:unboundedcoefficients}:

\begin{definition} \label{def:adapted}
A function $\phi : \Omega_0 \times \P^p(\X) \rightarrow \V$ is said to be \emph{adapted} if $\phi^{-1}(C) \in \F^{\xi,B,W,\mu}_t$ for each $C \in \F^\Lambda_t$ and $t \in [0,T]$. For $\rho \in \P^p_{c}[(\Omega_{0},\W_{\lambda}) \leadsto 
{\mathcal P}^p(\X)]$ (that is $(\xi,B,W)$ has law $\W_{\lambda}$ under $\rho$ and  
$B$ and $W$ are independent $(\F^{\xi,B,W,\mu}_t)_{t \in [0,T]}$-Wiener processes
under $\rho$), let $\A_a(\rho)$ denote the set of measures of the form
\begin{align}
\rho(d\omega,d\nu)\delta_{\phi(\omega,\nu)}(dq) = \rho \circ (\xi,B,W,\mu,\phi(\xi,B,W,\mu))^{-1} \label{def:adaptedform}
\end{align}
where $\phi$ is adapted and \emph{continuous}.
\end{definition}

\begin{lemma} \label{le:adapteddensity}
For each $\rho \in \P^p_c[(\Omega_{0},\W_{\lambda}) \leadsto 
{\mathcal P}^p(\X)]$, $\A_a(\rho)$ is a dense subset of $\A(\rho)$.
\end{lemma}

We also need continuity lemmas, the proofs of which are deferred to the end of the subsection. Notice that these lemmas are stated under assumption \ref{assumption:A}, not \ref{assumption:B}.

\begin{lemma} \label{le:rccontinuous}
Suppose a set $K \subset \P^p_{c}[(\Omega_{0},\W_{\lambda}) \leadsto \P^p(\X) \times \V]$ satisfies
\[
\sup_{P \in K}\E^P\left[\int_{\C^d}\|x\|_T^{p'}\mu^x(dx) + \int_0^T\int_A|a|^{p'}\Lambda_t(da)dt\right] < \infty.
\]
Under assumption \ref{assumption:A}, the map $\RC : K  \rightarrow \P^p(\Omega)$ 
(that maps $Q \in K$ to the law of the solution $(\xi,B,W,\mu,\Lambda,X)$ of \eqref{def:SDE} when $(\xi,B,W,\mu,\Lambda)$ has law $Q$) is continuous.
\end{lemma}

\begin{lemma} \label{le:jcontinuous}
Under assumption \ref{assumption:A}, 
the expected reward functional $J : \P^p(\Omega) \rightarrow \R$ given by \eqref{eq:cost:functional}
is upper semicontinuous. 
If also $A$ is compact, then $J$ is continuous.
\end{lemma}

\begin{lemma} \label{le:pinuniform}
Define $\Pi_n : \P(\Omega) \rightarrow \P(\Omega)$
 by
\[
\Pi_n(P) := P \circ \left(\xi,B,W,\mu,\Lambda,\hat{X}^n \right)^{-1}.
\]
(See \eqref{eq:interpolation}
for the definition of $\hat{X}^n$.)
If $P_n \rightarrow P$ in $\P^p(\Omega)$, then $\Pi_n(P_n) \rightarrow P$ in $\P^p(\Omega)$. 
\end{lemma}

\subsubsection{Proof of Theorem \ref{th:mainexistence-bounded}}
Let $\mu^n$, $\bar{\mu}^n$, $P_n$, and $\overline{P}_n$ be as in Lemma \ref{le:limitpresolution}, and let $\overline{P}$ denote any limit point. Relabel the subsequence, and assume that $\overline{P}_n$ itself converges. Let $\rho:= \overline{P} \circ (\xi,B,W,\mu)^{-1}$. By Lemma \ref{le:limitpresolution},  
$\rho \in \P^p_{c}[(\Omega_{0},\W_{\lambda}) \leadsto 
{\mathcal P}^p(\X)]$ and $\overline{P} \in \RC\A(\rho)$ is a MFG pre-solution, and it remains only to show that $\overline{P}$ is optimal, or $\overline{P} \in \RC\A^\star(\rho)$. Fix $P^\star \in \RC\A(\rho)$ arbitrarily with $J(P^\star) > -\infty$. Let
\begin{align*}
Q^\star &:= P^\star \circ (\xi,B,W,\mu,\Lambda)^{-1}.
\end{align*}
By Lemma \ref{le:adapteddensity}, we may find a sequence of $(\F^{\xi,B,W,\mu}_t)_{t \in [0,T]}$-adapted continuous functions $\phi_k : \Omega_0 \times \P^p(\X) \rightarrow \V$ such that 
\begin{align*}
Q^\star = \lim_{k\rightarrow\infty}Q^k, \text{ where } \quad Q^k := \rho \circ (\xi,B,W,\mu,\phi_k(\xi,B,W,\mu))^{-1}.
\end{align*}
Define $Q^k_n \in \A_f$ 
(see Paragraph 
\ref{subsubse:control} for the definition of $\A_{f}$)
by
\[
Q^k_n := \W_\lambda \circ \bigl(\xi,B,W,\phi_k(\xi,B,W,\bar{\mu}^n(B))\bigr)^{-1}.
\]
Note that $\overline{P}_n \rightarrow \overline{P}$ implies
\[
\rho = 
\lim_{n \rightarrow \infty} \overline{P}_{n}
\circ (\xi,B,W,\mu)^{-1}
=\lim_{n\rightarrow\infty}\W_\lambda \circ \bigl(\xi,B,W,\bar{\mu}^n(B)\bigr)^{-1},
\]
where the second equality comes from the definition of $\overline{P}_{n}$ in Lemma \ref{le:limitpresolution}.
Since $\phi_k$ is continuous with respect to $\mu$,
we deduce from Lemma \ref{le:componentwise} (that permits to handle the possible dicontinuity
of $\phi_{k}$ in the other variables):
\begin{align}
\lim_{n\rightarrow\infty}Q^k_n \circ (\xi,B,W,\bar{\mu}^n(B),\Lambda)^{-1} &= \lim_{n\rightarrow\infty}\W_\lambda \circ (\xi,B,W,\bar{\mu}^n(B),\phi_k(\xi,B,W,\bar{\mu}^n(B)))^{-1}
	= Q^k. \label{pf:limitsolution0}
\end{align}
Now let $P^k_n := \RC_f^n(\mu^n,Q^k_n)$. Since $P_n$ is optimal for $J_f(\mu^n,\cdot)$, 
\begin{align*}
J_f(\mu^n,P^k_n) \le J_f(\mu^n,P_n).
\end{align*}
Since $A$ is compact, Lemma \ref{le:jcontinuous} assures us that $J$ is continuous, and so
\begin{align*}
\lim_{n\rightarrow\infty}J_f(\mu^n,P_n) &= 
\lim_{n\rightarrow\infty}\E^{P_n}\left[\Gamma(\mu^n(B),\Lambda,X) \right] = 
\lim_{n\rightarrow\infty}J (\overline{P}_{n}) = J(\overline{P}),
\end{align*}
where the second equality follows simply from the definition of $J$.
We will complete the proof by showing that, on the other hand,
\begin{align}
J(P^\star) = \lim_{k\rightarrow\infty}\lim_{n\rightarrow\infty}J_f(\mu^n,P^k_n), \label{pf:limitsolution1}
\end{align}
and both limits exist; indeed, this implies $J(\overline{P}) \ge J(P^\star)$, completing the proof since $P^\star \in \RC\A(\rho)$ was arbitrary. Define $\Pi_n$ as in Lemma \ref{le:pinuniform}. 
The trick is to notice (just applying the basic definition of the different objects) that 
\begin{equation*}
\begin{split}
&P_{n}^k \circ (\xi,B,W,\bar{\mu}^n(B),\Lambda,X)^{-1} = \Pi_n\left(\RC\left(Q_{n}^k 
\circ (\xi,B,W,\bar{\mu}^n(B),\Lambda)^{-1}\right)\right), 
\\
&J_f(\mu^n,P_{n}^k) = J\left(P_{n}^k \circ (\xi,B,W,\bar{\mu}^n(B),\Lambda,X)^{-1}\right).
\end{split}
\end{equation*}
Now note that $P^k_n \circ (\mu^x)^{-1} = \W_\lambda \circ (\mu^n)^{-1} = \overline{P}_n \circ (\mu^x)^{-1}$, and thus by \eqref{pf:exist001} we have
\[
\sup_n\E^{P^k_n}\int_{\C^d}\|x\|_T^{p'}\mu^x(dx) < \infty.
\]
Since also $A$ is compact,
we may apply Lemma \ref{le:rccontinuous} (continuity of $\RC$), along with Lemma \ref{le:pinuniform} and \eqref{pf:limitsolution0}, to get
\begin{align*}
\lim_{n\rightarrow\infty}P^k_n \circ (\xi,B,W,\bar{\mu}^n(B),\Lambda,X)^{-1} &= \lim_{n\rightarrow\infty}\Pi_n\left(\RC\left(Q^k_n \circ (\xi,B,W,\bar{\mu}^n(B),\Lambda)^{-1}\right) \right)
= \RC(Q^k).
\end{align*}
Thus, again using continuity of $\RC$,
\begin{align*}
\lim_{k\rightarrow\infty}\lim_{n\rightarrow\infty}P^k_n \circ (\xi,B,W,\bar{\mu}^n(B),\Lambda,X)^{-1} = \RC(Q^\star) = P^\star.
\end{align*}
Finally, \eqref{pf:limitsolution1} follows from continuity of $J$.

\subsubsection{Proof of Lemma \ref{le:rccontinuous}}
Let $Q_n \rightarrow Q$ in $K$. Note that $\RC(Q_n) \circ (X_0,B,W,\mu,\Lambda)^{-1} = Q_n$ are relatively compact in $\P^p(\Omega_0 \times \P^p(\X) \times \V)$. It can be shown using Aldous' criterion (see Proposition \ref{pr:itocompact}) that this implies that $\RC(Q_n) \circ X^{-1}$ are relatively compact in $\P^p(\C^d)$, and thus $\RC(Q_n)$ are relatively compact in $\P^p(\Omega)$. Let $P$ be any limit point, so $\RC(Q_{n_k}) \rightarrow P$ for some $n_k$. Then
\[
P \circ (\xi,B,W,\mu,\Lambda)^{-1} = \lim_{k\rightarrow\infty}\RC(Q_{n_k}) \circ (\xi,B,W,\mu,\Lambda)^{-1} = \lim_{k\rightarrow\infty}Q_k = Q.
\]
It follows from the results of Kurtz and Protter \cite{kurtzprotter-weakconvergence} that the canonical processes verify the SDE \eqref{def:SDE} under $P$. Hence, $P = \RC(Q)$.

\subsubsection{Proof of Lemma \ref{le:jcontinuous}}
Since $f$ and $g$ are continuous in $(x,\mu,a)$, the upper bounds of $f$ and $g$ (which grow in order $p$ in $(x,\mu)$) along with Lemma \ref{le:usc} imply both that $\Gamma$ is upper semicontinuous and then also that $J$ is upper semicontinuous from $\P^p(\Omega)$ to $\R$. 
If $A$ is compact, then the same $p$-order upper bounds of $f$ and $g$ hold for the negative parts as well, and the second part of Lemma \ref{le:usc} provides the claimed continuity.
\hfill \qedsymbol

\subsubsection{Proof of Lemma \ref{le:pinuniform}}
This was essentially already proven in the second step of the proof of Lemma \ref{le:limitpresolution}.
Note that
\[
\ell_{\Omega,p}\left(\Pi_n(P_{n}),P\right)  \le \ell_{\Omega,p}\left(P_{n},P\right)
 + \ell_{\Omega,p}\left(P_n,\Pi_{n}(P_n)\right).
 \]
The first term tends to zero by assumption. Fix $\epsilon > 0$. Since $\{P_n : n \ge 1\}$ is relatively compact in $\P^p(\Omega)$, by Prohorov's theorem there exists a compact set $K \subset \C^d$ such that $\E^{P_n}[\|X\|_T^p1_{\{X \notin K\}}] \le \epsilon$ for all $n$. Use the obvious coupling and the fact that $\|\hat{x}^n\|_T \le \|x\|_T$ for all $x \in \C^d$ to get
\begin{align*}
\ell_{\Omega,p}\left(P_n,\Pi_{n}(P_n)\right) &\le \E^{P_n}\left[\|X- \hat{X}^n\|_T^p\right]^{1/p} \le (\epsilon 2^{p-1})^{1/p} + \sup_{x \in K}\|x - \hat{x}^n\|_T.
\end{align*}
We saw in the second step of the proof of Lemma \ref{le:limitpresolution} that $\hat{x}^n \rightarrow x$ as $n\rightarrow\infty$ \emph{uniformly} on compact subsets of $\C^d$, and so the proof is complete.

\subsection{Unbounded coefficients} \label{se:unboundedcoefficients}
Finally, with existence in hand for bounded state coefficients ($b$, $\sigma$, $\sigma_0$) and compact control space $A$, we turn to the general case. The goal is thus to complete the proof of Theorem \ref{th:existence}
under \ref{assumption:A} instead of \ref{assumption:B}.

The idea of the proof is to approximate the data $(b,\sigma,\sigma_0,A)$ by data satisfying Assumption \ref{assumption:B}. Let $(b^n,\sigma^n,\sigma^n_0)$ denote the projection of $(b,\sigma,\sigma_0)$ into the ball centered at the origin with radius $n$ in $\R^d \times \R^{d \times m} \times \R^{d \times m_0}$, respectively. Let $A_n$ denote the intersection of $A$ with the ball centered at the origin with radius $n$. For sufficiently large $n_0$, $A_n$ is nonempty and compact for all $n \ge n_0$, and thus we will always assume $n \ge n_0$ in what follows. Note that the data $(b^n,\sigma^n,\sigma^n_0,f,g,A_n)$ satisfy Assumption \ref{assumption:B}. Moreover, (A.4) and (A.5) hold for each $n$ \emph{with the same constants} $c_1,c_2,c_3$; this implies that Lemma \ref{le:stateestimate} holds with the same constant $c_4$ for each set of data, i.e. independent of $n$.

Define $\V_n$ as before in terms of $A_n$, but now view it as a subset of $\V$. That is, 
$\V_n := \{q \in \V : q([0,T] \times A_n^c)=0\}$.
Naturally, define $\A_n(\rho)$ to be the set of admissible controls with values in $A_n$:
\begin{equation}
\label{eq:Anrho}
\A_n(\rho) := \left\{Q \in \A(\rho) : Q(\Lambda \in \V_n) = 1\right\}.
\end{equation}
Finally, define $\RC_n(Q)$ to be the unique element $P$ of $\P(\Omega)$ such that $P \circ (\xi,B,W,\mu,\Lambda)^{-1} = Q$ and the canonical processes verify the SDE
\begin{equation}
\label{eq:SDE:truncated}
X_t = X_0 + \int_0^tds\int_A\Lambda_s(da)b^n(s,X_s,\mu^x_s,a) + \int_0^t\sigma^n(s,X_s,\mu^x_s)dW_s + \int_0^t\sigma^n_0(s,X_s,\mu^x_s)dB_s.
\end{equation}
Define naturally
\[
\RC_n\A^\star_n(\rho) := \arg\max_{P \in \RC_n\A_n(\rho)}J(P).
\]
By Theorem \ref{th:mainexistence-bounded}, there exists for each $n$ a MFG solution corresponding to the $n^\text{th}$ truncation of the data. In the present notation, this means there exist $\rho_n \in 
\P^p_{c}[(\Omega_{0},\W_{\lambda}) \leadsto 
{\mathcal P}^p(\X)]$ and $P_n \in \RC_n\A^\star_n(\rho_n)$ such that 
\begin{align}
\mu = P_n\left((W,\Lambda,X) \in \cdot \ | \ \F^{B,\mu}_T\right), \ P^n-a.s. \label{def:nfixedpoint}
\end{align}
Once again, the strategy of the proof is to show first that $P_n$ are relatively compact and then that each limit point is a MFG solution.

\subsubsection{Relative compactness} We start with
\begin{lemma} \label{le:moments}
The measures $P_n$ are relatively compact in $\P^p(\Omega)$. Moreover,
\begin{align}
\sup_n\E^{P_n}\int_0^T \int_{A} |a|^{p'} \Lambda_t(da) dt < \infty, \quad
\E^{P_n} \int_{\C^d} \|x\|_T^{p'} \mu^x(dx) = \sup_n\E^{P_n}\|X\|_T^{p'} < \infty. \label{pf:unbounded3}
\end{align}
\end{lemma}
\begin{proof}
Noting that the coefficients $(b^n,\sigma^n,\sigma^n_0)$ satisfy (A.1-5) with the same constants (independent of $n$), Lemma \ref{le:stateestimate} and \eqref{def:nfixedpoint} imply
\begin{align}
\E^{P_n} \int_{\C^d} \|x\|_T^p \mu^x(dx) = \E^{P_n}\|X\|^p_T \le c_4\biggl(1 + \E^{P_n}\int_0^T
\int_{A}|a|^p\Lambda_t(da) dt\biggr). \label{pf:unbounded1}
\end{align}
Fix $a_0 \in A_{n_0}$. Let $R_n$ denote the unique element of $\RC_n\A_n(\rho_n)$ satisfying $R_n(\Lambda_t = \delta_{a_0} \text{ for a.e. } t) = 1$. That is $R_n$ is the law of the solution of the state equation arising from the constant control equal to $a_0$, in the $n^\text{th}$ truncation. The first part of Lemma \ref{le:stateestimate} implies
\begin{align}
\E^{R_n}\|X\|^p_T &\le c_4\biggl(1 + \E^{R_n} \int_{\C^d}\|y\|_T^p \mu^x(dy) + T|a_0|^{p}\biggr). \label{pf:unbounded2}
\end{align}
Noting that $R_n \circ \mu^{-1} = P_n \circ \mu^{-1}$, we combine \eqref{pf:unbounded2} with \eqref{pf:unbounded1} to get
\begin{align}
\E^{R_n}\|X\|^p_T &\le C_0\biggl(1 + \E^{P_n}\int_0^T \int_{A}
|a|^p \Lambda_t(da) dt \biggr), \label{pf:unbounded2.5}
\end{align}
where $C_0 > 0$ depends only on $c_4$, $T$, and $|a_0|^p$. Use the optimality of $P_n$, the lower bounds on $f$ and $g$, and then \eqref{pf:unbounded1} and \eqref{pf:unbounded2.5} to get
\begin{align}
J(P_n) &\ge J(R_n) \ge -c_2(T+1)
\biggl(1 + \E^{R_n}\|X\|_T^p + \E^{R_n} \int_{\C^d}\|y\|_T^p \mu^x(dy) + |a_0|^{p'}\biggr) 
\nonumber \\
	&\ge -C_1\biggl(1 + \E^{P_n}\int_0^T \int_{A}|a|^p \Lambda_t(da) dt\biggr), \label{pf:unbounded2-1}
\end{align}
where $C_1 > 0$ depends only on $c_2$, $c_4$, $T$, $|a_0|^{p'}$, and $C_0$. On the other hand, we may use the upper bounds on $f$ and $g$ along with \eqref{pf:unbounded1} to get
\begin{align}
J(P_n) &\le c_2(T+1)\biggl(1 + \E^{P_n}\|X\|_T^p + \E^{P_n} \int_{\C^d} \|y\|^p_T \mu^x(dy) \biggr) - c_3\E^{P_n}\int_0^T \int_{A}|a|^{p'} \Lambda_t(da) dt \nonumber \\
	&\le C_2\biggl(1 + \E^{P_n}\int_0^T \int_{A} |a|^p \Lambda_t(da) dt\biggr) - c_3\E^{P_n}
	\int_0^T \int_{A}|a|^{p'}\Lambda_t(da) dt, \label{pf:unbounded2-2}
\end{align}
where $C_2 > 0$ depends only on $c_2$, $c_3$, $c_4$, and $T$. Combining \eqref{pf:unbounded2-1} and \eqref{pf:unbounded2-2} and rearranging, we find two constants, $\kappa_1 \in \R$ and $\kappa_2 > 0$, such that
\[
\E^{P_n}\int_0^T \int_{A} \bigl( |a|^{p'} + \kappa_{1} |a|^p \bigr) \Lambda_{t}(da) dt\le \kappa_2.
\]
(Note that $\E^{P_n}\int_0^T \int_{A}|a|^p \Lambda_{t}(da) dt < \infty$ for each $n$.) These constants are independent of $n$, and
the first bound in \eqref{pf:unbounded3} follows from the fact that $p'>p$. Combined with Lemma \ref{le:stateestimate}, this implies the second bound in \eqref{pf:unbounded3}.

To show that $P_n$ are relatively compact, we check that each of the sets of marginals is relatively compact; see Lemma \ref{le:productrelcompactness}. Compactness of $P_n \circ (B,W)^{-1}$ is obvious. Moreover, 
by \eqref{pf:unbounded3},
\[
\sup_n\E^{P_n}\biggl[\|W\|^{p'}_T + \int_0^T \int_{A}|a|^{p'} \Lambda_{t}(da) dt + \|X\|^{p'}_T\biggr] < \infty.
\]
Aldous' criterion (Proposition \ref{pr:itocompact}) shows that $P_n \circ (\Lambda,X)^{-1}$ are relatively compact. The mean measures of $P_n \circ \mu^{-1}$ are $P_n \circ (W,\Lambda,X)^{-1}$, which we have shown are relatively compact. Hence, by Proposition \ref{pr:ptight}, $P_n \circ \mu^{-1}$ are relatively compact in $\P^p(\P^p(\X))$.
\end{proof}

\subsubsection{Limit points}
Now that we know $P_n$ are relatively compact, we may fix $P \in \P^p(\Omega)$ and a subsequence $n_k$ such that $P_{n_k} \rightarrow P$ in $\P^p(\Omega)$. Define $\rho := P \circ (\xi,B,W,\mu)^{-1}$, and note that $\rho_{n_k} \rightarrow \rho$.

\begin{lemma} \label{le:limitpoint}
The limit point $P$ is a MFG pre-solution and satisfies
\begin{align*}
\E^P\int_0^T \int_{A}|a|^{p'} \Lambda_t(da) dt \le \liminf_{k\rightarrow\infty}\E^{P_{n_k}}
\int_0^T \int_{A}|a|^{p'} \Lambda_t(da) dt < \infty.
\end{align*}
\end{lemma}
\begin{proof}
Fatou's lemma and the first bound in \eqref{pf:unbounded3} imply the stated inequality. 
We now check (1), (2) and (3) in Definition \ref{def:mfgsolution}. 
Since $(B,\mu)$, $\xi$ and $W$ are independent under $P_n$, the same is true under the limit $P$, 
which gives (1). We now check (2). 
The strategy is to apply Lemma 
\ref{le:a.s.compatible}. By passage to the limit, it is well checked that 
$(B,W)$ is a Wiener process with respect to the filtration 
$({\mathcal F}_{t}^{\xi,B,W,\mu,\Lambda,X})_{t \in [0,T]}$
under $P$ (which implies in particular that $\rho \in \P^p_{c}[(\Omega_{0} \times {\mathcal P}^p(\X),
\rho) \leadsto \V]$). Moreover, it must also hold $P(X_{0}=\xi)=1$. Therefore, in order to 
prove (2), it sufficient to check (3) and to check that the state equation 
\eqref{def:SDE} is satisfied under $P$.

We first check (3). If $\psi : \C^{m_0} \times \P^p(\X) \rightarrow \R$ and $\phi : \X \rightarrow \R$ are bounded and continuous, we have
\begin{align*}
\E^P\left[\psi(B,\mu)\phi(W,\Lambda,X)\right] &= \lim_{k\rightarrow\infty}\E^{P_{n_k}}\left[\psi(B,\mu)\phi(W,\Lambda,X)\right] \\
	&= \lim_{k\rightarrow\infty}\E^{P_{n_k}}\left[\psi(B,\mu)\int_\X\phi\,d\mu\right] = \E^P\left[\psi(B,\mu)\int_\X\phi\,d\mu\right].
\end{align*}
Thus $\mu = P((W,\Lambda,X) \in \cdot \ | \ \F^{B,\mu}_T)$ a.s., which gives (3) in 
Definition \ref{def:mfgsolution}. Now, to check that the state equation is satisfied, define processes $(Z^q_t)_{t \in [0,T]}$ on $\Omega$ by
\[
Z^q_t := 1 + |X_t|^q + \left(\int_{\R^d}|y|^p\mu^x_t(dy)\right)^{q/p}, \ q > 0.
\]
Using the growth assumptions on $b$ of (A.4), note that $b(t,y,\nu,a) \neq b^n(t,y,\nu,a)$ if and only if 
\begin{align}
n &< |b(t,y,\nu,a)| \le c_1\biggl(1 + |y| + \left(\int_{\R^d}|z|^p\nu(dz)\right)^{1/p} + |a|\biggr), \label{eq:bdiffbn}
\end{align}
so that
\begin{align*}
\E^{P_n}\left|\int_0^tds\int_A\Lambda_s(da)(b^n-b)(s,X_s,\mu^x_s,a)\right| &\le 2c_1\E^{P_n}\int_0^tds\int_A\Lambda_s(da)\left(Z^1_s + |a|\right)1_{\{c_1(Z^1_s + |a|) > n\}}.
\end{align*}
By Lemma \ref{le:moments}, this tends to zero as $n \rightarrow\infty$. Similarly,
 $\sigma(t,y,\nu) \neq \sigma^n(t,y,\nu)$ if and only if 
\begin{align}
n^2 &< |\sigma(t,y,\nu)|^2 \le c_1
\biggl(1 + |y|^{p_\sigma} + \left(\int_{\R^d}|z|^p\nu(dz)\right)^{p_\sigma/p}\biggr),
\label{eq:sigmadiffsigman}
\end{align}
so that the Burkholder-Davis-Gundy inequality yields
\begin{align*}
\E^{P_n}\left|\int_0^t (\sigma^n-\sigma)(s,X_s,\mu^x_s)dW_s\right| &\le 
2 ( c_1 )^{1/2}
\E^{P_n}\biggl[\left(\int_0^t Z_s^{p_\sigma}1_{\{c_1Z_s^{p_\sigma} > n^2\}} ds\right)^{1/2}\biggr].
\end{align*}
This tends to zero as well, as does
$\E^{P_n}
|\int_0^t (\sigma^n_0-\sigma_0)(s,X_s,\mu^x_s)dB_s|$. It follows that 
\begin{align*}
0 = \lim_{n\rightarrow\infty}\E^{P_n}\sup_{0 \le t \le T}&\left|X_t - X_0 - \int_0^tds\int_A\Lambda_s(da)b(s,X_s,\mu^x_s,a) \right. \\
	&\left. - \int_0^t \sigma(s,X_s,\mu^x_s)dW_s - \int_0^t \sigma_0(s,X_s,\mu^x_s)dB_s\right|.
\end{align*}
Finally, combine this with the results of Kurtz and Protter \cite{kurtzprotter-weakconvergence} to conclude that the SDE \eqref{def:SDE} holds under $P$.
\end{proof}

\subsubsection{Optimality} \label{subse:optimality}
It remains to show the limit point $P$ in Lemma \ref{le:moments} is optimal. 
Generally speaking, the argument is as follows.
Fix $P' \in \RC\A(\rho)$, with $\rho := P \circ (\xi,B,W,\mu)^{-1}$. If we can prove that 
there exist $P'_n \in \RC_n\A_n(\rho_n)$ such that $J(P'_n) \rightarrow J(P')$, then, by
optimality of $P_n$ for each $n$, it holds that $J(P_n) \ge J(P'_n)$. Since $J$ is upper semicontinuous by Lemma \ref{le:jcontinuous}, we then get
\begin{align*}
J(P)	&\ge \limsup_{k\rightarrow\infty}J(P_{n_k}) \ge \lim_{k\rightarrow\infty}J(P'_{n_k}) = J(P').
\end{align*}
Since $P'$ was arbitrary, this implies that $P$ is optimal, or $P \in \RC\A^\star(\rho)$, which completes the proof of
Theorem \ref{th:existence}. 

The goal is thus to prove the existence of the sequence $(P'_{n})_{n \geq 1}$. 
For this, we need again to approximate general controls by adapted controls, as in Lemma \ref{le:adapteddensity}. To this end, now that $A$ is non-compact, we generalize the definition of the class $\A_a(\rho)$ in Definition 
\ref{def:adapted}: Let $\A_a(\rho)$ now denote the set of measures of the form
\[
\rho(d\omega,d\nu)\delta_{\phi(\omega,\nu)}(dq) = \rho \circ (\xi,B,W,\mu,\phi(\xi,B,W,\mu))^{-1}
\]
where $\phi : \Omega_0 \times \P^p(\X) \rightarrow \V_m$ is adapted and continuous and $m$ is some positive integer (see \eqref{eq:Anrho} for the definition of $\V_{m}$). In particular, a control $Q \in \A_a(\rho)$ satisfies $Q(\Lambda \in \V_m) = 1$ for some $m$ and renders $(\Lambda_t)_{t \in [0,T]}$ (a.s.-) adapted to 
$(\F^{\xi,B,W,\mu}_t)_{t \in [0,T]}$. Note that when $A$ is compact this definition specializes to the one provided before.
The construction of $(P'_{n})_{n \geq 1}$ then follows from the combination of the two next lemmas:

\begin{lemma} \label{le:limp'1}
For each $P' \in \RC\A(\rho)$ such that $J(P') > -\infty$, there exist $P'_n \in \RC\A_a(\rho)$ such that
$J(P') = \lim_{n\rightarrow\infty}J(P'_n)$. (As usual $\RC\A_a(\rho)$ is the image of $\A_{a}(\rho)$
by $\RC$.)
\end{lemma}

\begin{lemma} \label{le:limp'2}
For each $P' \in \RC\A_a(\rho)$, there exist $P'_n \in \RC_n\A_n(\rho_n)$ such that
$J(P') = \lim_{n\rightarrow\infty}J(P'_n)$.
(See \eqref{eq:Anrho} for the definition of $\A_{n}(\rho_{n})$.)
\end{lemma}

\subsubsection{Proof of Lemma \ref{le:limp'1}}
\textit{First step.}
First, assume $P' \in \RC\A_m(\rho)$ for some fixed $m$, so trivially $J(P') > -\infty$. That is, $P'(\Lambda \in \V_m) = 1$. Write $P' = \RC(Q')$, where $Q' \in \A_m(\rho)$. By Lemma \ref{le:adapteddensity}, $\A_a(\rho)$ is dense in $\A_m(\rho)$, and there exist $Q'_n \in \RC\A_a(\rho)$ such that $Q'_n \rightarrow Q'$ in $\P^p(\Omega_0 \times \P^p(\X) \times \V)$. Since $A_m$ is compact, $J \circ \RC$ is continuous on $\A_m(\rho)$ by Lemma \ref{le:jcontinuous}, and
$J(P') = \lim_{n\rightarrow\infty}J(\RC(Q'_n))$. 

\textit{Second step.}
Now assume $P' \in \RC\A(\rho)$ satisfies $J(P') > -\infty$. By the first step, it suffices to show that there exist $P'_n \in \RC\A_n(\rho)$ such that 
$J(P') = \lim_{n\rightarrow\infty}J(P'_n)$,
since we just showed that each $P'_n$ may be approximated by elements of $\RC\A_a(\rho)$. 

First, the upper bounds of $f$ and $g$ imply
\begin{align*}
-\infty < J(P') &\le c_2(T+1)\left(1 + \E^{P'}\|X\|_T^p + \E^{P'} \int_{\C^d} 
\|z\|^p_T\mu(dz) \right) - c_3\E^{P'}\int_0^Tdt \int_{A} |a|^{p'} \Lambda_t(da).
\end{align*}
Since $P' \in \RC\A(\rho)$, it must hold
$\E^{P'}\int_{\C^d} 
\|x\|^p_T\mu(dx) < \infty$ and $\E^{P'}\|X\|_T^p  < \infty$, which implies
\begin{align}
\E^{P'}\int_0^Tdt \int_{A} |a|^{p'} \Lambda_{t}(da) < \infty. \label{pf:lambdaintegrable2}
\end{align}
Let $\iota_n : A \rightarrow A$ denote any measurable function satisfying $\iota_n(A) \subset A_n$ and $\iota_n(a) = a$ for all $a \in A_n$, so that $\iota_n$ converges pointwise to the identity. Let $\Lambda^n$ denote the image under $\Lambda$ of the map $(t,a) \mapsto (t,\iota_n(a))$, so that $P'(\Lambda^n \in \V_n) = 1$. Let $Q'_n := P' \circ (\xi,B,W,\mu,\Lambda^n)^{-1}$, which is in $\RC\A_n(\rho)$. Since $\Lambda^n \rightarrow \Lambda$ $P'$-a.s., it follows that $Q'_n \rightarrow Q'$ in $\P^p(\Omega_0 \times \P^p(\X) \times \V)$, where $Q'$ satisfies $P'=\RC(Q')$. By continuity of $\RC$ (see Lemma \ref{le:rccontinuous}, which applies thanks to \eqref{pf:unbounded3}), $\RC(Q'_n)\rightarrow\RC(Q') = P'$ in $\P^p(\Omega)$.
Now, since $\vert \iota_{n}(a) \vert \leq \vert a \vert$, we have
\begin{equation*}
\int_{0}^T  \int_{A} \vert a \vert^{p'} \Lambda^n_{t}(da) dt
\leq \int_{0}^T  \int_{A} \vert a \vert^{p'} \Lambda_{t}(da) dt,
\end{equation*}
which implies that the sequence 
\begin{equation*}
\left( \int_{0}^T  \int_{A} \vert a \vert^{p'} \Lambda^n_{t}(da) dt
\right)_{n = 1}^\infty
\end{equation*}
is uniformly $P'$-integrable. By Lemma \ref{le:stateestimate}, we then have
\[
\sup_n\E^{\RC(Q'_n)}\left[\|X\|_T^{p'} + \int_{\C^d}\|z\|_T^{p'}\mu(dz)\right] < \infty.
\]
The growth assumptions of $f$ and $g$ imply that the rewards are uniformly integrable in the sense that
\[
\lim_{r\rightarrow\infty}\sup_n\E^{\RC(Q'_n)}\left[\Gamma(\mu,\Lambda,X)1_{\{|\Gamma(\mu,\Lambda,X)| > r\}}\right] = 0.
\]
Finally, from the continuity of $\Gamma$ we conclude that $J(\RC(Q'_n)) \rightarrow J(P')$.

\subsubsection{Proof of Lemma \ref{le:limp'2}}
Find $Q' \in \A_a(\rho)$ such that $P' = \RC(Q')$. There exist $m$ and an adapted function $\phi : \Omega_0 \times \P^p(\X) \rightarrow \V_m$ such that $\phi(\omega,\cdot)$ is continuous for each $\omega \in \Omega_0$ and
\[
Q' := \rho \circ \left(\xi,B,W,\mu,\phi(\xi,B,W,\mu)\right)^{-1}.
\]
Recalling the definition of $\rho_n$ from just before \eqref{def:nfixedpoint}, define 
\[
Q'_n := \rho_n \circ \left(\xi,B,W,\mu,\phi(\xi,B,W,\mu)\right)^{-1}.
\]
Note that $Q'_n(\Lambda \in \V_m) = 1$. Hence $Q'_n \in \A_n(\rho_n)$ for $n \ge m$. It follows from boundedness and continuity of $\phi$ in $\mu$ (and Lemma \ref{le:componentwise} to handle the fact that 
$\phi$ may not be continuous in $(\xi,B,W)$) that $Q'_n \rightarrow Q'$ 
The proof will be complete if we can show 
\begin{align}
\RC_n(Q'_n) \rightarrow P', \text{ in } \P^p(\Omega). \label{pf:limp'2-1}
\end{align}
Indeed, since $A_m$ is compact, we use the continuity of $J$ (see Lemma \ref{le:jcontinuous}) to complete the proof. We prove \eqref{pf:limp'2-1} with exactly the same argument as in Lemma \ref{le:rccontinuous}: Since $\RC_n(Q_n) \circ (\xi,B,W,\mu,\Lambda)^{-1} = Q_n$ are relatively compact in $\P^p(\Omega_0 \times \P^p(\X) \times \V)$, Aldous' criterion (see Proposition \ref{pr:itocompact} for details) implies that $\RC_n(Q'_n) \circ X^{-1}$ are relatively compact in $\P^p(\C^d)$. Thus $\RC_n(Q'_n)$ are relatively compact in $\P^p(\Omega)$. Conclude exactly as in the proof of Lemma \ref{le:rccontinuous} that any limit point must equal $P'$.

\section{Strict and strong controls} \label{se:strictstrongcontrols}
This section addresses the question of the existence of strict and strong controls. 
Recall that $Q \in \A(\rho)$ (resp. $P \in \RC\A(\rho)$) is a strict control if $Q(\Lambda \in \V_a) = 1$ (resp. $P(\Lambda \in \V_a) = 1$), where
\begin{align}
\V_a := \left\{q \in \V : q_t = \delta_{\alpha(t)} \text{ for some } \alpha \in L^p([0,T];A) \right\}. \label{def:va}
\end{align}
 Recall also that $Q$ is a strong control 
if  there exists an $A$-valued process $(\alpha_t)_{t \in [0,T]}$, progressively-measurable with respect to the $P$-completion of 
 $(\F^{\xi,B,W,\mu}_t)_{t \in [0,T]}$, such that $Q(\Lambda = dt\delta_{\alpha_t}(da))=1$. The first Subsection \ref{se:strictcontrols} addresses this point under a quite standard condition in control theory. The second Subsection \ref{se:strongcontrols} identifies more specialized assumptions which allow us to find a weak MFG solution with strong control. The idea in each case is the same as in references on relaxed controls: given any weak (relaxed) control, under suitable convexity assumptions, the optional projection of the control onto a suitable sub-filtration will yield an admissible control with a greater value than the original control, without disturbing the joint laws of the other processes.

\subsection{Strict controls} \label{se:strictcontrols}

The following assumption is well-known in control theory (dating to Filippov \cite{filippov-convexity}) and permits the construction of a weak MFG solution with weak \emph{strict} control.

\begin{assumption}{\textbf{C}} \label{assumption:C}
For each $(t,x,\mu) \in [0,T] \times \R^d \times \P^p(\R^d)$, the following set is convex:
\[
K(t,x,\mu) := \left\{\left(b(t,x,\mu,a),z\right) : a \in A, \ z \le f(t,x,\mu,a)\right\} \subset \R^d \times \R.
\]
\end{assumption}

The most obvious examples of assumption \ref{assumption:C} are the affine drifts $b$, i.e.
$b(t,x,\mu,a) = b^1(t,x,\mu)a + b^2(t,x,\mu)$, and objectives $f(t,x,\mu,a)$ which are concave in $a$. Here
is the main result of this subsection:

\begin{theorem} \label{th:strictexistence}
In addition to assumption \ref{assumption:A}, 
suppose also that assumption \ref{assumption:C} holds. Then there exists a weak MFG solution with weak strict control that satisfies $E \int_{0}^T \vert \alpha_{t} \vert^{p'} dt < \infty$. 
\end{theorem}

The proof of Theorem \ref{th:strictexistence} relies on 

\begin{proposition} \label{pr:strictexistence}
Assume \ref{assumption:A} and \ref{assumption:C} hold. Let $\rho \in 
\P^p_{c}[(\Omega_{0},\W_{\lambda})\leadsto 
{\mathcal P}^p(\X)]$ and $P \in \RC\A(\rho)$. Then there exists a strict control $P' \in \RC\A(\rho)$ such that 
\[
P' \circ (\xi,B,W,\mu,X)^{-1} := P \circ (\xi,B,W,\mu,X)^{-1}
\]
and $J(P') \ge J(P)$.
\end{proposition}
\begin{proof}[Proof of Proposition \ref{pr:strictexistence}.]
Note that
$\int_A\Lambda_t(da)(b,f)(t,X_t,\mu^x_t,a) \in K(t,X_t,\mu^x_t)$.
By \cite[Theorem A.9]{haussmannlepeltier-existence}, or rather a slight extension thereof in \cite[Lemma 3.1]{dufourstockbridge-existence}, there exist $(\F^{\xi,B,W,\mu,\Lambda,X}_t)_{t \in [0,T]}$-progressive processes $\hat{\alpha}$ and $\hat{z}$, taking values in $A$ and $[0,\infty)$, respectively, such that
\begin{align}
\int_A\Lambda_t(da)(b,f)(t,X_t,\mu^x_t,a) = &(b,f)(t,X_t,\mu^x_t,\hat{\alpha}_t) - (0,\hat{z}_t). \label{pf:strictexistence1}
\end{align}
Define $P' \in \P(\Omega)$ by
$P' = P \circ (\xi,B,W,\mu,dt\delta_{\hat{\alpha}_t}(da),X)^{-1}$.
Clearly, $(B,W)$ is a Wiener process with respect to $({\mathcal F}_{t}^{\xi,B,W,\mu,\Lambda,X})_{t \in [0,T]}$ under $P'$. 
Since the state equation \eqref{def:SDE} holds under $P$, it follows from \eqref{pf:strictexistence1} that the state equation holds under $P'$ as well, since the first coordinate of the remainder $(0,\hat{z}_t)$ in \eqref{pf:strictexistence1} is zero. Moreover,
\begin{align*}
J(P') &= \E^{P}\biggl[\int_0^Tdtf(t,X_t,\mu^x_t,\hat{\alpha}_t) + g(X_T,\mu^x_T)\biggr] \\
	&\ge \E^{P}\biggl[\int_0^Tdt\int_A\Lambda_t(da)f(t,X_t,\mu^x_t,a)dt + g(X_T,\mu^x_T)\biggr] = J(P).
\end{align*}
Letting
\[
Q' := P' \circ (\xi,B,W,\mu,\Lambda)^{-1} = P \circ (\xi,B,W,\mu,dt\delta_{\hat{\alpha}_t}(da))^{-1},
\]
$Q'$ is
is in  $\A(\rho)$. 
The reason is that $(\hat{\alpha}_{t})_{t \in [0,T]}$
is $(\F^{\xi,B,W,\mu,\Lambda,X}_{t})_{t \in [0,T]}$-progressive, 
so that, for each $t \in [0,T]$ and $C \in {\mathcal B}(A)$, 
$\int_{0}^t 1_{C}(\hat{\alpha}_{s}) ds$ is $\F^{\xi,B,W,\mu,\Lambda,X}_{t}$-measurable.  
Since the solution of the state equation 
\eqref{def:relaxedSDE-weak} is strong, $\int_{0}^t 1_{C}(\hat{\alpha}_{s}) ds$ coincides
$P$ a.s. with a $\F^{\xi,B,W,\mu,\Lambda}_{t}$-measurable random variable. 
By assumption, $\F^{\Lambda}_{t}$ and $\F^{\xi,B,W,\mu}_{T}$ are conditionally independent
under $P$  
given  $\F^{\xi,B,W,\mu}_{t}$. We deduce that 
$\sigma(\int_{0}^r 1_{C}(\hat{\alpha}_{s}) ds : r \leq t, C \in {\mathcal B}(A))$
and $\F^{\xi,B,W,\mu}_{T}$ are also conditionally independent 
under $P$ given  $\F^{\xi,B,W,\mu}_{t}$, which is enough to prove that $Q' \in \A(\rho)$. 
\end{proof}

\begin{proof}[Proof of Theorem \ref{th:strictexistence}]
Let $P \in \P(\Omega)$ be a MFG solution, whose existence is guaranteed by Theorem 
\ref{th:existence}, and set $\rho = P \circ (\xi,B,W,\mu)^{-1}$. By Proposition \ref{pr:strictexistence}, 
there exists $P' \in \RC\A(\rho)$ such that $P \circ (\xi,B,W,\mu,X)^{-1} = P' \circ (\xi,B,W,\mu,X)^{-1}$, $J(P) \le J(P')$, and $P'(\Lambda_t = \delta_{\alpha_t} \ a.e. \ t) = 1$ for some 
$(\F^{B,W,\mu,\Lambda,X}_t)_{t \in [0,T]}$-progressive process $(\alpha_t)_{t \in [0,T]}$. But since $P \in \RC\A^\star(\rho)$ (i.e. $P$ is optimal for the control problem corresponding to $\rho$), it follows that $P' \in \RC\A^\star(\rho)$. It remains to deal with the fixed point condition. Define
\[
\bar{\mu} := P'\left((W,\Lambda,X) \in \cdot \ | \ B,\mu\right).
\]
Conditioning on $(B,\bar{\mu})$ yields
$\bar{\mu} := P'((W,\Lambda,X) \in \cdot \ | \ B,\bar{\mu})$.
Now if $\phi : \X \rightarrow \R$ is $\F^{W,\Lambda,X}_t$-measurable then
\begin{align*}
\int_\X\phi\,d\bar{\mu} = \E^{P'}\bigl[ \left.\phi(W,\Lambda,X) \right| B,\mu \bigr] &= \E^{P'}\left[ \left.\E^{P'}\left[ \left.\phi(W,\Lambda,X)\right| \F^{\xi,B,W,\mu}_T\right] \right| \F^{B,\mu}_T \right] \\
	&= \E^{P'}\left[ \left.\E^{P'}\left[ \left.\phi(W,\Lambda,X)\right| \F^{\xi,B,W,\mu}_t\right] \right| \F^{B,\mu}_T \right] \\
	&= \E^{P'}\left[ \left.\phi(W,\Lambda,X) \right| \F^{B,\mu}_t \right],
\end{align*}
The second equality follows from the conditional independence of $\F^{\xi,B,W,\mu,\Lambda,X}_t$ and  $\F^{\xi,B,W,\mu}_T$ given $\F^{\xi,B,W,\mu}_t$ under $P'$, which holds because $P' \in \RC\A(\rho)$, and the last equality follows easily from the independence of $(\xi,W)$ and $(B,\mu)$. This holds for each $\phi$, and thus $\F^{\bar{\mu}}_t \subset \F^{B,\mu}_t$ for all $t$, up to $\rho \circ (B,\mu)^{-1}$-null sets. It follows that $(B,W)$ is a Wiener process under $P'$ with respect to the filtration generated by $(\xi,B,W,\bar{\mu},\Lambda,X)$, which is smaller than $(\xi,B,W,\mu,\Lambda,X)$. Moreover, 
by definition, $\bar{\mu}^x = P'(X \in \cdot \ | \ B,\mu)$, and 
since $P' \circ (B,\mu,X)^{-1} = P \circ (B,\mu,X)^{-1}$ and 
$\mu^x = P(X \in \cdot \ | \ B,\mu)$ imply $\mu^x = P'(X \in \cdot \ | \ B,\mu)$, we deduce that $P'(\bar{\mu}^x = \mu^x) = 1$. 

Now define $\overline{P} := P' \circ (\xi,B,W,\bar{\mu},\Lambda,X)^{-1}$; we will show that this is in fact the MFG solution we are looking for. Indeed, from $P'(\bar{\mu}^x=\mu^x)=1$
it follows that the canonical processes verify the state equation \eqref{def:SDE} under $\overline{P}$. 
Hence, in light of the above considerations, we may apply Lemma  \ref{le:a.s.compatible} to conclude that $\overline{P}$ is a MFG pre-solution (with weak strict control). In particular, we have $\overline{P} \in \RC\A(\overline{\rho})$, where $\overline{\rho} := \overline{P} \circ (\xi,B,W,\mu)^{-1} = \rho \circ (\xi,B,W,\bar{\mu})^{-1}$. Moreover, $P' \circ (\mu^x,\Lambda,X)^{-1} = \overline{P}\circ (\mu^x,\Lambda,X)^{-1}$ clearly implies $J(P') = J(\overline{P})$. Although $P' \in \RC\A^\star(\rho)$, this does not immediately imply that $\overline{P} \in \RC\A^\star(\overline{\rho})$, and we must complete the proof carefully.

Fix $\overline{Q} \in \A_a(\overline{\rho})$, where we recall the definition of $\A_a(\rho)$ from Section \ref{subse:optimality}. That is
\[
\overline{Q} = \overline{\rho} \circ (\xi,B,W,\mu,\phi(\xi,B,W,\mu))^{-1} = \rho \circ (\xi,B,W,\bar{\mu},\phi(\xi,B,W,\bar{\mu}))^{-1}
\]
for some adapted function $\phi : \Omega_0 \times \P^p(\X) \rightarrow \V$. Define
\[
\overline{Q}' :=  \rho \circ (\xi,B,W,\mu,\phi(\xi,B,W,\bar{\mu}))^{-1}.
\]
Then, since $\phi$ is adapted and $\F^{\bar{\mu}}_t \subset \F^{B,\mu}_t$ up to null sets, we conclude that $\overline{Q}'$ is compatible, or $\overline{Q}' \in \A(\rho)$.
Since $P'(\bar{\mu}^x=\mu^x)=1$, we have $\RC(\overline{Q}) \circ (\mu^x,\Lambda,X)^{-1} = \RC(\overline{Q}') \circ (\mu^x,\Lambda,X)^{-1}$. Thus $P' \in \RC\A^{\star}(\rho)$ implies
\begin{align*}
J(\overline{P}) &= J(P') \ge J(\RC(\overline{Q}')) = J(\RC(\overline{Q})).
\end{align*}
Since this holds for all $\overline{Q} \in \A_a(\overline{\rho})$, we finally conclude that $\overline{P} \in \RC\A^{\star}(\overline{\rho})$ by combining the density results of Lemmas \ref{le:limp'1} and \ref{le:limp'2}.
\end{proof}

\begin{remark}
It is possible to strengthen this result slightly to conclude that there exists a relaxed MFG solution with weak strict control $\alpha_t$ adapted to $\F^{\xi,B,W,\mu,X}_t$. Indeed, the argument could proceed along the lines of Proposition \ref{pr:strongcontrol} or by way of martingale problems, as in \cite{elkaroui-compactification,haussmannlepeltier-existence}.
\end{remark}

\subsection{Strong controls} \label{se:strongcontrols}
A strong but common linearity assumption on the coefficients $b$, $\sigma$, and $\sigma_0$ allows us to find strong controls.

\begin{assumption}{\textbf{D}} \label{assumption:D} {\ }
\begin{enumerate}
\item[(D.1)] $A$ is a convex subset of an Euclidean space, and the state coefficients are affine in $(x,a)$, in the following form:
\begin{align*}
&b(t,x,\mu,a) = b^1(t,\mu)x + b^2(t,\mu)a + b^3(t,\mu), 
\\
&\sigma(t,x,\mu) = \sigma^1(t,\mu)x + \sigma^2(t,\mu), \ \sigma_0(t,x,\mu) = \sigma^1_0(t,\mu)x + \sigma^1_0(t,\mu),
\end{align*}
\item[(D.2)] The objective functions are concave in $(x,a)$; that is, the maps $(x,a) \mapsto f(t,x,\mu,a)$ and $x \mapsto g(x,\mu)$ are concave for each $(t,\mu)$.
\item[(D.3)] $f$ is \emph{strictly} concave in $(x,a)$.
\end{enumerate}
\end{assumption}

\begin{proposition} \label{pr:strongcontrol}
Under assumptions \ref{assumption:A} and (D.1-2), then 
\begin{enumerate}
\item[(1)] For each $\rho \in \P^p_{c}[(\Omega_{0},\W_{\lambda}) \leadsto 
{\mathcal P}^p(\X)]$ there exists a strong optimal control; that is $\A^\star(\rho)$ contains a strong control.
\end{enumerate}
If also (D.3) holds, then 
\begin{enumerate}
\item[(2)] For each $\rho$ the optimal control is unique; that is $\A^\star(\rho)$ is a singleton for each $\rho$.
\item[(3)] Every weak MFG solution with weak control is a weak MFG solution with strong control.
\end{enumerate}
\end{proposition}

\begin{proof} \

\textit{Proof of (1).}
Let $P \in \RC\A^\star(\rho)$, which is nonempty by Lemma \ref{le:jcontinuous}. Under assumption \ref{assumption:D}, the state equation writes as
\begin{align}
X_t = X_0 &+ \int_0^t\left(b^1(s,\mu^x_s)X_s + b^2(s,\mu^x_s) \alpha_{s}
 + b^3(t,\mu^x_s)\right)ds \nonumber \\
	&+ \int_0^t\left(\sigma^1(s,\mu^x_s)X_s + \sigma^2(s,\mu^x_s)\right)dW_s + \int_0^t\left(\sigma^1_0(s,\mu^x_s)X_s + \sigma^2_0(s,\mu^x_s)\right)dB_s, \label{pf:strongcontrol1}
\end{align}
where we have let 
$\alpha_{s} :=
\int_Aa\Lambda_s(da)$.
Let $(\widetilde{\F}_{t}^{\xi,B,W,\mu})_{t \in [0,T]}$ denote the $P$-completion of the filtration $(\F_{t}^{\xi,B,W,\mu})_{t \in [0,T]}$.
By optional projection (see \cite[Appendix A.3]{kurtz-conditionalmtgproblem} for a treatment without right-continuity of the filtration),
there exist $(\widetilde{\mathcal F}_{t}^{\xi,B,W,\mu})_{t \in [0,T]}$-optional (and thus progressive) processes $(\widetilde{X}_{t})_{t \in [0,T]}$ and $(\widetilde{\alpha}_{t})_{t \in [0,T]}$ such that
such that, for each $t \in [0,T]$, 
\begin{align*}
\widetilde{X}_t &:= \E\bigl[ X_t |\widetilde{\F}_{t}^{\xi,B,W,\mu}\bigr], \quad \widetilde{\alpha}_t := \E\bigl[ \alpha_{t} |\widetilde{\F}^{\xi,B,W,\mu}_{t}\bigr], \ a.s. 
\end{align*}
In fact, it holds that for each $0 \le s \le t \le T$,
\begin{align}
\widetilde{X}_s &:= \E\bigl[ X_s |\widetilde{\F}_t^{\xi,B,W,\mu}\bigr],  \quad
\widetilde{\alpha}_s := \E\bigl[ \alpha_s |\widetilde{\F}^{\xi,B,W,\mu}_{t}\bigr],  \ a.s. \label{pf:strongcontrol1c}
\end{align}
Indeed, since $(\alpha_s,X_s)$ is $\F^{\xi,B,W,\mu,\Lambda,X}_s$-measurable, and since the solution of the state equation \ref{def:relaxedSDE-weak} is strong, we know that $(\alpha_s,X_s)$ is a.s. $\F^{\xi,B,W,\mu,\Lambda}_s$-measurable. By compatibility, $\F^{\xi,B,W,\mu}_t$ and $\F^\Lambda_s$  are conditionally independent given $\F^{\xi,B,W,\mu}_s$, and thus $\widetilde{\F}^{\xi,B,W,\mu}_t$ and the completion of $\F^\Lambda_s$ are conditionally independent given $\widetilde{\F}^{\xi,B,W,\mu}_s$. This implies \eqref{pf:strongcontrol1c}.

Now, for a given $t \in [0,T]$, 
take the conditional expectation with respect to $\widetilde{\F}_{t}^{\xi,B,W,\mu}$ in \eqref{pf:strongcontrol1}. 
Using a conditional version of Fubini's theorem together with \eqref{pf:strongcontrol1c}, we get that for each $t \in [0,T]$ it holds $P$-a.s. that 
\begin{equation}
\label{pf:strongcontrol1b}
\begin{split}
\widetilde{X}_t = \xi &+ \int_0^t\left(b^1(s,\mu^x_s)\widetilde{X}_s + b^2(s,\mu^x_s)\widetilde{\alpha}_s + b^3(t,\mu^x_s)\right)ds \\
	&+ \int_0^t\left(\sigma^1(s,\mu^x_s)\widetilde{X}_s + \sigma^2(s,\mu^x_s)\right)dW_s + \int_0^t\left(\sigma^1_0(s,\mu^x_s)\widetilde{X}_s + \sigma^2_0(s,\mu^x_s)\right)dB_s.
	\end{split}
\end{equation}
Since the right-hand side is continuous a.s. and the filtration is complete, we replace $\widetilde{X}$ with an a.s.-continuous modification, so that \eqref{pf:strongcontrol1b} holds for all $t \in [0,T]$, $P$-a.s. That is, the processes on either side of the equation are indistinguishable. 

Now define
$\widetilde{P} := P \circ (\xi,B,W,\mu,dt\delta_{\widetilde{\alpha}_t}(da),\widetilde{X})^{-1}$. It is clear from \eqref{pf:strongcontrol1b} that $\widetilde{P} \in \RC\A(\rho)$.
Jensen's inequality provides
\begin{align}
J(P) 	&\le \E^P\biggl[\int_0^Tf\bigl(t,X_t,\mu^x_t,\alpha_{t}\bigr)dt + g(X_T,\mu^x_T)\biggr] \label{pf:weakstrong1} \\
	&= \E^P\biggl[\int_0^T\E^P\bigl[ f\bigl(t,X_t,\mu^x_t,\alpha_{t}\bigr)\big|\widetilde{\F}^{\xi,B,W,\mu}_t\bigr]dt + \E^P\bigl[ g(X_T,\mu^x_T) |\widetilde{\F}^{\xi,B,W,\mu}_T\bigr]\biggr] \nonumber \\
	&\le \E^P\left[\int_0^Tf(t,\widetilde{X}_t,\mu^x_t,\widetilde{\alpha}_t)dt + g(\widetilde{X}_T,\mu^x_T)\right] \label{pf:weakstrong2} = J(\widetilde{P}).
\end{align}
Hence $\widetilde{P} \in \RC\A^\star(\rho)$, and (1) is proven.

\textit{Proof of (2) and (3)}
Now suppose assumption (D.3) holds. We prove only (2), from which (3) follows immediately. Unless $\Lambda$ is already a strict control, then inequality \eqref{pf:weakstrong1} is strict, and unless $\int_Aa\Lambda_t(da)$ is already 
$(\widetilde{\F}^{\xi,B,W,\mu}_t)_{t \in [0,T]}$-adapted, the inequality \eqref{pf:weakstrong2} is strict: $J(\widetilde{P}) > J(P)$. This proves that all optimal controls must be strict and $(\widetilde{\F}^{\xi,B,W,\mu}_t)_{t \in [0,T]}$-adapted. Now suppose we have two strict adapted optimal controls, which without loss of generality we construct on the same space $(\Omega_0 \times \P(\X),(\widetilde{\F}^{\xi,B,W,\mu}_t)_{t \in [0,T]},\rho)$. That is,
\begin{align*}
X^i_t = X_0 &+ \int_0^t\left(b^1(s,\mu^x_s)X^i_s + b^2(s,\mu^x_s)\alpha^i_s + b^3(s,\mu^x_s)\right)ds \\
	&+ \int_0^t\left(\sigma^1(s,\mu^x_s)X^i_s + \sigma^2(s,\mu^x_s)\right)dW_{s}+ \int_0^t\left(\sigma^1_0(s,\mu^x_s)X^i_s + \sigma^2_0(s,\mu^x_s)\right)dB_{s}, \ i=1,2,
\end{align*}
where $\alpha^i$ is $\F^{X_0,B,W,\mu}_t$-adapted. Define 
\begin{align*}
X^3_t &:= \frac{1}{2}X^1_t + \frac{1}{2}X^2_t, \quad 
\alpha^3_t := \frac{1}{2}\alpha^1_t + \frac{1}{2}\alpha^2_t.
\end{align*}
Again taking advantage of the linearity of the coefficients, it is straightforward to check that $(X^3,\alpha^3)$ also solve the state equation. Unless $\alpha^1 = \alpha^2$ holds $dt \otimes dP$-a.e., the strict concavity and Jensen's inequality easily imply that this new control achieves a strictly larger reward than either $\alpha^1$ or $\alpha^2$, which is a contradiction.
\end{proof}

\section{Counterexamples} \label{se:counterexamples}
In this section, simple examples are presented to illustrate two points. First, we demonstrate why we cannot expect existence of a strong MFG solution at the level of generality allowed by assumption \ref{assumption:A}. Second, by providing an example of a mean field game which fails to admit even a weak solution, we show that the exponent $p$ in both the upper and lower bounds of $f$ and $g$ cannot be relaxed to $p'$.

\subsection{Nonexistence of strong solutions}
Suppose $\sigma$ is constant, $g \equiv 0$, $p' = 2$, $p = 1$, $A = \R^d$, and choose the following data:
\begin{align*}
b(t,x,\mu,a) = a, \quad
f(t,x,\mu,a) = a^\top\tilde{f}(t,\bar{\mu}) - \frac{1}{2}|a|^2, \quad
\sigma_0(t,x,\mu) = \tilde{\sigma}_0(t,\bar{\mu}),
\end{align*}
for some bounded continuous functions $\tilde{f} : [0,T] \times \R^d \rightarrow \R^d$ and $\tilde{\sigma}_0 : [0,T] \times \R^d \rightarrow \R^{d \times m_0}$. Here we have abbreviated 
$\bar{\mu} := \int_\R z\mu(dz)$ for $\mu \in \P^1(\R)$. 
Proposition \ref{pr:strongcontrol} ensures that there exists a weak MFG solution $P$ with strong control. That is
(with the same notations as in Proposition \ref{pr:strongcontrol}), there 
exists an $(\widetilde{\F}^{\xi,B,W,\mu}_t)_{t \in [0,T]}$-progressive $\R^d$-valued process 
$(\alpha^\star_t)_{t \in [0,T]}$ such that
\[
P(\Lambda = dt\delta_{\alpha^\star_t}(da)) = 1, \quad \E^P\int_0^1|\alpha^\star_t|^2dt < \infty.
\]
If $(\alpha_t)_{t \in [0,T]}$ is any bounded $(\widetilde{\F}_{t}^{\xi,B,W,\mu})_{t \in [0,T]}$-progressive $\R^d$-valued processes, then optimality of $\alpha^\star$ implies
\[
\E\int_0^1\Bigl((\alpha^\star_t)^\top\tilde{f}(t,\bar{\mu}^x_t) - \frac{1}{2}|\alpha^\star_t|^2\Bigr)dt \ge 
\E\int_0^1\Bigl(\alpha_t^\top\tilde{f}(t,\bar{\mu}^x_t) - \frac{1}{2}|\alpha_t|^2\Bigr)dt. 
\]
Hence, $\alpha^\star_t = \tilde{f}(t,\bar{\mu}^x_t)$ holds $dt \otimes dP$-a.e. The optimally controlled state process is given by
\[
dX^\star_t = \tilde{f}(t,\bar{\mu}^x_t)dt + \sigma dW_t + \tilde{\sigma}_0(t,\bar{\mu}^x_t)dB_t.
\]
Conditioning on $(B,\mu)$ and using the fixed point property $\bar{\mu}^x_t = \E[X_t  |  B,\mu]$ yields
\[
d\bar{\mu}^x_t = \tilde{f}(t,\bar{\mu}^x_t)dt + \tilde{\sigma}_0(t,\bar{\mu}^x_t) dB_t, \ \ \bar{\mu}^x_0 = \E[X_0].
\]
We have only assumed that $\tilde{f}$ and $\tilde{\sigma}_0$ are bounded and continuous. For the punchline, note that 
uniqueness in distribution may hold for such a SDE even if it fails to possess a strong solution, in which case $\bar{\mu}^x_t$ cannot be adapted to the completion of $\F^B_t$ and the MFG solution cannot be strong. Such cases are not necessarily pathological; see Barlow \cite{barlow-sdenostrongsolution} for examples in dimension $d=1$ with $\tilde{f} \equiv 0$ and $\sigma_0$ bounded above and below away from zero.

\subsection{Nonexistence of weak solutions}
Unfortunately, assumption \ref{assumption:A} does not cover linear-quadratic models with quadratic objectives in $x$ or $\mu$. That is, we do not allow
\[
f(t,x,\mu,a) = -|a|^2 - c\biggl|x + c'\int_{\R^d} z\mu(dz)\biggr|^2, \ c,c' \in \R.
\]
Even when $c > 0$, so that $f$ and $g$ are bounded from above, we cannot expect a general existence result if $p'=p$. This was observed in \cite{carmonadelaruelachapelle-mkvvsmfg,lacker-mfgcontrolledmartingaleproblems} 
 in the case $\sigma_0=0$; the authors showed that only \emph{certain} linear-quadratic mean field games admit (strong) solutions. The following example reiterates this point in the setting of common noise and weak solutions, extending the example of \cite[Section 7]{lacker-mfgcontrolledmartingaleproblems}.

Consider constant volatilities $\sigma$ and $\sigma_0$, $d=1$, $p'=p=2$, $A = \R$, and and the following data:
\begin{align*}
b(t,x,\mu,a) = a, \quad
f(t,x,\mu,a) = -a^2, \quad
g(x,\mu) = -(x + c\bar{\mu})^2, \ c \in \R,
\end{align*}
where $\bar{\mu} := \int_{\R^d} z\mu(dz)$ for $\mu \in \P^1(\R^d)$. Choose $T > 0$, $c \in \R$, and $\lambda \in \P^2(\R)$ such that
\begin{align*}
c=(1-T)/T, \quad\quad T \neq 1, \quad\quad \bar{\lambda} \neq 0.
\end{align*}
Assumptions \ref{assumption:A}(1-5) hold with the one exception that the assumption $p' > p$ is violated. 
Suppose $P$ is a weak MFG solution with weak control and
then define $y_t := \E\bar{\mu}_t^x$. 
Arguing as in \cite{lacker-mfgcontrolledmartingaleproblems}, 
we get 
\begin{align*}
y_T =  \frac{y_0}{1-T} + y_T,
\end{align*}
which implies $y_0 = 0$ and which contradicts $\bar{\lambda} \neq 0$ since $y_0 = \E\bar{\mu}_0^x = \bar{\lambda}$. Hence, for this particular choice of data, there is no weak solution.

It would be interesting to find additional structural conditions under which existence of a solution holds in the case 
$p'=p$. This question has been addressed in \cite{carmonadelarue-mfg}
when $p'=p=2$, $b$ is linear, $\sigma$ is constant, $f$ and $g$ are convex
in $(x,\alpha)$ and without common noise. Therein, the strategy consists in solving approximating equations, for which the related $p$ is indeed less than $2$, and then in passing to the limit. In order to guarantee the tightness of the approximating solutions, the authors introduce a so-called \textit{weak mean-reverting condition}, which somehow generalizes the classical conditions for handling linear-quadratic MFG. 
It reads $\langle x,\partial_{x}g(0,\delta_{x}) \rangle \leq  c (1+ \vert x \vert)$
and $\langle x,\partial_{x}f(t,0,\delta_{x},0) \rangle \leq  c (1+ \vert x \vert)$, where 
$\delta_{x}$ is the Dirac mass at point $x$. This clearly imposes some restriction on the coefficients as, in full generality (when $p=p'=2$), $\partial_{x} g(0,\delta_{x})$ and
$\partial_{x} f(t,0,\delta_{x},0)$ are expected to be of order 1 in $x$. 
The \textit{weak mean-reverting condition} assures that  
the expectations of the approximating solutions remain bounded along the approximation, which actually suffices to prove tightness. We feel that the same strategy could be applied to our
setting by considering the conditional expectation given the common noise instead of the expectation itself. Anyhow, in order to limit the length of the paper, we refrain from 
discussing further this question.

\section{Uniqueness} \label{se:uniqueness}
We now discuss uniqueness of solutions. 
The goal is twofold. Inspired by the Yamada-Watanabe theory for weak and strong solutions to standard 
stochastic differential equations, we first claim that every weak MFG solution is actually a strong MFG solution provided 
the MFG solutions are \textit{pathwise unique}. This is a quite important point from the practical point of view as it guarantees that 
the equilibrium measure
$\mu^x$ is adapted to the common noise $B$. As an illustration, we prove a modest uniqueness result, inspired by the
earlier works by  Lasry and Lions \cite{lasrylionsmfg}. 
When there is no mean field term in the state coefficients, when the optimal controls are unique, and when the monotonicity condition of Lasry and Lions \cite{lasrylionsmfg} holds, we indeed have a form of \emph{pathwise uniqueness}.

\subsection{Pathwise uniqueness and uniqueness in law}
The starting point of our analysis is to notice that the law of a weak MFG solution is really determined by the law of $(B,\mu)$. 
Indeed, for
an element $\gamma \in \P^p(\C^{m_{0}} \times \P^p(\X))$, 
we can
define $M\gamma \in \P(\Omega)$ by
\[
M\gamma(d\xi,d\beta,dw,d\nu,dq,dx) = \gamma(d\beta,d\nu)\nu(dw,dq,dx)\delta_{x_{0}}(d\xi).
\]
We will say $\gamma$ is a \emph{MFG solution basis} if
the distribution $M\gamma$ together with the canonical 
processes on $\Omega$ form a weak MFG solution. 
We say \emph{uniqueness in law} holds for the MFG if there is at most one MFG solution basis, or equivalently if any two weak MFG solutions induce the same law on $\Omega$.
Given two MFG solution bases $\gamma^1$ and $\gamma^2$, we say $(\Theta,(\G_t)_{t \in [0,T]},Q,B,\mu^1,\mu^2)$ is a \emph{coupling} of $\gamma^1$ and $\gamma^2$ if:
\begin{enumerate}
\item $(\Theta,(\G_t)_{t \in [0,T]},Q)$ is a probability space with a complete filtration.
\item $B$ is a $(\G_t)_{t \in [0,T]}$-Wiener process on $\Theta$.
\item For $i=1,2$, $\mu^i : \Theta \rightarrow \P^p(\X)$ is such that, for each $t \in [0,T]$ and $C \in \F^{W,\Lambda,X}_t$, $\mu^i(C)$ is $\G_t$-measurable.
\item For $i=1,2$, $Q \circ (B,\mu^i)^{-1} = \gamma^i$.
\item $\mu^1$ and $\mu^2$ are conditionally independent given $B$.
\end{enumerate}

Suppose that for any coupling $(\Theta,(\G_t)_{t \in [0,T]},Q,B,\mu^1,\mu^2)$ of any two MFG solution bases $\gamma^1$ and $\gamma^2$ we have $\mu^1=\mu^2$ a.s. Then we say \emph{pathwise uniqueness} holds for the mean field game. The following proposition essentially follows from Theorem 1.5 and Lemma 2.10 of \cite{kurtz-yw2013}, but we include the proof since we use slightly different notions of compatibility and of pathwise uniqueness.

\begin{proposition} \label{pr:yamadawatanabe}
Suppose assumption \ref{assumption:A} and pathwise uniqueness hold. Then there exists a unique in law weak MFG solution with weak control, and it is in fact a strong MFG solution with weak control. 
\end{proposition}
\begin{proof}
Let $\gamma^1$ and $\gamma^2$ be any two MFG solution bases. Let $\Theta = \C^{m_0} \times \P^p(\X) \times \P^p(\X)$, and let $(B,\mu^1,\mu^2)$ denote the identity map on $\Theta$. Let $Q$ be the unique probability measure on $\Theta$ under which $Q \circ (B,\mu^i)^{-1} = \gamma^i$ for $i=1,2$ and also $\mu^1$ and $\mu^2$ are conditionally independent given $B$.
Define the $(\G_t)_{t \in [0,T]}$ to be the $Q$-completion of the filtration
\[
\left(\sigma(B_s,\mu^1(C),\mu^2(C) : s \le t, \ C \in \F^{W,\Lambda,X}_t)\right)_{t \in [0,T]}.
\]
Then $(\Theta,(\G_t)_{t \in [0,T]},Q,B,\mu^1,\mu^2)$ satisfies conditions (1) and (3-5) of the definition of a coupling. We will check that in a moment that (2) necessarily holds as well. It then follows from pathwise uniqueness that $\mu^1=\mu^2$ almost surely, which in turn implies $\gamma^1=\gamma^2$. Conclude in the usual way (as in \cite[Theorem 1.5]{kurtz-yw2013} or \cite[Theorem 3.20]{jacodmemin-weaksolution}) that this unique solution is in fact a strong solution.

To see that $B$ is a $(\G_t)_{t \in [0,T]}$-Wiener process, we need only to check that $\sigma(B_s - B_t : s \in [t,T])$ is independent of $\G_t$ for each $t \in [0,T]$. Fix $t \in [0,T]$ and let $Z_t$, $Z_{t+}$, $Y^1_t$, and $Y^2_t$ be bounded random variables, measurable with respect to $\F^B_t := \sigma(B_s : s \le t)$, $\sigma(B_s - B_t : s \in [t,T])$, $\sigma(\mu^1(C) : C \in \F^{W,\Lambda,X}_t)$, and $\sigma(\mu^2(C) : C \in \F^{W,\Lambda,X}_t)$, respectively. Since $\gamma^i$ is a MFG solution basis, we know that $B$ is a Wiener process with respect to the filtration
\[
\left(\sigma(B_u,\mu^i(C) : u \le s, \ C \in \F^{W,\Lambda,X}_s)\right)_{s \in [0,T]}.
\]
Thus $Y^i_t$ is independent of $\sigma(B_s - B_t : s \in [t,T])$. Conditional independence of $\mu^1$ and $\mu^2$ implies
\begin{align*}
\E^Q\left[\left.Y^1_tY^2_t\right| B \right] = \E^Q\left[\left.Y^1_t\right| B\right]\E^Q\left[\left.Y^2_t\right| B\right] = \E^Q\left[\left.Y^1_t\right|\F^B_t\right]\E^Q\left[\left.Y^2_t\right|\F^B_t\right].
\end{align*}
Thus, since $Z_{t+}$ is independent of $\F^B_t$, 
\begin{align*}
\E^Q\left[Z_{t+}Z_tY^1_tY^2_t\right] &= \E^Q\left[Z_{t+}Z_t\E^Q\left[\left.Y^1_t\right|\F^B_t\right]\E^Q\left[\left.Y^2_t\right|\F^B_t\right]\right] \\
	&= \E^Q\left[Z_{t+}\right]\E^Q\left[Z_t\E^Q\left[\left.Y^1_t\right|\F^B_t\right]\E^Q\left[\left.Y^2_t\right|\F^B_t\right]\right] \\
	&= \E^Q\left[Z_{t+}\right]\E^Q\left[Z_tY^1_tY^2_t\right].
\end{align*}
This implies that $\sigma(B_s - B_t : s \in [t,T])$ is independent of $\G_t$.
\end{proof}

\subsection{Lasry-Lions monotonicity condition}
An application of Proposition \ref{pr:yamadawatanabe} is possible under
\begin{assumption}{\textbf{U}} \label{assumption:U} {\ }
\begin{enumerate}
\item[(U.1)] $b$, $\sigma$, and $\sigma_0$ have no mean field term.
\item[(U.2)] $f$ is of the form $f(t,x,\mu,a) = f_1(t,x,a) + f_2(t,x,\mu)$.
\item[(U.3)] For all $\mu,\nu \in \P^p(\C^d)$ we have the Lasry-Lions monotonicity condition:
\begin{align}
\int_{\C^d}(\mu-\nu)(dx)\left[g(x_T,\mu_T) - g(x_T,\nu_T) + \int_0^T\left(f_2(t,x_t,\mu_t)-f_2(t,x_t,\nu_t)\right)dt\right] \le 0. \label{monotonicity}
\end{align}
\item[(U.4)] For any $\rho \in \P^p_{c}[(\Omega_{0},\W_{\lambda})\leadsto 
{\mathcal P}^p(\X)]$ the set $\A^\star(\rho)$ is a singleton, which means that 
the maximization problem \textit{in the environment $\rho$} has a unique (relaxed) solution. See
\eqref{eq:astar} for the definition of $\A^\star(\rho)$. 
\end{enumerate}
\end{assumption}
Note that assumptions (D.1-3) imply (U.4), by Proposition \ref{pr:strongcontrol}. We then claim:
\begin{theorem} \label{th:uniqueness}
Suppose assumptions \ref{assumption:A} and \ref{assumption:U} hold. Then there exists a unique in law weak MFG solution with weak control, and it is in fact a strong MFG solution with weak control. In particular, under \ref{assumption:A}, (D.1-3), and (U.1-3), the unique in law weak MFG solution with weak control is in fact a strong MFG solution with strong control.
\end{theorem}

\begin{proof}
\textit{First step.}
Let $\gamma^1$ and $\gamma^2$ be two MFG solution bases, and define
\begin{align*}
\rho^i &:= (M\gamma^i) \circ (\xi,B,W,\mu)^{-1}.
\end{align*}
Let $(\Theta,(\G_t)_{t \in [0,T]},Q,B,\mu^1,\mu^2)$ be any coupling of $\gamma^1$ and $\gamma^2$. In view of Proposition \ref{pr:yamadawatanabe}, we will prove $\mu^1 = \mu^2$ a.s. In fact, we may assume without loss of generality that
\begin{align*}
\Theta = \C^{m_0} \times \P^p(\X) \times \P^p(\X), \quad
\G_t = \F^B_t \otimes \F^\mu_t \otimes \F^\mu_t, 
\end{align*}
and $Q$ is the joint distribution of the canonical processes $B$, $\mu^1$, and 
$\mu^2$ on $\Theta$. 
For each $i=1,2$, there is a  kernel
\[
\Omega_0 \times \P^p(\X) \ni \omega \mapsto K^i_\omega \in \P(\V \times \C^d).
\]
such that
\[
M\gamma^i = \rho^i(d\omega)K^i_\omega(dq,dx).
\]
The key point is that $K^i$ is necessarily adapted to 
the completed filtration $(\widetilde{\F}_{t}^{\xi,B,W,\mu})_{t \in [0,T]}$, which means that, 
for each $t \in [0,T]$ and each $\F^{\Lambda,X}_t$-measurable $\phi : \V \times \C^d \rightarrow \R$, the map  
$\omega \mapsto \int\phi\,dK^i_\omega$ is $\widetilde{\F}_t^{\xi,B,W,\mu}$-measurable. 
The proof is as follows. Since $M \gamma^i$ is a weak MFG solution, 
the $\sigma$-fields $\F_{T}^{\xi,B,W,\mu}$ and $\F_{t}^{\Lambda}$
are conditionally independent under $M \gamma^i$ given 
$\F_{t}^{\xi,B,W,\mu}$. 
Since the solution of the state equation 
\eqref{def:relaxedSDE-weak} is strong, 
$\F_{t}^{\xi,B,W,\mu,\Lambda,X}$ is included in the $M\gamma^i$-completion of
$\F_{t}^{\xi,B,W,\mu,\Lambda}$, from which we deduce that 
$\F_{T}^{\xi,B,W,\mu}$ and $\F_{t}^{\Lambda,X}$
are conditionally independent under $M \gamma^i$ given 
$\F_{t}^{\xi,B,W,\mu}$. Therefore, for each $t \in [0,T]$ and each $\F^{\Lambda,X}_t$-measurable $\phi : \V \times \C^d \rightarrow \R$, we have
\[
\int\phi\,dK^i = \E^{M\gamma^i}\left[\left.\phi(\Lambda,X) \right| \F^{\xi,B,W,\mu}_T\right] = \E^{M\gamma^i}\left[\left.\phi(\Lambda,X) \right| \F^{\xi,B,W,\mu}_t\right], \ a.s.
\]

\textit{Second step.} Define now the extended probability space:
\begin{equation*}
\overline{\Omega} := \Theta \times (\R^d \times \C^m) \times (\V \times \C^d)^2, \quad
\overline{\F}_t := \G_t \otimes \F^{\xi,W}_t \otimes \F^{\Lambda,X}_t \otimes \F^{\Lambda,X}_t, 
\end{equation*}
endowed with the probability measure:
\begin{equation*}
\overline{P} := Q(d\beta,d\nu^1,d\nu^2)\lambda(d\xi)\W^m(dw)\prod_{i=1}^2K^i_{\xi,\beta,w,\nu^i}(dq^i,dx^i).
\end{equation*}
Let $(B,\mu^1,\mu^2,\xi,W,\Lambda^1,X^1,\Lambda^2,X^2)$ denote the coordinate maps on $\overline{\Omega}$. Let $\mu^{i,x} = (\mu^i)^x$. In words, we have constructed $\overline{P}$ so that the following hold:
\begin{enumerate}
\item $(B,\mu^1,\mu^2)$, $W$, and $\xi$ are independent.
\item $(\Lambda^1,X^1)$ and $(\Lambda^2,X^2)$ are conditionally independent given $(B,\mu^1,\mu^2,\xi,W)$.
\item The state equation holds, for each $i=1,2$:
\[
X^i_t = \xi + \int_0^tds\int_A\Lambda^i_s(da)b(s,X^i_s,a)ds + \int_0^t\sigma(s,X^i_s)dW_s + \int_0^t\sigma_0(s,X^i_s)dB_s.
\]
\end{enumerate}
For $i,j = 1,2$, define
\begin{align*}
P^{i,j} &:= \overline{P} \circ (\xi,B,W,\mu^i,\Lambda^j,X^j)^{-1}.
\end{align*}
By assumption {\bf U}(4), $P^{i,i}$ is the \emph{unique} element of $\RC\A^\star(\rho^i)$, for each $i=1,2$. On the other hand, we will verify that
\begin{align}
\label{eq:P12:P21}
P^{1,2} \in \RC\A(\rho^1) \quad\quad \text{ and } \quad\quad P^{2,1} \in \RC\A(\rho^2).
\end{align}
Indeed, defining
\[
Q^{1,2} := P^{1,2} \circ (\xi,B,W,\mu,\Lambda)^{-1} = \overline{P} \circ (\xi,B,W,\mu^1,\Lambda^2)^{-1},
\]
it is clear that $P^{1,2} = \RC(Q^{1,2})$ because of the lack of mean field terms in the state equation (by assumption (U.1)). It remains only to check that $Q^{1,2}$ is compatible with $\rho^1$ in the sense of (2) in Subsection \ref{subse:setup}, or equivalently 
that, 
under $\overline{P}$, ${\mathcal F}_{T}^{\xi,B,W,\mu^1}$
and ${\mathcal F}_{t}^{\Lambda^2}$ are conditionally independent given 
${\mathcal F}_{t}^{\xi,B,W,\mu^1}$. Given three bounded real-valued functions 
$\phi_{t}^1$, $\phi_{T}^1$ and $\psi_{t}^2$, 
where $\phi_{t}^1$ and $\phi_{T}^1$ are both defined on 
$\Omega_{0} \times \P^p(\X)$ and are ${\mathcal F}_{t}^{\xi,B,W,\mu}$-measurable and ${\mathcal F}_{T}^{\xi,B,W,\mu}$-measurable (respectively), 
and where $\psi_{t}^2$ is defined on $\V$ and is ${\mathcal F}_{t}^{\Lambda}$-measurable,
we have
\begin{equation*}
\begin{split}
\E^{\overline{P}} \bigl[ \bigl( \phi_{t}^1 \phi_{T}^1 \bigr)(\xi,B,W,\mu^1) \psi_{t}^2(\Lambda^2) \bigr]
&= \E^{\overline{P}} \biggl[ \bigl( \phi_{t}^1 \phi_{T}^1 \bigr) (\xi,B,W,\mu^1) 
\int_{\V}
\psi_{t}^2(q) K_{\xi,B,W,\mu^2}^2(dq) \biggr]
\\
&= \E^{\overline{P}} \biggl[ \bigl( \phi_{t}^1 \phi_{T}^1 \bigr) (\xi,B,W,\mu^1)
\E^{\overline{P}} \biggl[ 
\int_{\V}
\psi_{t}^2(q) K_{\xi,B,W,\mu^2}^2(dq) \big\vert {\mathcal F}_T^{\xi,B,W}
\biggr]
\biggr],
\end{split}
\end{equation*}
where the last equality follows from the fact that 
$\mu^1$ and $\mu^2$ are conditionally independent given $(\xi,B,W)$. 
Since $(B,W)$ is an $({\mathcal F}_{t}^{\xi,B,W,\mu^2})_{t \in [0,T]}$-Wiener process and 
$\int_{\V}\psi_{t}^2(q) K_{\xi,B,W,\mu^2}^2(dq)$ is $\tilde{\mathcal F}_{t}^{\xi,B,W,\mu^2}$-measurable by the argument above, the conditioning in the third line can be replaced by a conditioning by 
${\mathcal F}_{t}^{\xi,B,W}$. Then, using once again the fact that 
$\mu^1$ and $\mu^2$ are conditionally independent given $(\xi,B,W)$, the conditioning
by ${\mathcal F}_{t}^{\xi,B,W}$ can be replaced by a conditioning by 
${\mathcal F}_{t}^{\xi,B,W,\mu^1}$, which proves the required property of conditional independence. This shows that $Q^{1,2} \in \A(\rho^1)$ and thus 
$P^{1,2} \in \RC\A(\rho^1)$. The proof that $P^{2,1} \in \RC\A(\rho^2)$ is identical.

\textit{Third step.}
Note that $(X^i,\Lambda^i,W)$ and $\mu^j$ are conditionally independent given $(B,\mu^i)$, for $i \neq j$, and thus
\begin{align}
\overline{P}((W,\Lambda^i,X^i) \in \cdot \ | \ B,\mu^1,\mu^2) = \overline{P}((W,\Lambda^i,X^i) \in \cdot \ | \ B,\mu^i) = \mu^i, \ i=1,2. \label{pf:uniqueness1}
\end{align}
Now suppose it does not hold that $\mu^1 = \mu^2$ a.s. Suppose that both
\begin{align}
&P^{1,1} = P^{1,2}, \quad \text{i.e} \quad
\overline{P} \circ (\xi,B,W,\mu^1,\Lambda^1,X^1)^{-1} = \overline{P} \circ (\xi,B,W,\mu^1,\Lambda^2,X^2)^{-1},
\label{pf:uniqueness2} 
\\
&P^{2,2} = P^{2,1}, \quad \text{i.e.} \quad 
\overline{P} \circ (\xi,B,W,\mu^2,\Lambda^2,X^2)^{-1} = \overline{P} \circ (\xi,B,W,\mu^2,\Lambda^1,X^1)^{-1}. 
\label{pf:uniqueness3} 
\end{align}
It follows that
\begin{align*}
\overline{P}((W,\Lambda^2,X^2) \in \cdot \ |  B,\mu^1) &= \overline{P}((W,\Lambda^1,X^1) \in \cdot \ | B,\mu^1) = \mu^1, \\
\overline{P}((W,\Lambda^1,X^1) \in \cdot \ |  B,\mu^2) &= \overline{P}((W,\Lambda^2,X^2) \in \cdot \ |  B,\mu^2) = \mu^2.
\end{align*}
Combined with \eqref{pf:uniqueness1}, this implies
\begin{align*}
\E^{\overline{P}}[\mu^2  |  B,\mu^1] &= \E^{\overline{P}}[\overline{P}((W,\Lambda^2,X^2) \in \cdot \ |  B,\mu^1,\mu^2) \ |  B,\mu^1] = \mu^1, \\
\E^{\overline{P}}[\mu^1  |  B,\mu^2] &= \E^{\overline{P}}[\overline{P}((W,\Lambda^1,X^1) \in \cdot \ |  B,\mu^1,\mu^2) \ |  B,\mu^2] = \mu^2.
\end{align*}
These conditional expectations are understood in terms of mean measures. By conditional independence, $\E^{\overline{P}}[\mu^i  |  B,\mu^j] = \E^{\overline{P}}[\mu^i  |  B]$ for $i \neq j$, and thus
\begin{equation*}
\E^{\overline{P}}[\mu^2  |  B] = \mu^1, \quad \text{and} \quad
\E^{\overline{P}}[\mu^1  |  B] = \mu^2.
\end{equation*}
Thus $\mu^1$ and $\mu^2$ are in fact $B$-measurable and equal, which is a contradiction. Hence, one of the distributional equalities \eqref{pf:uniqueness2} or \eqref{pf:uniqueness3} must fail. By optimality of $P^{1,1}$ and $P^{2,2}$ and by \eqref{eq:P12:P21}, we have the following two inequalities, and assumption (U.4) implies that at least one of them is strict:
\begin{align*}
0 \le J(P^{2,2}) - J(P^{2,1}), \quad \text{and} \quad
0 \le J(P^{1,1}) - J(P^{1,2}).
\end{align*}
Writing out the definition of $J$ and using the special form of $f$ from assumption (U.2),
\begin{align*}
0 \le \E^{\overline{P}}&\int_0^Tdt\left[\int_A\Lambda^2_t(da)f_1(t,X^2_t,a) + f_2(t,X^2_t,\mu^{2,x}_t) - \int_A\Lambda^1_t(da)f_1(t,X^1_t,a) - f_2(t,X^1_t,\mu^{2,x}_t) \right] \\
	& + \E^{\overline{P}}\left[g(X^2_T,\mu^{2,x}_T) - g(X^1_T,\mu^{2,x}_T)\right], \\
0 \le \E^{\overline{P}}&\int_0^Tdt\left[\int_A\Lambda^1_t(da)f_1(t,X^1_t,a) + f_2(t,X^1_t,\mu^{1,x}_t) - \int_A\Lambda^2_t(da)f_1(t,X^2_t,a) - f_2(t,X^2_t,\mu^{1,x}_t) \right] \\
	& + \E^{\overline{P}}\left[g(X^1_T,\mu^{1,x}_T) - g(X^2_T,\mu^{1,x}_T)\right],
\end{align*}
one of the two inequalities being strict.
Add these inequalities to get
\begin{align}
0 < \E^{\overline{P}}&\left[\int_0^T\left(f_2(t,X^2_t,\mu^{2,x}_t) - f_2(t,X^2_t,\mu^{1,x}_t) + f_2(t,X^1_t,\mu^{1,x}_t) - f_2(t,X^1_t,\mu^{2,x}_t)\right)dt\right] \nonumber \\
	&+\E^{\overline{P}}\left[g(X^2_T,\mu^{2,x}_T) - g(X^2_T,\mu^{1,x}_T) + g(X^1_T,\mu^{1,x}_T) - g(X^1_T,\mu^{2,x}_T) \right] \label{pf:uniqueness4}
\end{align}
Then, conditioning on $(B,\mu^1,\mu^2)$ inside of \eqref{pf:uniqueness4} and applying \eqref{pf:uniqueness1} yields
\begin{align*}
0 < \E^{\overline{P}}\int_{\C^d}(\mu^{2,x}-\mu^{1,x})(dx)\left[\int_0^T\left(f_2(t,x_t,\mu^{2,x}_t) - f_2(t,x_t,\mu^{1,x}_t)\right)dt + g(x_T,\mu^{2,x}_T) - g(x_T,\mu^{1,x}_T) \right].
\end{align*}
This contradicts assumption (U.3), and so $\mu^1 = \mu^2$ a.s.
\end{proof}

\appendix

\section{Topology of Wasserstein spaces} \label{ap:wasserstein}
Recall the definition of the Wasserstein metric from \eqref{def:wasserstein}. For ease of reference, this appendix compiles several known results on Wasserstein spaces.

\begin{proposition}[Theorem 7.12 of \cite{villanibook}] \label{pr:wasserstein}
Let $(E,\ell)$ be a metric space, and suppose $\mu,\mu_n \in \P^p(E)$. Then the following are equivalent
\begin{enumerate}
\item $\ell_{E,p}(\mu_n,\mu) \rightarrow 0$.
\item $\mu_n \rightarrow \mu$ weakly and for some $x_0 \in E$ we have
\begin{align*}
\lim_{r \rightarrow \infty}\sup_n\int_{\{x : \ell^p(x,x_0) \ge r\}}\mu_n(dx)\ell^p(x,x_0) = 0.
\end{align*}
\item $\int\phi\,d\mu_n \rightarrow \int\phi\,d\mu$ for all continuous functions $\phi : E \rightarrow \R$ such that there exists $x_0 \in E$ and $c > 0$ for which $|\phi(x)| \le c(1 + \ell^p(x,x_0))$ for all $x \in E$.
\end{enumerate}
In particular, (2) implies that a sequence $\{\mu_n\} \subset \P^p(E)$ is relatively compact if and only it is tight (i.e. relatively compact in $\P(E)$) and satisfies the uniform integrability condition (2).
\end{proposition}

The rest of the results listed here are borrowed from Appendices A and B of \cite{lacker-mfgcontrolledmartingaleproblems}, but the proofs are straightforward and essentially just extend known results on weak convergence using a homeomorphism between $\P(E)$ and $\P^p(E)$. Indeed, if  $x_0 \in E$ is fixed and $\psi(x) := 1 + \ell^p(x,x_0)$, then the map $\mu \mapsto \psi\,\mu/\int\psi\,d\mu$ is easily seen to define a homeomorphism from $(\P^p(E),\ell_{E,p})$ to $\P(E)$ with the weak topology, where for each $\mu \in \P^p(E)$ the measure $\psi\,\mu \in \P(E)$ is defined by by $\psi\,\mu(C) = \int_BC\psi\,d\mu$ for $C \in \B(E)$. For $P \in \P(\P(E))$, define the mean measure $mP \in \P(E)$ by
\[
mP(C) := \int_{\P(E)}\mu(C)P(d\mu), \ C \in \B(E).
\]

\begin{proposition} \label{pr:ptight}
Let $(E,\ell)$ be a complete separable metric space. Suppose $K \subset \P^p(\P^p(E))$ is such that $\{mP : P \in K\} \subset \P(E)$ is tight and
\[
\sup_{P \in K}\int_E mP(dx)\ell^{p'}(x,x_0) < \infty, \text{ for some } p' > p, \ x_0 \in E.
\]
Then $K$ is relatively compact.
\end{proposition}

In the next two lemmas, let $(E,\ell_E)$ and $(F,\ell_F)$ be two complete separable metric spaces. We equip $E \times F$ with the metric formed by adding the metrics of $E$ and $F$, given by $((x_1,x_2),(y_1,y_2)) \mapsto \ell_E(x_1,y_1) + \ell_F(x_2,y_2)$, although this choice is inconsequential.

\begin{lemma} \label{le:productrelcompactness}
A set $K \subset \P^p(E \times F)$ is relatively compact if and only if $\{P(\cdot \times F) : P \in K\} \subset \P^p(E)$ and $\{P(E \times \cdot) : P \in K\} \subset \P^p(F)$ are relatively compact.
\end{lemma}

\begin{lemma} \label{le:componentwise}
Let $\phi : E \times F \rightarrow \R$ satisfy the following:
\begin{enumerate}
\item $\phi(\cdot,y)$ is measurable for each $y \in F$.
\item $\phi(x,\cdot)$ is continuous for each $x \in E$.
\item There exist $c > 0$, $x_0 \in E$, and $y_0 \in F$ such that
\[
|\phi(x,y)| \le c(1 + \ell^p_1(x,x_0) + \ell_2^p(y,y_0)), \ \forall(x,y) \in E \times F.
\]
\end{enumerate}
If $P^n \rightarrow P$ in $\P^p(E \times F)$ and $P^n(\cdot \times F) = P(\cdot \times F)$ for all $n$, then $\int\phi\,dP^n \rightarrow \int\phi\,dP$. 
\end{lemma}

The last result we state specialize the above to the space $\V$, defined in Section \ref{subse:relaxed:controls}.

\begin{lemma} \label{le:usc}
Let $(E,\ell)$ be a complete separable metric space. Let $\phi : [0,T] \times E \times A \rightarrow \R$ be measurable with $\phi(t,\cdot)$ jointly continuous for each $t \in [0,T]$. Suppose there exist $c > 0$ and $x_0 \in E$ such that
\[
\phi(t,x,a) \le c(1 + \ell^p(x,x_0) + |a|^p).
\]
Then the following map is upper semicontinuous:
\[
C([0,T];E) \times \V \ni (x,q) \mapsto \int q(dt,da)\phi(t,x_t,a).
\]
If also $|\phi(t,x,a)| \le c(1 + \ell^p(x,x_0) + |a|^p)$, then this map is continuous.
\end{lemma}

\section{A compactness result for It\^{o} processes} \label{ap:itocompact}
Recall from assumption \ref{assumption:A} that $A$ is a closed subset of a Euclidean space, $p' > p \ge 1 \vee p_\sigma$, $p_\sigma \in [0,2]$, and $\lambda \in \P^{p'}(\R^d)$. Recall that $\V$ was defined in Section \ref{subse:relaxed:controls}.

\begin{proposition} \label{pr:itocompact}
Let $d$ be a positive integer, and fix $c > 0$. Let $\Q \subset \P(\V \times \C^d)$ be the set of laws of $\V \times \C^d$-valued random variables $(\Lambda,X)$ defined on some complete filtered probability space $(\Theta,(\G_t)_{t \in [0,T]},P)$ satisfying:
\begin{enumerate}
\item $dX_t = \int_AB(t,X_t,a)\Lambda_t(da)dt + \Sigma(t,X_t)dW_t$.
\item $W$ is a $k$-dimensional $(\G_t)_{t \in [0,T]}$-Wiener process.
\item $\Sigma : [0,T] \times \Theta \times \R^d \rightarrow \R^{d \times k}$ and $B : [0,T] \times \Theta \times \R^d \times A \rightarrow \R^d$ are jointly measurable, using the $(\G_t)_{t \in [0,T]}$-progressive $\sigma$-field on $[0,T] \times \Theta$.
\item $X_0$ has law $\lambda$ and is $\G_0$-measurable.
\item There exists a nonnegative $\G_T$-measurable random variable $Z$ such that, for each $(t,x,a) \in [0,T] \times \R^d \times A$,
\begin{align*}
|B(t,x,a)|&\le c\left(1 + |x| + Z + |a|\right), \quad |\Sigma(t,x)|^2 \le c\left(1 + |x|^{p_\sigma} + Z^{p_\sigma}\right)
\end{align*}
and
\[
\E^P\left[|X_0|^{p'} + Z^{p'} + \int_0^T\int_A|a|^{p'}\Lambda_t(da)dt\right] \le c.
\]
\end{enumerate}
(That is, we vary $\Sigma$, $B$, and the probability space of definition.) Then $\Q$ is a relatively compact subset of $\P^p(\V \times \C^d)$.
\end{proposition}
\begin{proof}
For each $P \in \Q$ with corresponding probability space $(\Theta,(\G_t)_{t \in [0,T]},P)$ and coefficients $B$, $\Sigma$, standard estimates as in Lemma \ref{le:stateestimate} yield
\begin{align*}
\E^P\|X\|^{p'}_T \le C\E^P\left[1 + |X_0|^{p'} + Z^{p'} + \int_0^T\int_A|a|^{p'}\Lambda_t(da)dt\right].
\end{align*}
where $C > 0$ does not depend on $P$. Hence assumption (6) implies
\begin{align}
\sup_{P \in \Q}\E^P\|X\|^{p'}_T \le C(1 + c) < \infty. \label{pf:itocompact1}
\end{align}
Suppose we can show that $\Q_X := \{P \circ X^{-1} : P \in \Q\} \subset \P(\C^d)$ is tight. Then, from \eqref{pf:itocompact1} (and Proposition \ref{pr:wasserstein}) that $\Q_X$ is relatively compact in $\P^p(\C^d)$. Moreover,
\[
\sup_{P \in \Q}\E^P\int_0^T\int_A|a|^{p'}\Lambda_t(da)dt < \infty
\]
implies that $\{P \circ \Lambda^{-1} : P \in \Q\}$ is relatively compact in $\P^p(\V)$, by Proposition \ref{pr:ptight}. Hence, $\Q$ is relatively compact in $\P^p(\V \times \C^d)$, by Lemma \ref{le:productrelcompactness}. It remains to check that $\Q_X$ is tight, which we will check by verifying Aldous' criterion (see \cite[Lemma 16.12]{kallenberg-foundations}) for tightness, or
\begin{align}
\lim_{\delta\downarrow 0}\sup_{P \in \Q}\sup_\tau \E^P\left[|X_{(\tau + \delta) \wedge T} - X_\tau|^p\right] = 0, \label{pf:itocompact2}
\end{align}
where the supremum is over stopping times $\tau$ valued in $[0,T]$. The Burkholder-Davis-Gundy inequality implies that there exists a constant $C' > 0$ (which does not depend on $P$ but may change from line to line) such that
\begin{align*}
\E^P\left[|X_{(\tau + \delta) \wedge T} - X_\tau|^p\right] \le & \ C'\E^P\left[\left|\int_\tau^{(\tau + \delta) \wedge T}dt\int_A\Lambda_t(da)B(t,X_t,a)\right|^p\right] \\
	&+ C'\E^P\left[\left(\int_\tau^{(\tau + \delta) \wedge T}|\Sigma(t,X_t)|^2dt\right)^{p/2} \right] \\
	\le & \ C'\E^P\left[\left|\int_\tau^{(\tau + \delta) \wedge T}dt\int_A\Lambda_t(da)c(1 + \|X\|_T + Z + |a|)\right|^p\right] \\
	&+ C'\E^P\left[\left(\int_\tau^{(\tau + \delta) \wedge T}c(1 + \|X\|_T^{p_\sigma} + Z^{p_\sigma})dt\right)^{p/2} \right] \\\
	\le & \ C'\E^P\left[(\delta^p + \delta^{p/2}) (1+\|X\|_T^p + Z^p) + \int_\tau^{(\tau + \delta) \wedge T}\int_A|a|^p\Lambda_t(da)dt \right]
\end{align*}
The last line simply used Jensen's inequality with $p \ge 1$, and we used also the fact that $p_\sigma \le 2$. Since
\[
\sup_{P \in \Q}\E^P\left[\|X\|_T^p + Z^p\right] < \infty,
\]
it follows that
\[
\lim_{\delta\downarrow 0}\sup_{P \in \Q}\sup_\tau \E^P\left[(\delta^p + \delta^{p/2}) (1+\|X\|_T^p + Z^p)\right] = 0.
\]
By assumption,
\[
\sup_{P \in \Q}\E^P\int_0^T\int_A|a|^{p'}\Lambda_t(da)dt \le c < \infty,
\]
and since $p < p'$ it follows that
\[
\lim_{\delta \downarrow 0}\sup_{P \in \Q}\sup_\tau\E^P\int_\tau^{(\tau + \delta) \wedge T}\int_A|a|^p\Lambda_t(da)dt = 0.
\]
Putting this together proves \eqref{pf:itocompact2}.
\end{proof}

\section{Density of adapted controls} \label{ap:density}
The goal of this section is to prove Lemma \ref{le:adapteddensity}, which is essentially an \emph{adapted} analog of the following version of a classical result. 
\begin{proposition} \label{pr:density}
Suppose $E$ and $F$ are complete separable metric spaces and $\mu \in \P(E)$. 
If $\mu$ is nonatomic, then the set 
\[
\left\{\mu(dx)\delta_{\phi(x)}(dy) \in \P(E \times F) : \phi : E \rightarrow F \text{ is measurable} \right\}
\]
is dense in $\P(E,\mu;F) := \{P \in \P(E \times F) : P(\cdot \times F) = \mu \}$. If additionally $F$ is (homeomorphic to) a convex subset of a locally convex space $H$, then the set
\[
\left\{\mu(dx)\delta_{\phi(x)}(dy) \in \P(E \times F) : \phi : E \rightarrow F \text{ is continuous} \right\}
\]
is also dense in $\P(E,\mu;F)$.
\end{proposition}
\begin{proof}
This first claim is well known, and can be found for example in \cite[Theorem 2.2.3]{youngmeasuresontopologicalspaces}. 
To prove the second claim from the first, it suffices to show that any measurable function $\phi : E \rightarrow F$ can be obtained as the $\mu$-a.s. limit of \emph{continuous} functions. By Lusin's theorem \cite[Theorem 7.1.13]{bogachev-measuretheory}, for each $\epsilon>0$ we may find a compact $K_\epsilon \subset E$ such that $\mu(K_\epsilon^c) \le \epsilon$ and the restriction $\phi|_{K_\epsilon} : K_\epsilon \rightarrow F$ is continuous. Using a generalization of the Tietze extension theorem due to Dugundji \cite[Theorem 4.1]{dugundji1951extension}, we may find a continuous function $\tilde{\phi}_\epsilon : E \rightarrow H$ such that $\tilde{\phi}_\epsilon = \phi$ on $K_\epsilon$ and such that the range $\tilde{\phi}_\epsilon(E)$ is contained in the convex hull of $\phi|_{K_\epsilon}(E)$, which is itself contained in the convex set $F$. We may thus view $\tilde{\phi}_\epsilon$ as a continuous function from $E$ to $F$. Since $\mu(\tilde{\phi}_\epsilon \neq \phi) \le \mu(K_\epsilon^c) \le \epsilon$, we may find a subsequence of $\tilde{\phi}_\epsilon$ which converges $\mu$-a.s. to $\phi$.
\end{proof}

As in  Lemma \ref{le:adapteddensity}, we work under assumption \ref{assumption:B}. Recall the definition of an adapted function, given in Definition \ref{def:adapted}.

\begin{proof}[Proof of Lemma \ref{le:adapteddensity}]
It is clear from the definition of an adapted function that $\A_a(\rho) \subset \A(\rho)$. Let $S = (\xi,B,W,\mu)$ abbreviate the identity map on $\Omega_0 \times \P^p(\X)$, and let and
\begin{align*}
S^t &:= (\xi,B_{\cdot \wedge t},W_{\cdot\wedge t},\mu^t), \text{ where} \\
\mu^t &:= \mu \circ (W_{\cdot \wedge t},1_{[0,t]}\Lambda,X_{\cdot \wedge t})^{-1}.
\end{align*}
On $\Omega' := \Omega_0 \times \P^p(\X) \times \V$, define the filtrations $(\F^S_t)_{t \in [0,T]}$ and $(\F^{S,\Lambda}_t)_{t \in [0,T]}$ by $\F^S_t := \sigma(S^t)$ and $\F^{S,\Lambda}_t := \sigma(S^t,1_{[0,t]}\Lambda)$. Equivalently, our notational conventions allow us to write $\F^S_t = \F^{\xi,B,W,\mu}_t$ and $\F^{S,\Lambda}_t = \F^{\xi,B,W,\mu,\Lambda}_t$.

Fix $Q \in \A(\rho) \subset \P^p(\Omega')$. It is clear that we may approximate elements of $\V$ (in the topology of $\V$) by piece-wise constant $\P(A)$-valued paths; that is, we may find a sequence of piece-wise constant $(\F^{S,\Lambda}_t)_{t \in [0,T]}$-adapted $\P(A)$-valued processes $(\alpha^k(t))_{t \in [0,T]}$ on $\Omega'$ such that $dt\alpha^k(t)(da) \rightarrow \Lambda$, $Q$-a.s., and a fortiori $Q \circ (S,dt\delta_{\alpha^k(t)}(da))^{-1} \rightarrow Q$ weakly. Since $A$ is compact and the $S$-marginal is fixed, this convergence happens also in $\P^p(\Omega')$, and thus we need not bother to distinguish $\P^p(\Omega')$-convergence from weak convergence in what follows. Here, a \emph{piece-wise constant} $\F^{S,\Lambda}_t$-adapted $\P(A)$-valued process $(\alpha(t))_{t \in [0,T]}$ is of the form
\[
\alpha(t) = a_01_{[0,t_0]}(t) + \sum_{i=1}^na_i1_{(t_i,t_{i+1}]}(t),
\]
where $a_0 \in \P(A)$ is deterministic, $a_i$ is an $\F^{S,\Lambda}_{t_i}$-measurable $\P(A)$-valued random variable, and $0 < t_0 < t_1 < \ldots < t_{n+1} = T$ for some $n$. 
It remains to show that, for any piece-wise constant $(\F^{S,\Lambda}_t)_{t \in [0,T]}$-adapted $\P(A)$-valued process $(\alpha(t))_{t \in [0,T]}$, there exists a sequence $(\alpha^k(t))_{t \in [0,T]}$ of $(\F^S_t)_{t \in [0,T]}$-adapted $\P(A)$-valued processes such that $Q \circ (S,dt\delta_{\alpha^k(t)}(da))^{-1} \rightarrow Q \circ (S,dt\delta_{\alpha(t)}(da))^{-1}$ weakly. The proof is an inductive application of Proposition \ref{pr:density}, the second part of which applies because of the convexity of $\P(A)$.

By the second part of Proposition \ref{pr:density}, there exists a sequence of continuous $\F^S_{t_1}$-measurable functions $a^j_1 : \Omega_0 \times \P^p(\X) \rightarrow \P(A)$ such that $Q \circ (S^{t_1},a^j_1(S))^{-1} \rightarrow Q \circ (S^{t_1},a_1)^{-1}$. Since $Q \in \A(\rho)$, $\F^{S,\Lambda}_t$ and $\F^S_T$ are conditionally independent given $\F^S_t$. In particular, $S$ and $(S^{t_1},a_1)$ are conditionally independent given $S^{t_1}$, and so are $S$ and $(S^{t_1},a^j_1(S))$. Now let $\phi : \Omega_0 \times \P^p(\X) \rightarrow \R$ be bounded and measurable, and let $\psi : \P(A) \rightarrow \R$ be continuous. Letting $\E$ denote expectation under $Q$, Lemma \ref{le:componentwise} implies
\begin{align*}
\lim_{j\rightarrow\infty}\E[\phi(S)\psi(a^j_1(S))] &= \lim_{j\rightarrow\infty}\E\left[\E\left[\left.\phi(S)\right| S^{t_1}\right]\psi(a^j_1(S))\right]	\\
	&= \E\left[\E\left[\left.\phi(S)\right| S^{t_1}\right]\psi(a_1)\right] \\
	&= \E\left[\E\left[\left.\phi(S)\right| S^{t_1}\right] \E\left[\left.\psi(a_1)\right| S^{t_1}\right]\right] \\
	&= \E\left[\E\left[\left.\phi(S)\psi(a_1)\right| S^{t_1}\right]\right] \\
	&= \E\left[\phi(S)\psi(a_1)\right]
\end{align*}
This is enough to show that $Q \circ (S,a^j_1(S))^{-1} \rightarrow Q \circ (S,a_1)^{-1}$ (see e.g. \cite[Proposition 3.4.6(b)]{ethierkurtz}).

We proceed inductively as follows: suppose we are given $a^j_1,\ldots,a^j_i : \Omega_0 \times \P^p(\X) \rightarrow \P(A)$ for some $i \in \{1,\ldots,n-1\}$, where $a^j_k$ is $\F^S_{t_k}$-measurable for each $k=1,\ldots,i$, and
\[
\lim_{j\rightarrow\infty}Q \circ (S,a^j_1(S),\ldots,a^j_i(S))^{-1} = Q \circ (S,a_1,\ldots,a_i)^{-1}.
\]
By Proposition \ref{pr:density}, there exists a sequence of continuous $\F^S_{t_{i+1}} \otimes \B(\P(A)^i)$-measurable functions $\hat{a}^k : (\Omega_0 \times \P^p(\X)) \times \P(A)^i \rightarrow \P(A)$ such that
\[
\lim_{k\rightarrow\infty}Q \circ (S^{t_{i+1}},a_1,\ldots,a_i,\hat{a}^k(S,a_1,\ldots,a_i))^{-1} = Q \circ (S^{t_{i+1}},a_1,\ldots,a_i,a_{i+1})^{-1}.
\]
It follows as above that in fact
\[
\lim_{k\rightarrow\infty}Q \circ (S,a_1,\ldots,a_i,\hat{a}^k(S,a_1,\ldots,a_i))^{-1} = Q \circ (S,a_1,\ldots,a_i,a_{i+1})^{-1}.
\]
By continuity of $\hat{a}^k$, it holds for each $k$ that
\[
\lim_{j\rightarrow\infty}Q \circ (S,a^j_1(S),\ldots,a^j_i(S),\hat{a}^k(S,a^j_1(S),\ldots,a^j_i(S)))^{-1} = Q \circ (S,a_1,\ldots,a_i,\hat{a}^k(S,a_1,\ldots,a_i))^{-1}
\]
These above two limits imply that there exists a subsequence $j_k$ such that 
\[
\lim_{k\rightarrow\infty}Q \circ (S,a^{j_k}_1(S),\ldots,a^{j_k}_i(S),\hat{a}^k(S,a^{j_k}_1(S),\ldots,a^{j_k}_i(S)))^{-1} = Q \circ (S,a_1,\ldots,a_i,a_{i+1})^{-1}
\]
Define $a^k_{i+1}(S) := \hat{a}^k(S,a^{j_k}_1(S),\ldots,a^{j_k}_i(S))$ 
to complete the induction.

By the above argument, we construct $n$ sequences $a^k_i : \Omega_0 \times \P^p(\X) \rightarrow \P(A)$, for $i=1,\ldots,n$, where $a^k_i$ is continuous and $\F^S_{t_i}$-measurable, and
\[
\lim_{k\rightarrow\infty}Q \circ (S,a^k_1(S),\ldots,a^k_n(S))^{-1} = Q \circ (S,a_1,\ldots,a_n)^{-1}.
\]
Define 
\[
\alpha^k(t) = a_01_{[0,t_0]}(t) + \sum_{i=1}^na^k_i(S)1_{(t_i,t_{i+1}]}(t).
\]
The map
\[
\P(A)^n \ni (\alpha_1,\ldots,\alpha_n) \mapsto dt\left[a_0(da)1_{[0,t_0]}(t) + \sum_{i=1}^n\alpha_i(da)1_{(t_i,t_{i+1}]}(t)\right] \in \V
\]
is easily seen to be continuous, and thus $Q \circ (S,dt\alpha^k(t)(da))^{-1} \rightarrow Q \circ (S,dt\alpha(t)(da))^{-1}$, completing the proof.
\end{proof}

\bibliographystyle{amsplain}
\bibliography{MFGcommonnoise-bib}

\providecommand{\bysame}{\leavevmode\hbox to3em{\hrulefill}\thinspace}
\providecommand{\MR}{\relax\ifhmode\unskip\space\fi MR }
\providecommand{\MRhref}[2]{%
  \href{http://www.ams.org/mathscinet-getitem?mr=#1}{#2}
}
\providecommand{\href}[2]{#2}
\begin{thebibliography}{10}

\bibitem{ahuja-mfgwellposedness}
S.~Ahuja, \emph{Wellposedness of mean field games with common noise under a
  weak monotonicity condition}, arXiv preprint arXiv:1406.7028 (2014).

\bibitem{aliprantisborder}
C.~Aliprantis and K.~Border, \emph{Infinite dimensional analysis: {A}
  hitchhiker's guide}, 3 ed., Springer, 2007.

\bibitem{barlow-sdenostrongsolution}
M.T. Barlow, \emph{One dimensional stochastic differential equations with no
  strong solution}, Journal of the London Mathematical Society \textbf{2}
  (1982), no.~2, 335--347.

\bibitem{bensoussan-mfgbook}
A.~Bensoussan, J.~Frehse, and P.~Yam, \emph{Mean field games and mean field
  type control theory}, Springer, 2013.

\bibitem{bensoussan-masterequation}
\bysame, \emph{The master equation in mean field theory}, arXiv preprint
  arXiv:1404.4150 (2014).

\bibitem{bogachev-measuretheory}
V.~Bogachev, \emph{Measure theory}, vol.~2, Springer, 2007.

\bibitem{bremaudyor-changesoffiltrations}
P.~Br{\'e}maud and M.~Yor, \emph{Changes of filtrations and of probability
  measures}, Zeitschrift f{\"u}r Wahrscheinlichkeitstheorie und Verwandte
  Gebiete \textbf{45} (1978), no.~4, 269--295.

\bibitem{cardaliaguet-mfgnotes}
P.~Cardaliaguet, \emph{Notes on mean field games}.

\bibitem{carmonadelarue-mfg}
R.~Carmona and F.~Delarue, \emph{Probabilistic analysis of mean field games},
  SIAM Journal on Control and Optimization \textbf{51} (2013), 2705--2734.

\bibitem{carmona:delarue:4lyons}
\bysame, \emph{The master equation for large population equilibriums}, arXiv
  preprint arXiv:1404.4694 (2014).

\bibitem{carmonadelaruelachapelle-mkvvsmfg}
R.~Carmona, F.~Delarue, and A.~Lachapelle, \emph{Control of
  {M}c{K}ean--{V}lasov dynamics versus mean field games}, Mathematics and
  Financial Economics \textbf{7} (2013), no.~2, 131--166.

\bibitem{carmonafouque-systemicrisk}
R.~Carmona, J.P. Fouque, and L.H. Sun, \emph{Mean field games and systemic
  risk}, arXiv preprint arXiv:1308.2172 (2013).

\bibitem{carmonalacker-probabilisticweakformulation}
R.~Carmona and D.~Lacker, \emph{A probabilistic weak formulation of mean field
  games and applications}, Annals of Applied Probability \textbf{25} (2015),
  no.~3, 1189--1231.

\bibitem{youngmeasuresontopologicalspaces}
C.~Castaing, P.R.~De Fitte, and M.~Valadier, \emph{Young measures on
  topological spaces: with applications in control theory and probability
  theory}, vol. 571, Springer, 2004.

\bibitem{dufourstockbridge-existence}
F.~Dufour and R.H. Stockbridge, \emph{On the existence of strict optimal
  controls for constrained, controlled {M}arkov processes in continuous time},
  Stochastics An International Journal of Probability and Stochastic Processes
  \textbf{84} (2012), no.~1, 55--78.

\bibitem{dugundji1951extension}
J.~Dugundji, \emph{An extension of {T}ietze's theorem}, Pacific J. Math
  \textbf{1} (1951), no.~3, 353--367.

\bibitem{ethierkurtz}
S.N. Ethier and T.G. Kurtz, \emph{Markov processes: characterization and
  convergence}, 2 ed., vol. 282, Wiley-Interscience, 2005.

\bibitem{filippov-convexity}
A.F. Filippov, \emph{On certain questions in the theory of optimal control},
  Journal of the Society for Industrial \& Applied Mathematics, Series A:
  Control \textbf{1} (1962), no.~1, 76--84.

\bibitem{gomessaude-mfgsurvey}
D.A. Gomes and J.~Sa{\'u}de, \emph{Mean field games models-a brief survey},
  Dynamic Games and Applications (2013), 1--45.

\bibitem{gueantlasrylionsmfg}
O.~Gu\'{e}ant, J.M. Lasry, and P.L. Lions, \emph{Mean field games and
  applications}, Paris-Princeton Lectures on Mathematical Finance 2010, Lecture
  Notes in Mathematics, vol. 2003, Springer Berlin / Heidelberg, 2011,
  pp.~205--266.

\bibitem{haussmannlepeltier-existence}
U.G. Haussmann and J.P. Lepeltier, \emph{On the existence of optimal controls},
  SIAM Journal on Control and Optimization \textbf{28} (1990), no.~4, 851--902.

\bibitem{huangmfg1}
M.~Huang, R.~Malham\'{e}, and P.~Caines, \emph{Large population stochastic
  dynamic games: closed-loop {M}c{K}ean-{V}lasov systems and the {N}ash
  certainty equivalence principle}, Communications in Information and Systems
  \textbf{6} (2006), no.~3, 221--252.

\bibitem{jacodmemin-stable}
J.~Jacod and J.~M{\'e}min, \emph{Sur un type de convergence interm{\'e}diaire
  entre la convergence en loi et la convergence en probabilit{\'e}},
  S{\'e}minaire de probabilit{\'e}s de Strasbourg \textbf{15} (1981), 529--546.

\bibitem{jacodmemin-weaksolution}
\bysame, \emph{Weak and strong solutions of stochastic differential equations:
  existence and stability}, Stochastic integrals, Springer, 1981, pp.~169--212.

\bibitem{kallenberg-foundations}
O.~Kallenberg, \emph{Foundations of modern probability}, Springer, 2002.

\bibitem{elkaroui-compactification}
N.~El Karoui, D.H. Nguyen, and M.~Jeanblanc-Picqu{\'e}, \emph{Compactification
  methods in the control of degenerate diffusions: existence of an optimal
  control}, Stochastics \textbf{20} (1987), no.~3, 169--219.

\bibitem{kurtz-conditionalmtgproblem}
T.G. Kurtz, \emph{Martingale problems for conditional distributions of {M}arkov
  processes.}, Electronic J. Probab. \textbf{3} (1998), no.~9, 1--29.

\bibitem{kurtz-yw2013}
\bysame, \emph{Weak and strong solutions of general stochastic models}, arXiv
  preprint arXiv:1305.6747 (2013).

\bibitem{kurtzprotter-weakconvergence}
T.G. Kurtz and P.~Protter, \emph{Weak limit theorems for stochastic integrals
  and stochastic differential equations}, The Annals of Probability (1991),
  1035--1070.

\bibitem{lacker-mfgcontrolledmartingaleproblems}
D.~Lacker, \emph{Mean field games via controlled martingale problems:
  {E}xistence of {M}arkovian equilibria}, Stochastic Processes and Their
  Applications \textbf{125} (2015), no.~7, 2856--2894.

\bibitem{lasrylionsmfg}
J.M. Lasry and P.L. Lions, \emph{Mean field games}, Japanese Journal of
  Mathematics \textbf{2} (2007), 229--260.

\bibitem{gueantlasrylionsmfg-growth}
J.M. Lasry, P.L. Lions, and O.~Gu{\'e}ant, \emph{Application of mean field
  games to growth theory},  (2008).

\bibitem{stroockvaradhanbook}
D.W. Stroock and S.R.S. Varadhan, \emph{Multidimensional diffusion processes},
  2 ed., Grundlehren Der Mathematischen Wissenschaften, Springer, 1979.

\bibitem{villanibook}
C.~Villani, \emph{Topics in optimal transportation}, Graduate Studies in
  Mathematics, American Mathematical Society, 2003.

\bibitem{yamadawatanabe-uniqueness}
T.~Yamada and S.~Watanabe, \emph{On the uniqueness of solutions of stochastic
  differential equations}, J. Math. Kyoto Univ \textbf{11} (1971), no.~1,
  155--167.

\end{thebibliography}

\end{document}